\documentclass[notitlepage]{amsart}

\usepackage{amsfonts}
\usepackage[latin1]{inputenc}
\usepackage[all]{xy} 
\usepackage{latexsym,amssymb,amscd}
\usepackage{amsmath}
\usepackage{mathrsfs}
\usepackage{multicol}

\usepackage{pb-diagram, pb-xy}

\usepackage{graphicx}
\usepackage{fancyhdr}
\usepackage{color}
\usepackage[T1]{fontenc}

\usepackage{tikz}
\usepackage{tikz-cd}
\usetikzlibrary{arrows}
 \usetikzlibrary{snakes}
 \usetikzlibrary{shapes}

\newtheorem{pro}{Proposition}[section]
\newtheorem{teo}[pro]{Theorem}
\newtheorem{lem}[pro]{Lemma}
\newtheorem{defi}[pro]{Definition}
\newtheorem{cor}[pro]{Corollary}
\newtheorem{rk}[pro]{Remark}
\newtheorem{ex}[pro]{Example}

\newcommand{\Ab}{\mathrm{Ab}}

\newcommand{\modu}{\mathrm{mod}}

\newcommand{\imagen}{\mathrm{Im}}
\newcommand{\Ex}{\mathrm{Ext}}
\newcommand{\Tr}{\mathrm{Tr}}

\newcommand{\Mod}{\mathrm{Mod}}
\newcommand{\Nat}{\mathrm{Nat}}

\newcommand{\ind}{\mathrm{ind}}
\newcommand{\rad}{\mathrm{rad}}
\newcommand{\Supp}{\mathrm{Supp}}

\newcommand{\add}{\mathrm{add}}

\newcommand{\proj}{\mathrm{proj}}
\newcommand{\Proj}{\mathrm{Proj}}
\newcommand{\F}{\mathcal{F}}

\newcommand{\K}{\mathbb{K}}
\newcommand{\X}{\mathcal{X}}

\newcommand{\R}{\mathfrak{R}}

\newcommand{\Hom}{\mathrm{Hom}}
\newcommand{\End}{\mathrm{End}}
\newcommand{\finp}{\mathrm{fin.p}}
\newcommand{\Ext}{\mathrm{Ext}}

\newcommand{\A}{\mathcal{A}}
\newcommand{\B}{\mathcal{B}}
\newcommand{\C}{\mathcal{C}}

\newcommand{\AF}{\mathsf{AF}}
\newcommand{\ltt}{\mathcal}
\newcommand{\fl}{\longrightarrow}
\newcommand{\sta}{\stackrel}

\newenvironment{dem}{\noindent\bf Proof. \rm }{$\ \Box$}
\usepackage{latexsym,amssymb,amscd}
\usepackage{amsmath}

\begin{document}
\title[Standardly stratified ringoids ]{A generalization of the theory of standardly stratified algebras I: Standardly stratified ringoids}
\author{O. Mendoza, M. Ort\'{\i}z, C.S\'aenz, V. Santiago}
\thanks{2010 {\it{Mathematics Subject Classification}}. Primary 18E10. Secondary 16G99.\\
The authors thank the Projects PAPIIT-Universidad Nacional
Aut\'onoma de M\'exico IA105317 and  IN100520.}
\maketitle

\date{}
\begin{abstract}
We extend the classical notion of standardly stratified $k$-algebra (stated for finite dimensional $k$-algebras) to the more general class 
of rings, possibly without $1,$  with enough idempotents. We show that many of the fundamental results, which are known for classical standardly stratified algebras, can  be generalized to this context. Furthermore, new classes of rings appear as: ideally standardly stratified and ideally quasi-hereditary. In the classical theory, it is known that quasi-hereditary and ideally quasi-hereditary algebras are equivalent notions, but in our general setting this is no longer true. 
To develop the theory, we use the well known connection between rings with enough idempotents and skeletally small categories (ringoids or rings with several objects). 
\end{abstract}

\setcounter{tocdepth}{1}
\tableofcontents

\section{Introduction}

The notions of quasi-hereditary algebra and highest weight category were introduced and studied by  E. Cline, B. Parshall and L. L. Scott 
\cite{CPS1, CPS0, Scott}. Highest weight categories are a very special kind of abelian categories that  arise in the representation theory of Lie algebras 
and algebraic groups. The  Highest weight categories with a finite number of simple objects are precisely the module categories of 
quasi-hereditary algebras. It is worth mentioning that quasi-hereditary algebras were originally defined through a  special chain of ideals.

For the setting of finite dimensional algebras, quasi-hereditary algebras were amply studied, among others,  by V. Dlab and M. Ringel in 
\cite{Dlab0, DR2, DR1, Rin}. Dlab and Ringel introduced the set of standard modules ${}_\Lambda\Delta$  associated to a finite dimensional  algebra 
$\Lambda.$ Later on, M. Ringel established a  relationship between  quasi-hereditary algebras and tilting theory \cite{Rin}, which has been very fruitful for 
the study of quasi-hereditary algebras. 
For doing so, M. Ringel studied the homological properties of the category
$\mathcal{F}({}_\Lambda\Delta)$ of ${}_\Lambda\Delta$-filtered $\Lambda$-modules and constructed the
characteristic module ${}_\Lambda T$ (which turned out to be tilting) associated to $\mathcal{F}(_\Lambda\Delta).$ Moreover,
M. Ringel proved that the endomorphism ring $\mathrm{End}({}_\Lambda T)$ is again a quasi-hereditary algebra. Since then, this tilting module is known as the Ringel's characteristic tilting module  associated with a quasi-hereditary algebra.

Because of the success of the applications of the theory of quasi-hereditary algebras,  it was natural to find useful generalizations of the notion of  
quasi-hereditary algebra. One  step in this direction was given by  V. Dlab, who  introduced the concept of standardly stratified algebra 
\cite{Dlab}. These finite dimensional algebras have been amply studied  \cite{ADL, AHLU, Frisk1, Frisk2, FriskMaz, FKM, Mazor1, Mazor2,MazOv, MazPark, MSX, MenSan, Webb, Xi}  and have become  an useful tool for different areas in mathematics.

Once we have the notion of standardly stratified algebra \cite{Dlab}, in the context of finite dimensional algebras,  a natural question is to find a more general class of algebras on which has sense to define the notion of standardly stratified algebra. These are precisely the rings with enough idempotents. These kind of rings appear very naturally in different contexts, for example, as a generalization of Ringel's notion of species
\cite{Rin2} or in connection with the Galois covering in the sense of Bongartz-Gabriel \cite{BG} or De la Pe\~na-Martinez \cite{DM}. More generally, this type of rings appears as ``Gabriel functor rings''  (see discussion after proposition 1 on pag. 346 in
\cite{Gabriel}) or ``rings with several objects'' in \cite{Mitchell}.

The context of rings with several objects (ringoids, in modern terminology) has become very fruitful as a tool that allows us to understand more deeply certain branches of mathematics. For example, motivated by the work on functor categories in \cite{Aus2, Aus},  
R. Mart\'{\i}nez-Villa and \O. Solberg  studied the Auslander-Reiten components of finite dimensional algebras. They did so, in order to stablish when the category of graded functors is noetherian  \cite{MVS1, MVS2, MVS3}. Recently R. Mart\'{\i}nez-Villa and M. Ort\'{\i}z studied in \cite{MVO1, MVO2} tilting theory in arbitrary functor categories. They proved that most of the properties that are satisfied by a tilting module over an Artin algebra  also hold true for functor categories. To mention some, Brenner-Butler's Theorem and Happel's Theorem are valid on this more general context.

Inspired by the works mentioned above and the fact that in the theory of quasi-hereditary algebras the notion of tilting module is relevant, 
M. Ort\'{\i}z introduced in \cite{Martin} the concept of quasi-hereditary category. He did so, in order to study the Auslander-Reiten components of a 
finite dimensional algebra $\Lambda.$ In a similar way as the standard modules appear in the theory of quasi-hereditary algebras, M. Ort\'{\i}z 
defined the  concept of standard functors, which turned out to be a generalization of the notion of standard modules  \cite{Martin}. In particular, 
he established a connection between highest weight categories and quasi-hereditary categories. He did so, by following the ideas introduced by 
H. Krause in \cite{Krause2}, that is, M. Ort\'{\i}z  compared the notion of standard objects in an abelian length category and standard subcategories 
of the category of $\mathcal{C}$-modules over a quasi-hereditary category $\C.$

In this paper, we define the notions of standardly stratified ringoid and quasi-hereditary ringoid. These definitions generalize the notion of quasi-hereditary category given by  M. Ort\'{\i}z in \cite{Martin}. To start with, we recall that for any class of objects $\B$ of a category $\C,$ $\ind\,\B$ denotes the class of 
iso-classes of local objects $B\in\B,$  where $B$ local means that $\End_\C(B)$ is a local endomorphism ring.

Let $\K$ be a commutative ring and  $\Lambda$ be a $\K$-algebra (possibly without $1$)  such that $\Lambda^2=\Lambda.$ For such algebra $\Lambda,$ we denote by 
$\Mod(\Lambda)$ the category of all the unitary left $\Lambda$-modules $M,$ where unitary means that $\Lambda M=M.$ The finitely generated unitary left $\Lambda$-modules form a full subcategory of $\Mod(\Lambda)$ and it is usually denoted by
 $\modu(\Lambda).$ The class of finitely generated projective objects in $\Mod(\Lambda)$ is denoted by $\proj(\Lambda).$ We denote by $f.\ell(\K)$ the class of all the $\K$-modules of  finite length. In this context, in the category $\Mod(\Lambda)$ usually there exist infinitely many
finitely generated indecomposable projective $\Lambda$-modules,
in contrast to the case when $\Lambda$ is an Artin $R$-algebra.
In order to define a right standardly stratified algebra, in the classical sense, 
we have to construct the family of standard modules $\Delta=\{\Delta(i)\}_{i=1}^{n},$ one standard module for each element in 
$\ind\,\proj(\Lambda^{op})=\{P(i)\}_{i=1}^{n}.$  In the general case of 
a  $\K$-algebra  without $1,$  it  is not clear that $\ind\,\proj(\Lambda^{op})$ is even a set (at first glance) and we do not have a reasonable description of this class.  In order to fix this problem, we consider a family $\{e_i\}_{i\in I}$ of orthogonal idempotents in $\Lambda$ satisfying mild conditions (sufficiency, Hom-finiteness). Using the family  $\{e_i\}_{i\in I},$ we produce  partitions $\tilde{\A}$ of the set $\ind\, \proj(\Lambda^{op})$ 
and each one of these partitions give us a sort of stratification of the class of finitely generated projective $\Lambda$-modules. 
Having a partition $\tilde{\A},$ as above, we can define the set of standard modules $\Delta:=\{\Delta(i)\}_{i<\alpha},$
where $\alpha$ is an ordinal number giving the size of the partition $\tilde{\A}.$
\

A $\K$-algebra  with enough idempotents ({\bf w.e.i $\K$-algebra}, for short) is a pair $(\Lambda,\{e_i\}_{i\in I}),$ where $\Lambda$ is a 
$\K$-algebra and $\{e_i\}_{i\in I}$ is a family of orthogonal idempotents of $\Lambda$ such that  $\Lambda=\oplus_{i\in I}\,e_i\Lambda=\oplus_{i\in I}\Lambda e_i.$ In this case  we have that $\Lambda^2=\Lambda.$ It is said that $(\Lambda,\{e_i\}_{i\in I})$ is {\bf Hom-finite} if
$\{e_j\Lambda e_i\}_{i,j\in I}\subseteq f.\ell(\K).$
\

Let  $(\Lambda,\{e_i\}_{i\in I})$ be a Hom-finite w.e.i. $\K$-algebra. Then, by Corollary \ref{WeiEquiv3} (b), for each $i\in I,$ there exists a unique (up to permutations)  family $\overline{e}_i:=\{e_{k,i}\}_{k=1}^{n_i}$ of primitive orthogonal idempotents in 
$\Lambda$ such that $e_i=\sum_{k=1}^{n_i}e_{k,i}.$ Denote by $\ind\,\{e_i\}_{i\in I}$ the quotient of the set 
$\cup_{i\in I}\overline{e}_i$ by the equivalence relation $\sim,$ where 
$f\sim g$ if, and only if, $f\Lambda\simeq g\Lambda.$ Let $[e]$ be the equivalence class of $e\in \cup_{i\in I}\overline{e}_i.$ Then, by Corollary \ref{WeiEquiv3'} (b), we have
\begin{center}
$\ind\,\proj(\Lambda^{op})=\{e\Lambda\;:\;[e]\in \ind\,\{e_i\}_{i\in I}\}.$
\end{center}
 The set of standard modules can be constructed by choosing a partition $\tilde{\A}=\{\tilde{\A}_i\}_{i<\alpha}$
 of the set $\ind\,\{e_i\}_{i\in I},$  where $\alpha$ is an ordinal number (the size of the partition $\tilde{\A}$) and each ordinal $i<\alpha$ is the $i$-th level of the given partition. 
 Define ${}_{\Lambda^{op}}P_e(i):=e\Lambda$ for any $[e]\in\tilde{\A}_i,$ and let
 ${}_{\Lambda^{op}}P:=\{{}_{\Lambda^{op}}P(i)\}_{i<\alpha},$ where ${}_{\Lambda^{op}}P(i):=\{{}_{\Lambda^{op}}P_e(i)\}_{e\in\tilde{\A}_i}.$ The
 family of $\tilde{\A}$-standard right $\Lambda$-modules ${}_{\Lambda^{op}}\Delta=\{\Delta(i)\}_{i<\alpha},$ where
 $\Delta(i):=\{\Delta_{e}(i)\}_{e\in \tilde{\A}_i},$ is defined as follows
 $$\Delta_{e}(i):=\frac{{}_{\Lambda^{op}}P_{e}(i)}{\mathrm{Tr}_{\oplus_{j<i}\overline{P}(j)}({}_{\Lambda^{op}}P_{e}(i))},$$
 where $\overline{P}(j):=\bigoplus_{r\in\tilde{\A}_j}\,{}_{\Lambda^{op}}P_r(j)$ and  $\mathrm{Tr}_{\oplus_{j<i}\overline{P}(j)}({}_{\Lambda^{op}}P_{e}(i))$ is the trace of the $\Lambda$-module $\oplus_{j<i}\overline{P}(j)$ in ${}_{\Lambda^{op}}P_{e}(i).$ We say that the pair $(\Lambda, \tilde{\A})$ is a right
{\bf standardly stratified} algebra if $\mathrm{Tr}_{\oplus_{j<i}\overline{P}(j)}({}_{\Lambda^{op}}P_{e}(i))\in \ltt{F}_f(\bigcup_{j<i}\Delta(j)),$
for any $i<\alpha$ and $e\in\tilde{\A}_i.$ Here $\ltt{F}_f(\bigcup_{j<i}\Delta(j))$ is the class of all the $\Lambda$-modules admitting a finite filtration in 
$\bigcup_{j<i}\Delta(j).$ Moreover, we say that $(\Lambda, \tilde{\A})$ is right {\bf quasi-hereditary} if it is standardly stratified and $\End (\Delta_e(i))$ is a division ring for any $[e]\in \tilde{A}_i$ and $i<\alpha.$
\

Let  $\Lambda$ be a basic Artin $\K$-algebra and let $\{e_i\}_{i=1}^n$ be a complete family of primitive ortogonal idempotents of $\Lambda.$ 
Then, we have that $\ind\,\{e_i\}_{i=1}^n=\{e_i\}_{i=1}^n.$ The classical notion of standardly stratified algebra for $\Lambda$ corresponds to 
the given one for the very particular pair $(\Lambda, \tilde{\mathbb{T}}),$ where $\tilde{\mathbb{T}}$ is the  one-point partition $\tilde{\mathbb{T}}=\{\tilde{\mathbb{T}}_i\}_{i<n},$ 
defined as $\tilde{\mathbb{T}}_i:=\{e_{i+1}\}$ for $i\in[0,n).$ Note that we can choose different partitions $\tilde{\A}$ of the set $\{e_i\}_{i=1}^n,$ not only the trivial one. 

In this paper, we also define ideally standardly stratified and ideally quasi-hereditary $\K$-ringoids. We explain their meaning in terms of rings with enough idempotents. Let $(\Lambda,\{e_i\}_{i\in S})$ be a w.e.i $\K$-algebra. An ideal $I\unlhd\Lambda$ is {\bf right stratifying} if $I^2=I$  and
 $eI\in\proj(\Lambda^{op})$ for any $e\in\{e_i\}_{i\in S}.$ We say that $I$ is {\bf right hereditary} if it is right stratifying and $I\rad(\Lambda)I=0.$
 A right stratifying (respectively,  hereditary) chain in $\Lambda$ is a chain $\{I_i\}_{i<\alpha}$ of ideals of $\Lambda$ such that
 $\sum_{i<\alpha}\,I_i= \Lambda$ and  $I_i/I'_i$ is right stratifying (respectively, hereditary) in $\Lambda/I'_i,$ where $I'_i:=\sum_{j<i}\,I_j.$ 
 \
 
 Assume now that $(\Lambda,\{e_i\}_{i\in S})$ is Hom-finite. Let $\tilde{\A}=\{\tilde{\A}_i\}_{i<\alpha}$ be a partition of the set 
 $\ind\,\{e_i\}_{i\in S}.$ The partition $\tilde{\A}$ induces a chain $\{I_{\A_i}\}_{i<\alpha}$ of ideals 
 in $\Lambda$ satisfying that $\sum_{i<\alpha}\,I_{\A_i}= \Lambda,$ where $I_{\A_i}$ is the ideal generated by the set of idempotents 
 $\{e\,:\;[e]\in \A_i \}$ and
 $\A_i:=\bigcup_{j\leq i}\,\tilde{\A}_j$ (see Lemma \ref{AuxB}). We say that $(\Lambda, \tilde{\A})$  is right
 {\bf ideally standardly stratified} (respectively, right {\bf ideally quasi-hereditary})  if the associated 
 chain $\{I_{\A_i}\}_{i<\alpha}$ of ideals in $\Lambda$ is right stratifying (respectively, hereditary). 
 
 The following question arises naturally: Are 
 the definitions of ideally standardly stratified (respectively, ideally quasi-hereditary) and standardly stratified (respectively, quasi-hereditary) 
 equivalent? In the case of an Artin algebra $\Lambda$ and the one-point  partition $\tilde{\mathbb{T}},$ defined above, it is well known that both 
 notions are equivalent. For the general case, we have the following results, that are a consequence of Theorems \ref{FilTrProj} and  \ref{stqs}. In order 
 to state the following two theorems, we recall (see in section 5) the notion of right noetherian partition. Let $(\Lambda,\{e_i\}_{i\in S})$ be a Hom-finite w.e.i $\K$-algebra and $\tilde{\A}=\{\tilde{\A}_i\}_{i<\alpha}$ be a partition of the set 
 $\ind\,\{e_i\}_{i\in S}.$ We say that $\tilde{\A}$ is {\bf right noetherian} if for any $i<\alpha$  and $[e]\in\tilde{\A}_i$ the following statements 
 hold true: the set $\{j<\alpha\;:\; eI_{\A_j}/eI'_{\A_j}\neq 0\}$ is finite and there is some $i_0<\alpha$ such that $eI_{\A_j}=e\Lambda$ for any $j\geq i_0.$
 \
 
 \noindent 
 {\bf Theorem A} Let $(\Lambda,\{e_i\}_{i\in S})$ be a Hom-finite w.e.i $\K$-algebra and $\tilde{\A}=\{\tilde{\A}_i\}_{i<\alpha}$ be a partition of $\ind\,\{e_i\}_{i\in S}.$ Then, the following statements are equivalent.
\begin{itemize}
\item[(a)] $(\Lambda,\tilde{\A})$ is right standardly stratified.
\item[(b)] The partition $\tilde{\A}$ is right noetherian and for any $i<\alpha,$ $[e]\in\tilde{\A}_i$ and $t<\alpha,$ we have that 
$eI_{\A_t}/eI'_{\A_t}$ is a finitely generated projective right $\Lambda/I'_{\A_t}$-module.
\end{itemize}

As a consequence of the theorem above, it can be shown (see Corollary \ref{stiss}) the following result.
\

\noindent
{\bf Corollary B} Let $(\Lambda,\{e_i\}_{i\in S})$ be a Hom-finite w.e.i $\K$-algebra and $\tilde{\A}$ be a partition of $\ind\,\{e_i\}_{i\in S}.$ Then, the following statements are equivalent.
\begin{itemize}
\item[(a)] $(\Lambda,\tilde{\A})$ is right standardly stratified.
\item[(b)] The partition $\tilde{\A}$ is right noetherian and $(\Lambda,\tilde{\A})$ is right ideally standardly stratified.
\end{itemize}

\noindent
{\bf Theorem C} Let $(\Lambda,\{e_i\}_{i\in S})$ be a Hom-finite w.e.i $\K$-algebra and $\tilde{\A}=\{\tilde{\A}_i\}_{i<\alpha}$ be a partition of $\ind\,\{e_i\}_{i\in S}.$ Then, the following statements are equivalent.
\begin{itemize}
\item[(a)] $(\Lambda,\tilde{\A})$ is right quasi-hereditary and $\Hom(\Delta_e(i),\Delta_{e'}(i))=0$ for $[e]\neq [e']$ in $\tilde{\A}_i$ and $i<\alpha.$
\item[(b)] The partition $\tilde{\A}$ is right noetherian and $(\Lambda,\tilde{\A})$ is right ideally quasi-hereditary.
\end{itemize}
 
Given a Hom-finite w.e.i $\K$-algebra $(\Lambda,\{e_i\}_{i\in S})$ and a partition $\tilde{\A}=\{\tilde{\A}_i\}_{i<\alpha}$ of 
$\ind\,\{e_i\}_{i\in S},$ we  have the family of standard modules $\Delta=\{\Delta(i)\}_{i<\alpha}$ and the  category $\ltt{F}_{f}(\Delta)$ of 
all the right $\Lambda$-modules which has a finite filtration through the objects of $\Delta.$ 
\

In this the paper, we also study some important properties of $\ltt{F}_{f}(\Delta).$ 
As in the classic case, we prove (see Theorem \ref{filt-sum-direct})  that if the standard modules $\Delta_{e}(i)$ are finitely presented, then 
$\ltt{F}_{f}(\Delta)$ is a Krull-Schmidt skeletally small category and all the objects in this category are finitely presented. Furthermore, 
the multiplicity $[M:\Delta_{e}(i)]$ of each $\Delta_{e}(i),$ for a module $M\in \ltt{F}_{f}(\Delta),$ does not depend on any $\Delta$-filtration of $M.$
In order to prove that fact, we introduce an analogous of the trace filtration given by Dlab and Ringel
\cite[Lemma 1.4]{DR1} and  characterize the modules which belongs to $\ltt{F}_{f}(\Delta)$ in terms of this trace filtration (see Theorem \ref{filtraza}).
It is worth mentioning that the proofs of these results uses transfinite induction, in contrast with the classic case, where the usual induction is enough to 
handle the situation. 
As an application of the trace filtration, we show that $\ltt{F}_{f}(\Delta)$
is closed under kerneles of epimorphisms between its objects, a fact that is well known for the classical case. 

\section{Preliminaries}

{\sc Left standardly stratified algebras.} For a positive integer $m,$ we set $[1,m]:=\{1,2,\ldots,m\}.$ Let $\Lambda$ be an Artin algebra. We denote by $\modu\,(\Lambda)$ the category of finitely generated left $\Lambda$-modules. Given a class $\X$ of $\Lambda$-modules, we denote by
$\mathcal{F}(\mathcal{X})$ the full subcategory of
 $\mathrm{mod}\,(\Lambda)$ consisting of all modules having a filtration by objects in
$\mathcal{\X}$. That is, a  finite chain
\[0=M_0\subseteq M_1\subseteq\cdots\subseteq M_m=M\] of submodules of
$M$ such that $M_i/M_{i-1}$ is isomorphic to a module in $\X$
for every $i\in[1,m].$ Note that $\mathcal{F}(\mathcal{X})$  is closed
under extensions. In general, $\mathcal{F}(\mathcal{X})$ fails to be
closed under direct summands \cite{Rin}.
\

  Let $n$ be the rank of the Grothendieck group of the Artin algebra $\Lambda.$ In this case, we can choose a complete family $\{S(i)\}_{i\in[1,n]}$ of representatives of the pairwise non isomorphic simple $\Lambda$-modules. We also have a complete family $\{P(i)\}_{i\in[1,n]}$ of representatives of the indecomposable and pairwise  non isomorphic projective $\Lambda$-modules, where $P(i)$ is the projective cover of the simple $S(i).$ Then, we have 
  that $\ind\,\proj\,(\Lambda)=\{P(i)\}_{i\in[1,n]}.$
\

Let  $\leq$ be a linear order on  $[1,n].$ Then, for each $i\in[1,n],$ we have the $\Lambda$-module $U(i):=\mathrm{Tr}_{\oplus_{j<i} P(j)}P(i),$ which is the trace of $\oplus_{j<i} P(j)$ in $P(i).$ Thus, $\Delta(i):=P(i)/U(i)$ is the $i$-th standard $\Lambda$-module and  $\Delta:=\{\Delta(i)\}_{i\in[1,n]}$ the family of standard $\Lambda$-modules. One can see that:
$\mathrm{Hom}_\Lambda(\Delta(j),\Delta(i))=0$ for $j< i,$ and $\Ext^1_\Lambda(\Delta(j),\Delta(i))=0$ for $j\leq i$ \cite{Rin}. It is said that the algebra $\Lambda$ is (left) standardly stratified, with respect to the linear order $\leq$ on $[1,n],$ if ${}_\Lambda\Lambda\in\mathcal{F}(\Delta).$ It is well known that ${}_\Lambda\Lambda\in\mathcal{F}(\Delta)$ if, and only if, $U(i)\in\mathcal{F}(\{\Delta(j): j< i\})$ for any $i\in[1,n].$

{\sc Functor categories and ringoids.} Let $\mathbb{K}$ be a commutative ring with $1.$  A category $\C$ is said to be a {\bf $\mathbb{K}$-category} if $\Hom_\C(X,Y)$ is a $\mathbb{K}$-module for any $(X,Y)\in\C^2,$ and the composition of morphisms in $\C$ is $\mathbb{K}$-bilinear.  We denote by $[\A,\B]$ the category of additive (covariant) functors between two $\mathbb{K}$-categories $\A$ and $\B,$ where $\A$ is skeletally small. For any $F,G\in[\A,\B],$ we have that
$\Hom_{[\A,\B]}(F,G)$ is the class $\Nat(F,G)$  of natural morphisms from $F$ to $G.$ For the sake of simplicity, we write $(-,?)$ instead of $\Hom(-,?)$ wherever this $\Hom(-,?)$ is defined. The term {\bf subcategory} means full subcategory.
\

Let $\C$ be a  $\mathbb{K}$-category. We say that an object $C\in\C$ is {\bf local} if  $\End_\C(C)$ is a local ring.   For any subclass
$\B$ of objects in
$ \C,$ the class of iso-classes of local objects $B\in\B$  will be denoted by $\ind\,\B.$ For any $B\in\B,$ which is local, we write $[B]$ for the corresponding iso-class. That is,
$\ind\,\B:=\{[B]\;\text{ such that $B\in\B$ is local}\}.$ For simplicity, sometimes we  write $B$ instead of
$[B].$ If $\C$ is an additive category, we denote by $\add\,(\B)$ the class of all direct summands of finite coproducts of copies of objects in $\B.$
\

A very useful tool in the theory of categories is Yoneda's Lemma. We state this lemma for the case of $\mathbb{K}$-categories since this is precisely the context where we are working. Let $\C$ be a $\mathbb{K}$-category.  Yoneda's Lemma states that Yoneda's functor
$${\small Y=Y_{\C}:\C\to [\C^{op},\Ab],\quad(a\xrightarrow{f}b)\mapsto(\Hom_\C(-,a)\xrightarrow{\Hom_\C(-,f)}\Hom_\C(-,b))}$$
is full and faithful. Moreover, for any $c\in\C,$ we have an isomorphism of abelian groups $\Hom(Y(c),F)\to F(c),$ $\alpha\mapsto \alpha_c(1_c).$
\

Following B. Mitchell in \cite{Mitchell}, we recall  that  a {\bf $\mathbb{K}$-ringoid} (or $\mathbb{K}$-algebra with several objects) is just a skeletally small $\mathbb{K}$-category. A {\bf ringoid} is just a $\mathbb{Z}$-ringoid (or ring with several objects). Note that any $\mathbb{K}$-ringoid is in particular a ringoid.
\

Let  $\mathfrak{S}$ be a ringoid. A left $\mathfrak{S}$-module is an additive
covariant functor $F:\mathfrak{S}\to \Ab,$ where $\Ab$ is the category of abelian groups. The category of left $\mathfrak{S}$-modules is  $\Mod\,(\mathfrak{S}):=[\mathfrak{S},\Ab].$  Note that $\Mod\,(\mathfrak{S})$ is abelian and  bicomplete,  since $\Ab$ is so. Furthermore, it can be shown  that $\Mod\,(\mathfrak{S})$ is a Grothendieck category.
\

A right $\mathfrak{S}$-module is an additive
contravariant functor $F:\mathfrak{S}\to \Ab.$ Thus, $\mathrm{Mod}_{\rho}(\mathfrak{S}):=\Mod\,(\mathfrak{S}^{op})$  is the category of right $\mathfrak{S}$-modules, where
$\mathfrak{S}^{op}$ is the opposite category of $\mathfrak{S}.$
\

We denote by $\Proj\,(\mathfrak{S})$ the class of projective left $\mathfrak{S}$-modules and $\proj\,(\mathfrak{S})$ denotes the class of finitely generated projective left $\mathfrak{S}$-modules. We also have
$$\Proj_\rho(\mathfrak{S}):=\Proj\,(\mathfrak{S}^{op})\;\text{ and }\;
\proj_\rho(\mathfrak{S}):=\proj\,(\mathfrak{S}^{op}).$$

A $\mathbb{K}$-ringoid with one element is basically a $\mathbb{K}$-algebra with $1.$ Namely, let $\mathfrak{S}$ be a $\mathbb{K}$-ringoid with one element. In this case, $\mathfrak{S}$ has a
single element $*$ and $\End_{\mathfrak{S}}(*)$ is a $\mathbb{K}$-algebra with $1.$ Reciprocally, let $S$ be any $\mathbb{K}$-algebra with $1.$ One
can see $S$ as the $\mathbb{K}$-ringoid with one object   $\mathfrak{R}_S,$  where $\mathfrak{R}_S$ has a  single element $*$ and $\End_{\mathfrak{R}_S}(*):=S.$ Note that the category
$\Mod\,(\mathfrak{R}_S)$ and the category $\Mod\,(S)$ of left $S$-modules are isomorphic.
\

In all that follows, we shall consider the above notation. Let $\mathfrak{R}$ be a ringoid. We say that a ringoid $\mathfrak{S}$ is Morita equivalent to $\mathfrak{R}$ if the module categories
$\Mod_\rho(\mathfrak{R})$ and $\Mod_\rho(\mathfrak{S})$ are equivalent.
\

Using Yoneda's functor $Y:\mathfrak{R}\to \Mod_\rho(\mathfrak{R})$ it can be proved that $M\in\proj_\rho(\mathfrak{R})$ iff $M$ is a direct summand of $\coprod_{i\in I}\,Y(a_i)$ for some finite family $\{a_i\}_{i\in I}$ of objects in $\mathfrak{R}.$ Thus, the ringoid $\mathfrak{R}$ can be seen as a full subcategory of $\proj_\rho(\mathfrak{R}).$
\

Following M. Auslander \cite{Aus}, it is said that $M\in\Mod_\rho(\R)$ is {\bf finitely presented} if there is an exact sequence $P_1\to P_0\to M\to 0,$ where $P_1,P_0\in\proj_\rho(\R).$ We denote by $\finp_\rho(\R)$
the full subcategory of $\Mod_\rho(\R)$ whose objects are all the finitely presented right $\R$-modules.  A {\bf projective cover} of $M\in \Mod_\rho(\R)$ is an essential epimorphism $P\to M$ in $\Mod_\rho(\R)$ with $P\in\Proj_\rho(\R).$  A projective presentation $P_1\xrightarrow{g} P_0\xrightarrow{f} M\to 0$ is {\bf minimal} if the epimorphisms $P_0\xrightarrow{f} M$ and $P_1\to \imagen\,(g)$ are projective covers.
\

Let $\R$  be an {\bf additive} ringoid, that is, $\R$ is a skeletally small aditive category. It is well known (see \cite{Aus}, \cite[Theorem 1.4]{Fre} and \cite{He}) that  $\finp_\rho(\R)$ is a full abelian subcategory of $\Mod_\rho(\R)$ if and only if $\R$ has pseudokernels.
\

We say that a ringoid $\mathfrak{R}$ is {\bf thick} if $\mathfrak{R}$ is an additive category whose idempotents
split. In this case, Yoneda's functor $Y:\mathfrak{R}\to\Mod_\rho(\mathfrak{R})$ induces an equivalence of categories $\mathfrak{R}\simeq\proj_\rho(\mathfrak{R}).$ Therefore, there is not loss of generality, from the  viewpoint of module categories,  if we work with thick ringoids. A ringoid $\mathfrak{R}$ is {\bf Krull-Schmidt} (KS ringoid, for short) if it is a Krull-Schmidt category (that is an additive category in which every non-zero object decomposes into a finite direct sum of objects having local endomorphism ring). It can be shown \cite[Corollary 4.4]{Krause} that any ringoid is  Krull-Schmidt if it is thick and the endomophism ring of every object is semiperfect.

\begin{lem}\label{Fpp} Let $\R$ be a Krull-Schmidt $\K$-ringoid. Then, any $M\in\finp_\rho(\R)$ has a minimal projective presentation in 
$\proj_\rho(\R).$
\end{lem}
\begin{dem} For any $C\in\R,$ we have that $R_C:=\End_\R(C)^{op}$ is a semi-perfect ring. Then, the result follows from \cite[Corollary 4.13]{Aus}.
\end{dem}
\vspace{0.2cm}

Let $\R$ be a thick $\K$-ringoid. For any additive full subcategory $\B$ of $\R,$ we consider the class $I_\B$ of all the morphisms in $\R$ which factor through objects of $\B.$ Note that $I_\B$ is an ideal of $\R,$ and it is known as the ideal associated with $\B.$ For $M,N\in\Mod_\rho(\R),$ we denote by $\mathrm{Tr}_M(N)$ the trace of $M$ in $N.$
\

We say that a $\K$-ringoid $\R$ is {\bf Hom-finite} if the $\K$-module $\Hom_\R(a,b)$ is of finite length, for any $(a,b)\in\R^2.$ A {\bf locally finite}  $\K$-ringoid is a $\K$-ringoid which is Hom-finite and Krull-Schmidt.  A  locally finite $\K$-ringoid with pseudokernels is called  {\bf strong locally finite $\K$-ringoid}.

\begin{lem}\label{LARp1} For a Krull-Schmidt $\K$-ringoid $\R,$ the following statements hold true.
\begin{itemize}
\item[(a)] $\proj_\rho(\R)=\{P=\oplus_{i\in I}\,Y(a_i)\text{ for a finite family } \{a_i\}_{i\in I} \text{ in } \ind\,(\R)\}.$
\item[(b)] $\ind\,(\proj_\rho(\R))=\{Y(a)\;:\; a\in\ind\,(\R)\}.$
\item[(c)] For any $a,b\in\R,$ we have that $Y(a)\simeq Y(b)$ iff $a\simeq b.$
\end{itemize}
\end{lem}
\begin{dem} We start by proving (c).  Let $\eta:Y(a)\to Y(b)$ be an isomorphism of functors. Consider $f_a\in\Hom_\R(a,b)$ and $g_b\in\Hom_\R(b,a),$
where $f_a:=\eta_a(1_a)$ and $\eta_b(g_b)=1_b.$ By using $\eta:Y(a)\to Y(b),$ it can be shown that $f_a\circ g_b=1_b$ and $g_b\circ f_a=1_a.$
\

The proof of (a) and (b) follows from (c), since $\R$ is a Krull-Schmidt $\K$-ringoid and thus Yoneda's functor $Y:\R\to\Mod_\rho(\R)$ gives an equivalence between  $\R$ and $ \proj_\rho(\R).$
\end{dem}

Let $\R$ be a $\K$-ringoid. For any $c\in\R,$ we consider the opposite ring $\R_c:=\End_\R(c)^{op},$ which is a $\K$-algebra via the morphism of
rings
$$\varphi_c:\K\to \R_c,\quad k\mapsto (k1_c).$$
Therefore, by change of rings, the morphism $\varphi_c$ induces the functor
$$\varphi_c^*:\Mod(\R_c)\to \Mod\,(\K).$$
On the other hand, we have the evaluation functor at $c\in\R$
$$E_c:\Mod_\rho(\R)\to \Mod(\R_c),\quad(M\xrightarrow{\alpha}N)\mapsto (M(c)\xrightarrow{\alpha_c}N(c)),$$
where the structure of left $\R_c$-module on $M(c)$ is given by the action 
\begin{center}$fx:=M(f)(x)$ for $f\in \R_c$ and $x\in M(c).$
\end{center}
Thus, we have the functor
$\varepsilon_c:\Mod_\rho(\R)\to \Mod\,(\K),$ where $\varepsilon_c:=\varphi_c^*\circ E_c.$
\

Let $[\R^{op},\Mod(\K)]$ be the category of $\K$-linear contravariant functors from $\R$ to $\Mod(\K).$
It can be shown, that
$$\varepsilon:\Mod_\rho(\R)\to [\R^{op},\Mod(\K)], (M\xrightarrow{\lambda}N)\mapsto(\varepsilon(M)\xrightarrow{\varepsilon(\lambda)}\varepsilon(N))$$
is an isomorphism of categories, where $\varepsilon(M)(a):=\varepsilon_a(M),$ $\varepsilon(\lambda)_a:=\lambda_a,$ and for any $a\stackrel{f}{\to}b$ in $\R,$
we set $\varepsilon(M)(f):=M(f).$ Hence, in this case, we can identify $\Mod_\rho(\R)$ with $[\R^{op},\Mod(\K)].$
\

Let $M\in\Mod_\rho(\R)$ for some $\K$-ringoid $\R.$ The {\bf support} of $M$ is the set $\Supp\,(M):=\{e\in\ind(\R)\;:\; M(e)\neq 0\}.$ We say that $\R$ is
{\bf right support finite} if $\Supp\,(Y(e))$ is finite for any $e\in\ind(\R),$ where $Y(e):=\Hom_\R(-,e).$

\begin{lem}\label{LARp2} Let $\R$ be a locally finite $\K$-ringoid and $\B$ be an additive full subcategory of $\R.$ Then,  the following statements hold true.
 \begin{itemize}
  \item[(a)] $\mathrm{Tr}_{\{Y(b)\}_{b\in\B}}\,(Y(e))=I_\B(-,e),$ for any $e\in\R.$
 \item[(b)]  If $\R$ is right support finite, then $Y(e)/I_\B(-,e)\in\finp_\rho(\R),$ for any $e\in\ind(\R).$
 \end{itemize}
\end{lem}
\begin{dem} (a) The proof of \cite[Lemma 3.1]{Martin} can be adapted to get (a).
\

(b) Let $\R$ be right support finite and $e\in\ind(\R).$ By (a), we get
$$I_\B(-,e)=\Tr_{\{Y(b)\}_{b\in\B}}\,(Y(e))=\Tr_{\bigoplus_{b\in\ind(\B)}\, Y(b)}\,(Y(e)).$$
Since $\Hom(\bigoplus_{b\in\ind(\B)}\, Y(b),Y(e))=\prod_{b\in\ind(\B)}Y(e)(b)$ and $\R$ is right support finite,  there are some $b_1,b_2,\cdots,b_n$ in
$\ind(\B)$  and $Q:=\oplus_{i=1}^nY(b_i)$   such that
$$\Hom(\bigoplus_{b\in\ind(\B)}\, Y(b),Y(e))=\Hom(Q,Y(e)).$$
Note that  $\Hom(Q,Y(e))$ is a $\K$-module of finite length, since $\R$ is Hom-finite. Therefore $I_\B(-,e)=\Tr_Q(Y(e))$ is a finitely generated right $\R$-module. Finally, by
\cite[Proposition 4.2 (c)]{Aus} we get (b).
\end{dem}

\begin{pro}\label{LARp3}
Let $\R$ be a locally finite $\K$-ringoid. Then $\finp_\rho(\R)$ is a locally finite $\K$-ringoid.
\end{pro}
\begin{dem}
First, we prove that
$\finp_\rho(\R)$ is  $\Hom$-finite.
 Indeed,  let $F, G\in \finp_\rho(\R).$ Then, there are morphisms $a\stackrel{f}{\to}b$  and $a'\stackrel{f'}{\to}b'$ in $\R$ and  exact sequences in 
 $\Mod_\rho(\R)$
\begin{eqnarray}
Y(a)\xrightarrow{Y(f)}Y(b)\xrightarrow{\lambda} F\to 0,\label{presentationF}\\
Y(a')\xrightarrow{Y(f')}Y(b')\xrightarrow{\lambda'} G\to 0.\label{presentationG}
\end{eqnarray}
By (\ref{presentationG}) we get an epimorphism $\Hom_\R(b,b')\xrightarrow{\lambda'_b} G(b)$ of
$\K$-modules,  and since
$\R$ is Hom-finite, we get that $G(b)$ is a $\K$-module of finite length. By applying the functor $(-,G)$ to the sequence (\ref {presentationF}), we obtain
 a monomorphism $(\lambda,G):(F,G)\to (Y(b),G)$ of $\K$-modules. Therefore $(F,G)$ is of finite length since $(Y(b),G)\simeq G(b).$ In particular $\End(M)$  is
 a left Artin ring for any $M\in \finp_\rho(\R).$
\

Now, we prove that $\finp_\rho(\R)$  is a  Krull-Schmidt $\K$-category. By \cite[Proposition 4.2 (d)]{Aus}, it follows that the idempotents in $\finp_\rho(\R)$ split, and moreover it  is an additive category. Finally, from \cite[Corollary 4.4]{Krause}, we get that $\finp_\rho(\R)$  is a  Krull-Schmidt $\K$-category since $\End(M)$  is a semi-perfect  ring, for any $M\in \finp_\rho(\R).$
\end{dem}
\vspace{0.2cm}

{\sc Filtrations.}
Let $\A$ be an abelian category and $\ltt{X}\subseteq \A.$ We denote by $\ltt{X}^{\oplus}$ the class of all the objects of $\A$ which are a finite direct sum of
objects in $\ltt{X}.$
\

 We say that $M\in \A$ is
$\ltt{X}$-$\textbf{filtered}$ if there exists a continuous chain $\{M_{i}\}_{i<\alpha}$ of subobjects
 of $M,$ for some ordinal number $\alpha,$ such that $M_{i+1}/M_{i}\in \ltt{X}^{\oplus}$ for $i+1\leq\alpha.$
In case $\alpha<\aleph_0,$ we say that $M$ has a {\bf finite $\X$-filtration} of length $\alpha.$ We denote by $\ltt{F}(\ltt{X})$ the class of
objects which are $\ltt{X}$-filtered and by
$\ltt{F}_f(\ltt{X})$ the class of objects  having a  finite filtration. Note that, for $M\in \ltt{F}_{f}(\X),$ the $\X$-length of $M$ can be defined as follows
$$\ell_{\X}(M):=\mathrm{min}\,\{n\in \mathbb{N}\mid M\,\,
\text{has an $\X$-filtration of length}\,\, n\}$$
\

 By using the notion of $\X$-length and induction, it can be proven the following useful remark.
\begin{rk}\label{filclodesext}
Let $\ltt{X}$ be a class of objects in an abelian category $\A.$ Then, the class $\ltt{F}_f(\ltt{X})$ is closed under extensions.
\end{rk}

\section{Standardly stratified ringoids}

In this section we define the concept of standardly stratified algebra for the class of rings with several objects.  We also prove some main properties which generalize several well known facts from the classical theory of stardardly stratified algebras.

\begin{defi} Let $\R$ be  a Krull-Schmidt $\K$-ringoid and $\C\subseteq\R$ be a class of objects of $\R$ such that $\add\,(\C)=\C.$ Let  $\tilde{\A}:=\{\tilde{\A}_{i}\}_{i<\alpha}$
be a  partition of the set $\ind(\C),$ where  $\alpha$ is an ordinal number (the size of the partion $\tilde{\A}$). For each $i\in[0,\alpha),$ we set $\A_i:=\bigcup_{j\leq i}\,\tilde{\A_j}$ and
$\B_i(\A):=\add\,(\A_i).$ We say that $\B(\A):=\{\B_i(\A)\}_{i<\alpha}$ is the family of subcategories of $\C$ related to the partition $\tilde{\A}.$ We
denote by $\wp(\C)$ the class of all the partitions of the set $\ind\,(\C).$
\end{defi}

\begin{defi}\label{AS} Let $\R$ be  a Krull-Schmidt $\K$-ringoid and $\C\subseteq\R$ be such that $\add\,(\C)=\C.$
Let $\B:=\{\B_{i}\}_{i<\alpha}$ be a family of subcategories of $\C,$ where $\alpha$ is an  ordinal number (the size of the family $\B$). We say that
$\B$ is $\textbf{admissible}$ in $\C,$ if the following conditions hold true:
\begin{enumerate}
\item [(a)] $\add\,(\B_{i})=\B_{i}$ for any $i<\alpha;$
\item [(b)] $\B_{i}\subseteq \B_{j}$ if $i\leq j<\alpha;$
\item[(c)]  $\C=\bigcup_{i<\alpha}\,\B_i;$
\item [(d)] $\sigma_{i}(\B):=\ind(\B_{i})-\bigcup_{j<i}\,\ind(\B_{i})\neq\emptyset$ for any $ i< \alpha.$
\end{enumerate}
We call  $\sigma_{i}(\B)$ the {\bf $i$-th section} of $\B.$ An admissible family $\B$ in $\C$ is said to be $\textbf{exhaustive}$ in $\R,$ if $\C=\R.$ We set
$\sigma(\B):=\{\sigma_{i}(\B)\}_{i<\alpha}.$ The class of all the admissible families of subcategories of $\C$ will be denoted by $\AF(\C).$
\end{defi}

\begin{pro}\label{P=AS} Let $\R$ be  a Krull-Schmidt $\K$-ringoid and $\C\subseteq\R$ be a class of subobjects of $\R$ such that $\add\,(\C)=\C.$ Then, the correspondence
$\sigma:\AF(\C)\to \wp(\C),$ $\B\mapsto \sigma(\B),$ is bijective with inverse $\tilde{\A}\mapsto \B(\A).$
\end{pro}
\begin{dem} {\it{From admissible families to partitions}}: Let $\B=\{\B_i\}_{i<\alpha}$ be an admissible family in $\C.$ We prove that $\sigma(\B)$ is a partition of
 $\ind(\C)$ and $\B(\sigma(\B))=\B.$ By the definition of admissible families, we  have that $\sigma_{i}(\B)$ is not empty. Furthermore, by Definition \ref{AS} (b) and (d), we get that
 \begin{center}
 $\sigma_{i}(\B)=\bigcap_{j<i}\,(\ind(\B_{i})-\ind(\B_{j})),$ for any $i<\alpha.$
 \end{center}

 Let us check that $\ind\,(\C)=\bigcup_{i<\alpha}\,\sigma_i(\B).$ Consider $X\in\ind\,(\C).$ Then, by Definition \ref{AS} (c) there is some $j<\alpha$ such that
 $X\in\ind\,(\B_j)$ and thus
 the set $S:=\{j<\alpha\;:\;X\in\ind\,(\B_j)\}$ is not empty. Now, for $k:=\min S$ it follows that $X\in\ind\,(\B_k)$ and $X\not\in\ind\,(\B_j)$ for any $j<k,$
 which means that $X\in\sigma_k(\B).$
 \

 We show that $\sigma_k(\B)\cap\sigma_l(\B)=\emptyset$ for $k<l<\alpha.$ Indeed, suppose that there is some $X\in\sigma_k(\B)\cap\sigma_l(\B).$ In particular $X\in \sigma_l(\B)$ and thus, for any $j<l$ $X\in \ind\,(\B_l)$ and $X\not\in\ind\,(\B_j).$ But, for $j=k,$ the former conditions say
 that $X\not\in\ind\,(\B_k),$ contradicting that $X\in\sigma_k(\B).$
 \

 Let $D:=\B(\sigma(\B)).$ We assert that $D=\B.$ Consider some $i<\alpha.$ By definition, we have
 $$D_i:=\add\,(\bigcup_{j\leq i}\sigma_j(\B))=\add\,\big(\bigcup_{j\leq i}\big( \ind\,(\B_j)- \bigcup_{k<j}\,\ind\,(\B_k) \big) \big).$$
 Therefore, in order to prove that $D_i=\B_i,$  it is enough to show that
 $$\ind\,(\B_i)= \bigcup_{j\leq i}\big( \ind\,(\B_j)- \bigcup_{k<j}\,\ind\,(\B_k) \big) .$$
 Let $X\in \ind\,(\B_i).$ Thus, the set $S_X:=\{j\leq i<\alpha\;:\; X\in\ind\,(\B_j)\}$ is not empty. Then for $k_0:=\min S_X,$ we get that
 $X\in \ind\,(\B_{k_0})$ and $X\not\in\ind\,(\B_l)$ for any $l<k_0.$ Therefore, for $k_0\leq i$ we obtain that
 $X\in \ind\,(\B_{k_0})$ and $X\not\in\bigcup_{l<k_0}\ind\,(\B_l).$ This says that
 $X\in \bigcup_{j\leq i}\big( \ind\,(\B_j)- \bigcup_{k<j}\,\ind\,(\B_k) \big),$ proving that $D_i=\B_i.$
 \

 {\it{From partitions to admissible families}}: Let $\tilde{\A}=\{\tilde{\A}_i\}$ be a partition of the set $\ind\,(\C).$ Let $\B(\A)=\{\B_i(\A)\}_{i<\alpha}$ be
 the family of subcategories of $\C$ related to the partition $\tilde{\A}.$ Note that \ref{AS} (a), (b) and (c) hold true by construction. In order to prove that
 $\B(\A)\in\AF(\C)$ and $\sigma(\B(\A))=\tilde{\A},$ it is enough to show that $\tilde{\A}_i=\sigma_i(\B(\A))$ for any $i<\alpha.$
 \

 Let $i<\alpha.$ For any $X\in\ind\,(\C), $ we assert that
 \[(*)\quad X\in\sigma_i(\B(\A))\quad \Leftrightarrow\quad \forall\,j<i,\;\exists\,k\in(j,i]\;\text{such that}\,X\in\tilde{\A}_k. \]
 Indeed, the assertion above follows from the following sequel of equivalences
 \begin{align*}
 X\in\sigma_i(\B(\A))& \quad \Leftrightarrow\quad X\in \bigcap_{j<i}\big(\ind\,\B_i(\A) - \ind\,\B_j(\A)\big)\\
    & \quad \Leftrightarrow\quad \forall\; j<i\quad X\in\bigcup_{k\leq i}\tilde{\A}_k\quad\text{and}\quad X\not\in\bigcup_{l\leq j}\tilde{\A}_k\\
    & \quad \Leftrightarrow\quad \forall\,j<i,\;\exists\,k\in(j,i]\;\text{such that}\,X\in\tilde{\A}_k.
 \end{align*}
 By $(*),$ it is clear that $\tilde{\A}_i\subseteq \sigma_i(\B(\A)).$ Let $X\in \sigma_i(\B(\A)).$ Then by $(*)$ there is some $k\in(j,i]$ such that
 $X\in\tilde{\A}_k.$ Suppose that $k<i.$ Then, again by $(*)$ there is some $k'\in(k,i]$ such that $X\in\tilde{\A}_{k'}.$ Therefore
 $X\in\tilde{\A}_{k}\cap\tilde{\A}_{k'},$ contradicting that $\tilde{\A}$ is a partition of $\ind\,(\C).$ Then, $k=i$ and thus $X\in\tilde{\A}_i.$
\end{dem}

Associated to a partition $\tilde{\A}$ of $\ind\,(\C),$ as above, we can compute the $(\tilde{\A},\C)$-standard right $\R$-modules. These modules play an important role in the definition of a right  standardly stratified ringoid. In order to define such modules, we consider the Yoneda's contravariant
functor $Y:\R\to \Mod_\rho(\R),$ where $Y(e):=\Hom_\R(-,e).$

\begin{defi} Let $\R$ be  a Krull-Schmidt $\K$-ringoid and $\C\subseteq\R$ be a class of objects of $\R$ such that $\add\,(\C)=\C,$ and let $\tilde{\A}=\{\tilde{\A}_i\}_{i<\alpha}$ be
a partition of the set $\ind\,(\C).$ Consider the projective right $\R$-modules
$P^{op}_e(i):=Y(e)$ for $e\in\tilde{\A}_i$ and $i<\alpha.$ Let $P^{op}=P^{op}(\tilde{\A}):=\{ P^{op}(i)\}_{i<\alpha}$ where
$P^{op}(i):=\{P^{op}_e(i)\}_{e\in\tilde{\A}_i}.$ We say that $P^{op}(\tilde{\A})$ is the family of projective modules associated with the partition
$\tilde{\A}.$   We define the family ${}_{(\tilde{\A},\C)}\Delta=\{\Delta(i)\}_{i<\alpha}$ of $(\tilde{\A},\C)$-standard right $\R$-modules,  where
$\Delta(i):=\{\Delta_e(i)\}_{e\in\tilde{\A}_i}$ is defined as follows
$$\Delta_{e}(i):=\frac{P^{op}_{e}(i)}{\mathrm{Tr}_{\oplus_{j<i}\overline{P}(j)}(P^{op}_{e}(i))},$$
where $\overline{P}(j):=\bigoplus_{r\in\tilde{\A}_j}\,P^{op}_r(j)$ and  $\mathrm{Tr}_{\oplus_{j<i}\overline{P}(j)}(P^{op}_{e}(i))$ is the trace of
$\oplus_{j<i}\overline{P}(j)$ in $P^{op}_{e}(i).$ In case $\R=\C,$ we just write ${}_{\tilde{\A}}\Delta$ instead of ${}_{(\tilde{\A},\C)}\Delta,$ and we say
that ${}_{(\tilde{\A},\C)}\Delta$ is the family of $\tilde{\A}$-standard right $\R$-modules.
 \end{defi}

\begin{defi} Let $\R$ be  a Krull-Schmidt $\K$-ringoid and $\C\subseteq\R$ be a class of objects of $\R$ such that $\add\,(\C)=\C.$ For any admissible family $\B=\{\B_i\}_{i<\alpha}$
of subcategories of $\C,$ we know by Proposition \ref{P=AS} that $\sigma(\B)$ is a partition of $\ind\,(\C).$ Then,
${}_{(\B,\C)}\Delta:={}_{(\sigma(\B),\C)}\Delta$  is called the family of
$(\B,\C)$-standard $\R$-modules. In case $\R=\C,$ we just write ${}_{\B}\Delta$ instead of
${}_{(\B,\C)}\Delta,$ and we say
that ${}_{\B}\Delta$ is the family of $\B$-standard right $\R$-modules.
\end{defi}

 Let $\R$ be  a Krull-Schmidt $\K$-ringoid and $\C\subseteq\R$ be a class of objects of $\R$ such that $\add\,(\C)=\C.$ For any partition $\tilde{\A}$ 
 of $\ind\,(\C),$ we point out that by Proposition \ref{P=AS}, it holds that ${}_{(\B(\tilde{\A}),\C)}\Delta={}_{(\tilde{\A},\C)}\Delta.$

\begin{defi}\label{standardly stra ringoid} A right {\bf standardly  stratified $\K$-ringoid} is a pair $(\mathfrak{R},\tilde{\A}),$
where $\mathfrak{R}$ is a Krull-Schmidt $\K$-ringoid and $\tilde{\A}$ is a partition of $\ind\,(\R)$ such that the  $\tilde{\A}$-standard family
$\Delta={}_{\tilde{\A}}\Delta$ of right $\R$-modules satisfies the following condition,  for any $i<\alpha$ and $e\in\tilde{\A}_i,$
$$\mathrm{Tr}_{\oplus_{j<i}\overline{P}(j)}(P^{op}_{e}(i))\in \ltt{F}_f(\bigcup_{j<i}\Delta(j)).$$
\end{defi}

\begin{defi}\label{quasihedringoid} A right standardly stratified
$\K$-ringoid $(\mathfrak{R},\tilde{A})$ is {\bf  quasi-hereditary} if $\End(\Delta_{e}(i))$ is a division ring, for any $e\in\tilde{\A}_i$ and $i<\alpha.$
\end{defi}

Let $\Lambda$  be a basic Artin $\K$-algebra and let $\{e_i\}_{i=1}^n$ be a complete family of primitive ortogonal idempotents of $\Lambda.$ There is
a $\K$-ringoid $\R(\Lambda),$ associated to $\Lambda,$  where the objects are $e_1,e_2,\cdots,e_n$ and the morphisms  are
$\Hom_{\R(\Lambda)}(e_i,e_j):=e_j\Lambda e_i$ for any $1\leq i,j\leq n.$ The composition of morphism  in $\R(\Lambda)$
is just the multiplication in $\Lambda.$ Note that $\ind\,\R(\Lambda)=\{e_1,e_2,\cdots,e_n\}.$
We have the canonical isomorphism of categories
$$\delta:\Mod_\rho(\R(\Lambda))\to\Mod(\Lambda^{op}),\quad M\mapsto \oplus_{i=1}^n M(e_i).$$
For the Yoneda's functor $Y:\R(\Lambda)\to \Mod_\rho(\R(\Lambda)), $ we have
$$\delta(Y(e_i))=\oplus_{j=1}^n\Hom_{\R(\Lambda)}(e_j,e_i)=\oplus_{j=1}^n e_i\Lambda e_j= e_i\Lambda.$$
Let  $\tilde{\A}_i:=\{e_i\},$ $P(i):=Y(e_i)$ and $P:=\{\{P(i)\}\}_{i=1}^n.$ Consider the standard modules 
${}_{\R(\Lambda)}\Delta:={}_{\tilde{\A}}\Delta.$ Note that $\delta({}_{\R(\Lambda)}\Delta(i))\simeq{}_{\Lambda}\Delta(i)$ for any
$i\in[1,n].$ Therefore $(\R(\Lambda),\tilde{\A})$ is a right standardly stratified $\K$-ringoid if, and only if,
 $\Lambda$ is a right standardly stratified $\K$-algebra as in the classical sense.\\

We recall that for a given abelian category $\A$ and $\ltt{X}\subseteq \A,$ $\ltt{X}^{\oplus}$ denotes the class of all the objects of $\A$ which are a finite direct sum of
objects in $\ltt{X}.$
\

\begin{pro}\label{DSuppFin} Let $\R$ be  a Krull-Schmidt $\K$-ringoid and $\C\subseteq\R$ be a class of objects in $\R$ such that $\add\,(\C)=\C.$ For any admissible family
$\B=\{\B_{i}\}_{i< \alpha}$ of subcategories of $\C,$ the following statements hold true.
\begin{itemize}
\item[(a)] For any $e\in\sigma_i(\B),$ we have
$${}_{(\B,\C)}\Delta_e(i)=\frac{Y(e)}{\Tr_{\{Y(t)\}_{t\in{\bigcup_{j< i}\B_{j}}}}(Y(e))}.$$
Moreover, if $\R$ is locally finite, then
$$\Tr_{\{Y(t)\}_{t\in{\bigcup_{j< i}\B_{j}}}}(Y(e))=I_{\cup_{j< i}\B_{j}}(-,e)\quad\text{and}\quad {}_{(\B,\C)}\Delta_e(i)\neq 0.$$
\item[(b)] If $\R$ is  locally finite and right support finite, then
${}_{(\B,\C)}\Delta^\oplus\subseteq\finp_\rho(\R).$
\end{itemize}
\end{pro}
\begin{dem} (a) We have ${}_{(\B,\C)}\Delta={}_{(\sigma(\B),\C)}\Delta$ and $\sigma_i(\B)=\ind(\B_i)-\bigcup_{j<i}\,\ind(\B_j).$ We assert
that
\[(*)\quad\add\big (\bigcup_{j<i}\,\sigma_j(\B)\big)=\bigcup_{j<i}\,\B_j. \]

Indeed,  $\add(\bigcup_{j<i}\B_{j})=\bigcup_{j<i}\B_{j}$ since $\B_{j}\subseteq \B_{j'}$ if $j\leq j'$ and $\add(\B_{j})=\B_{j}$
for every $j$. Now, using $\sigma_{j}(\B)\subseteq \B_{j}$ for every $j$, it follows that
$\bigcup_{j<i}\sigma_{j}(\B)\subseteq \bigcup_{j<i} \B_{j}$. Then we have that
$\add\big (\bigcup_{j<i}\,\sigma_j(\B)\big)  \subseteq \add(\bigcup_{j<i}\B_{j})=\bigcup_{j<i}\B_{j}$.
Now, let $X\in \bigcup_{j<i}\B_{j}$, then there exists $j'<i$ such that $X\in \B_{j'}$.
Thus $X=\bigoplus_{k=1}^{n} X_{k}$  with $X_{k}\in B_{j'}$ and local. For
each $X_{k}$ consider the set $S(X_{k}):=\{j<i\;:\; X_{k}\in \B_j\}$ which is not empty. For $j_{k}:=\min S(X_{k}),$ it follows that
$X_{k}\in(\ind (\B_{j_{k}})-\bigcup_{j<j_{k}}\,\ind(\B_j))=\sigma_{j_{k}}(\B)$ and therefore
$X\in \add\big (\bigcup_{j<i}\,\sigma_j(\B)\big);$ proving $(*).$

Using $(*)$ and  $\overline{P}(j)=\bigoplus_{r\in\sigma_j(\B)}\,P^{op}_r(j)=\bigoplus_{r\in\sigma_j(\B)}\,Y(r),$   we obtain the following sequence of equalities
\[
\begin{split}
\text{{\LARGE Tr}}_{_{\displaystyle {}_{\bigoplus_{\!\!\!\!\!\!\!\!\!_{\displaystyle {}_{_{j<i}}}}\overline{P}(j)}}}(P^{op}_e(i)) & =
\quad \quad\text{{\LARGE Tr}}_{_{\displaystyle {}_{\big\{\bigoplus_{\!\!\!\!\!\!\!\!\!_{\displaystyle {}_{_{j<i}}}}
\bigoplus_{\!\!\!\!\!\!\!\!\!_{\displaystyle {}_{_{r\in \sigma_{j}(B)}}}}\!\!\!\!\!\!\!\!\!\!Y(r)\big\}}}}
(Y(e))\\
& =\quad \quad \text{{\LARGE Tr}}_{_{\displaystyle {}_{{\big\{Y(t)\big \}}_{t\in \bigcup_{\!\!\!\!\!\!_{\displaystyle {}_{_{j<i}}}}\!\!\!\!\sigma_{j}(B)}}}}\!\!\!\!\!\!\!\!\!\!\!\!\!\!\!\!(Y(e))\\
& =\quad \quad \text{{\LARGE Tr}}_{_{\displaystyle {}_{{\big\{Y(t)\big \}}_{t\in \mathrm{add}(\bigcup_{\!\!\!\!\!\!_{\displaystyle {}_{_{j<i}}}}\!\!\!\!\sigma_{j}(B))}}}}\!\!\!\!\!\!\!\!\!\!\!\!\!\!\!\!\!\!\!\!\!\!\!\!\!\!\!(Y(e))\\
& =\quad \quad \text{{\LARGE Tr}}_{_{\displaystyle {}_{{\big\{Y(t)\big \}}_{t\in \bigcup_{\!\!\!\!\!\!_{\displaystyle {}_{_{j<i}}}}\!\!\!\!B_{j}}}}}\!\!\!\!\!\!\!\!\!\!\!\!(Y(e)).
\end{split}
\]

Let $\R$ be locally finite. Then by Lemma \ref{LARp2} (a),  $\Tr_{\{Y(t)\}_{t\in{\bigcup_{j< i}\B_{j}}}}Y(e)$ is equal to $I_{\cup_{j< i}\B_{j}}(-,e).$ We assert that $\Delta_e(i)\neq 0.$ In order to prove this, it is enough to see that $\Delta_e(i)(e)\neq 0.$
\

Suppose that $\Delta_e(i)(e)=0.$ Then $I_{\cup_{j< i}\B_{j}}(e,e)=\R(e,e)$ and thus $1_e\in I_{\cup_{j< i}\B_{j}}(e,e).$ Therefore $1_e$ factorizes
through some $X\in\B_j,$ where $j<i.$ Then $e$ is a direct summand of $X$ and so $e\in\B_j,$ contradicting that $e\in\sigma_i(\B).$

(b)  Let $\R$ be  locally finite and right support finite.  By Lemma \ref{LARp2} (b),   the item (a) and \cite[Proposition 4.2 (d)]{Aus},  we get
${}_\B\Delta^\oplus\subseteq\finp_\rho(\R).$
\end{dem}

\begin{defi} Let $\B=\{\B_{i}\}_{i< \alpha}$ be an admissible family of subcategories of $\C\subseteq\R,$ for some Krull-Schmidt $\K$-ringoid $\R.$
Let  $\Delta={}_{(\B,\C)}\Delta$ be the $(\B,\C)$-standard family of right $\R$-modules. We say that $M\in \ltt{F}'_{f}(\Delta)$ if there exists a filtration
$0=M_{0}\subseteq M_{1}\subseteq \dots  \subseteq M_{n-1}\subseteq M_{n}=M$
 such that $M_{i}/M_{i-1}\in \Delta(s_i)^{\oplus},$ for some $s_i<\alpha$ and $i\in[1,n].$
\end{defi}

\begin{pro}\label{f=f'}  Let $\B=\{\B_{i}\}_{i<\alpha}$ be an admissible family of subcategories of $\C\subseteq\R,$ for some Krull-Schmidt $\K$-ringoid $\R,$  and let $\Delta={}_{(\B,\C)}\Delta.$ Then,
$\ltt{F}_{f}(\Delta)=\ltt{F}'_{f}(\Delta)$.
\end{pro}
\begin{proof} It can be shown that $\ltt{F}'_{f}(\Delta)$ is closed under extensions. Moreover, it is also clear
that $\ltt{F}'_{f}(\Delta)\subseteq \ltt{F}_{f}(\Delta).$
\

In order to prove that $\ltt{F}_{f}(\Delta)\subseteq \ltt{F}'_{f}(\Delta),$ we proceed by induction on the $\Delta$-length of objects in $\ltt{F}_{f}(\Delta).$ Let $M\in \ltt{F}_{f}(\Delta).$
\

 If $\ell_{\Delta}(M)=1$, then we have that
$$M = \bigoplus_{k=1}^{n}\bigoplus_{e\in J_{i_{k}}}
\Delta_{e}(i_{k})^{\mu_{e,k}}$$
where $J_{i_{k}}\subseteq \sigma_{i_{k}}(\mathcal{B})$  is finite.
Let $M(i_{k}):=\bigoplus_{e\in J_{i_{k}}}\Delta_{e}(i_{k})^{\mu_{e,k}}\in \Delta(i_{k})^{\oplus}$.
Then, the filtration
$$0\subseteq M(i_{1})\subseteq M(i_{1})\oplus M(i_{2})\subseteq \dots \subseteq \bigoplus_{s=1}^{n-1}M(i_{s})\subseteq
\bigoplus_{s=1}^{n}M(i_{s})=M$$
is the one required.
\

$\ell_{\Delta}(M)=m\geq 2$. Thus, we have a $\Delta$-filtration
$$\xi:\quad 0=M_{0}\subseteq M_{1}\subseteq M_{2}\subseteq \dots \subseteq M_{m-1}\subseteq M_{m}=M$$
and the exact sequence
$0\to M_{m-1}\to   M_{m}\to Z_{m}\to 0,$
where $\ell_{\Delta}(M_{m-1})=m-1$ and
$\ell_{\Delta}(Z_{m})=1.$ By inductive hypothesis, it follows that  $M_{m-1}\in \ltt{F}'_{f}(\Delta).$
Using the fact that $\ltt{F}'_{f}(\Delta)$ is closed under extensions, we conclude that
$M\in \ltt{F}'_{f}(\Delta).$
\end{proof}

\begin{lem}\label{claveD} Let $\R$ be a  locally finite $\K$-ringoid, and let $\B:=\{\B_{i}\}_{i< \alpha}$ be  an admissible family of subcategories of
$\R.$ Then, the family ${}_\B\Delta,$ of $\B$-standard modules, satisfies the following conditions.
\begin{itemize}
\item[(a)]  If ${}_\B\Delta\subseteq\finp_\rho(\R)$  then $\Delta_e(i)$ is local, for any $i<\alpha$ and $e\in\sigma_i(\B).$
\item[(b)] $\Hom(\Delta_e(i),\Delta_{e'}(i))\simeq\Delta_{e'}(i)(e),$ for any $e,e'\in \sigma_i(\B).$
\item[(c)] $\Hom(\Delta_e(i),\Delta_{e'}(i'))=0$ if $i<i'$ and $e\in \sigma_i(\B),$ $e'\in \sigma_{i'}(\B).$
\item[(d)] $\mathrm{Ext}^{1}(\Delta_{e}(i),\Delta_{e'}(i'))=0$ if $i\leq i'$ and $e\in \sigma_i(\B),$ $e'\in \sigma_{i'}(\B).$
\end{itemize}
\end{lem}
\begin{dem} Let $e\in\sigma_i(\B).$ By Proposition \ref{DSuppFin},  we have that $\Delta_e(i)=Y(e)/U_e(i),$ where
$U_e(i):=I_{\cup_{j< i}\B_{j}}(-,e)=\mathrm{Tr}_{\bigoplus_{a\in\cup_{j<i}\B_j}Y(a)}(Y(e)),$ and   $\Delta_e(i)\neq 0.$

(a) Let ${}_\B\Delta\subseteq\finp_\rho(\R).$  Since $e$ is local in $\R$ and $\finp_\rho(\R)$ is a Krull-Schmidt category (see Proposition \ref{LARp3}),  the
epimorphism  $Y(e)\rightarrow \Delta_e(i)$  is a projective cover. Let $\Delta_e(i)=\oplus_{k=1}^n M_k$ be a decomposition of $\Delta_e(i),$ where each
$M_k$ is local.  Consider the projective cover $P_k\to M_k,$ for $k\in[1,n].$ Using the fact that a finite coproduct of projective
covers is a projective cover, it follows that $Y(e)=\oplus_{k=1}^n P_k.$ Therefore $n=1$ since $Y(e)$ is local. Then we get that $\Delta_e(i)=M_1,$ proving that  $\Delta_e(i)$ is local.
\

(b), (c) and (d)
Let $i,i'\in[0, \alpha)$ and $e\in\sigma_i(\B), e'\in\sigma_{i'}(\B)$. Thus, we have the exact sequences of right $\R$-modules
\begin{eqnarray*}
0\rightarrow I_{\cup_{t<i}\mathcal{B}_t}(-,e)\rightarrow Y(e)\rightarrow \Delta_e(i)\rightarrow 0,\\
\bigoplus_{a\in\cup_{t<i}\mathcal{B}_t}\, Y(a)\rightarrow I_{\cup_{t<i}\mathcal{B}_t}(-,e)\rightarrow 0.
\end{eqnarray*}

Then, by applying $\Hom(-,\Delta_{e'}(i'))$  to the above exact sequences, we get  the exact sequence of abelian groups
\begin{center}
$0\rightarrow (\Delta_e(i), \Delta_{e'}(i'))\rightarrow(Y(e),\Delta_{e'}(i'))\to (I_{\cup_{t<i}\mathcal{B}_{t}}(-,e),\Delta_{e'}(i')),$
\end{center}
an epimorphism $\Hom( I_{\cup_{t<i}\mathcal{B}_{t}}(-,e),\Delta_{e'}(i'))\to \Ext^1(\Delta_e(i),\Delta_{e'}(i'))$ and a monomorphism 
$\Hom( I_{\cup_{t<i}\mathcal{B}_{t}}(-,e),\Delta_{e'}(i'))\to \Hom(\bigoplus_{a\in\cup_{t<i}\mathcal{B}_t}Y(a),\Delta_{e'}(i')).$

By Yonedas's Lemma, we have that
$$\Hom(\bigoplus_{a\in\cup_{t<i}\mathcal{B}_t}Y(a),\Delta_{e'}(i'))\simeq \prod_{a\in\cup_{t<i}\mathcal{B}_t}\Delta_{e'}(i')(a).$$
On the other hand, we know that
$$\Delta_{e'}(i')(a)= \frac{\R(a,e')}{I_{\cup_{t<i'}\mathcal{B}_t}(a,e')}.$$
\

Let $i\leq i'.$ Then, we get $\cup_{t<i}\mathcal{B}_{t}\subseteq\cup_{t<i'}\mathcal{B}_{t}$ and hence $\Delta_{e'}(i')(a)=0$ for any $a\in\cup_{t<i}\mathcal{B}_t.$ Therefore, we conclude that
$$\mathrm{Ext}^{1}(\Delta_{e}(i),\Delta_{e'}(i'))=0\;\text{and}\;(\Delta_e(i), \Delta_{e'}(i'))\simeq(Y(e),\Delta_{e'}(i'))\simeq\Delta_{e'}(i')(e).$$
If $i<i',$ it follows that $e\in\sigma_i(\B)\subseteq\B_i\subseteq\cup_{t<i'}\mathcal{B}_t$ and so
$\Delta_{e'}(i')(e)=\frac{\R(e,e')}{I_{\cup_{t<i'}\mathcal{B}_t}(e,e')}=0;$ proving that $(\Delta_e(i), \Delta_{e'}(i'))=0.$
\end{dem}

\begin{lem}\label{claveD1} Let $\R$ be a  locally finite $\K$-ringoid, and let $\B:=\{\B_{i}\}_{i< \alpha}$ be  an admissible family of subcategories of
$\R.$ Then, for the family ${}_{\B}\Delta$ of  $\B$-standard  right $\R$-modules and  any $i<\alpha,$ the following statements are equivalent.
\begin{itemize}
\item[(a)]  $\End(\Delta_e(i))$ is a division ring, for any  $e\in\sigma_i(\B).$
\item[(b)] $I_{\cup_{t<i}\mathcal{B}_{t}}(e,e)=\rad_{\R}(e,e),$ for any $e\in \sigma_i(\B).$
\item[(c)] $\End(\Delta_e(i))\simeq \End_\R(e)/\rad\,\End_\R(e),$ for any  $e\in\sigma_i(\B).$
\end{itemize}
\end{lem}
\begin{dem}  Let $e\in\sigma_i(\B).$ Then, by Lemma \ref{claveD} (b), it follows that
$$(*)\quad \End(\Delta_e(i))\simeq \End_{\R}(e)/I_{\cup_{t<i}\mathcal{B}_{t}}(e,e).$$
(a) $\Rightarrow$ (b) Assume that $\End(\Delta_e(i))$ is a division ring.  Let $f\in I_{\cup_{t<i}\mathcal{B}_{t}}(e,e).$ Then, there are morphisms $e\xrightarrow{v} b\xrightarrow{u}e,$ with $b\in \cup_{t<i}\mathcal{B}_{t}$ and such that $f=uv.$ Since $b\in\cup_{t<i}\mathcal{B}_{t},$ we get that $f$ is not an isomorphism and hence
$f\in\rad_{\R}(e,e).$
\

Let $f\in\rad_{\R}(e,e).$ Suppose that $f\not\in I_{\cup_{t<i}\mathcal{B}_{t}}(e,e).$ Then, by $(*),$ the class $\overline{f}=f+ I_{\cup_{t<i}\mathcal{B}_{t}}(e,e)$ is invertible in $\End(\Delta_e(i))$  and  there is $g:e\to e$ such that $fg-1_e\in I_{\cup_{t<i}\mathcal{B}_{t}}(e,e).$ Note that
$fg\in\rad_{\R}(e,e)$ and thus $fg-1_e$ is invertible in $\End_\R(e).$ As a consequence, $1_e\in I_{\cup_{t<i}\mathcal{B}_{t}}(e,e)$ and so
$e\in \cup_{t<i}\mathcal{B}_{t},$ which is a contradiction. Therefore $f\in I_{\cup_{t<i}\mathcal{B}_{t}}(e,e).$
\

The implications (b) $\Rightarrow$ (c) $\Rightarrow$ (a) follow from $(*)$ and the fact that $e$ is local in $\R.$
\end{dem}

\begin{lem}\label{claveD2}  Let $\R$ be a  locally finite $\K$-ringoid, and let $\B:=\{\B_{i}\}_{i< \alpha}$ be  an admissible family of subcategories of
$\R.$ Then, for the family ${}_{\B}\Delta$ of  $\B$-standard  right $\R$-modules and  any $i<\alpha,$ the following statements are equivalent.
\begin{itemize}
\item[(a)]  $\Hom(\Delta_e(i),\Delta_{e'}(i))=0,$  for any  $e\neq e'$ in $\sigma_i(\B).$
\item[(b)] $I_{\cup_{t<i}\mathcal{B}_{t}}(e,e')=\rad_{\R}(e,e'),$ for any $e\neq e'$ in $\sigma_i(\B).$
\end{itemize}
\end{lem}
\begin{dem} Let $e\neq e'$ in $\sigma_i(\B).$ By Lemma \ref{claveD} (b) and Proposition \ref{DSuppFin}, we get the exact sequence
$$ 0\to I_{\cup_{t<i}\mathcal{B}_{t}}(e,e') \to \R(e,e')\to \Hom(\Delta_e(i),\Delta_{e'}(i))\to 0.$$
On the other hand, note that $\rad_{\R}(e,e')=\R(e,e'),$ since $e$ and $e'$ are not isomorphic local objects. Therefore, the result follows.
\end{dem}

\begin{lem}\label{interc} Let $\R$ be a  locally finite $\K$-ringoid,  $\B:=\{\B_{i}\}_{i< \alpha}$ be  an admissible family of subcategories of
$\R,$  and let $\Delta={}_\B\Delta.$ Then, the following statements hold true.
\begin{itemize}
\item[(a)] Let  $ L\subseteq M \subseteq  N$ be a chain of right $\R$-submodules,
with $M/L\in \Delta(i')^{\oplus},$  $N/M\in \Delta(i)^{\oplus}$ and
$i<i'$. Then, there exists a chain of right $\R$-submodules
$L\subseteq M'\subseteq N$ such that $M'/L\simeq N/M\in \Delta(i)^{\oplus}$ and $N/M'\simeq M/L\in \Delta(i')^{\oplus}.$
\item[(b)] Let $\{\eta_{i}:\;0\to M_{i-1}\to M_{i}\to  X_{i}\to 0\}_{i=1}^{n}$ be a family of exact sequences
in $\Mod_\rho(\R),$
where $X_{i}\in \Delta(j)^{\oplus},$ for every $i\in[1,n]$ and some $j<\alpha.$  Then,  for each $k\in [1,n],$
there exists an exact sequence of the form
$\xi_{k}:\; 0\to M_{0}\to M_{k}\to  Z_{k}\to 0,$
where $Z_{k}=\oplus_{i=1}^k\, X_{i} \in \Delta(j)^{\oplus}.$
\end{itemize}
\end{lem}
\begin{proof} (a) From the chain of submodules $ L\subseteq M \subseteq  N,$ we construct the following exact and commutative diagram
$$\xymatrix{ & 0\ar[d] & 0\ar[d]\\
& L\ar@{=}[r]\ar[d]^{i} & L\ar[d]^{i'i}\\
0\ar[r] & M\ar[r]^{i'}\ar[d]^{d} & N\ar[r]^{d'}\ar[d]^{\beta} & \frac{N}{M}\ar@{=}[d]\ar[r] & 0 \\
0\ar[r] & \frac{M}{L}\ar[r]^{\alpha_1}\ar[d] & A\ar[r]^{\alpha_2}\ar[d] & \frac{N}{M}\ar[r] & 0\\
& 0 & 0.}$$
By Lemma \ref{claveD} (d), the bottom exact sequence, in the above diagram, splits. Thus, we have the
exact sequence  $\xi:0\sta{}{\fl} \frac{N}{M}\sta{\beta_2}{\fl}
A\sta{\beta_1}{\fl} \frac{M}{L}\sta{}{\fl} 0.$ Then, we get the exact and commutative diagram
$$\xymatrix{& & 0\ar[d] & 0\ar[d]\\
0\ar[r] & L\ar[r]\ar@{=}[d] & M'\ar[r]\ar[d] &
\frac{N}{M}\ar[d]^{\beta_{2}}\ar[r] & 0\\
0\ar[r] & L\ar[r]^{i'i} & N\ar[r]^{\beta}\ar[d]^{\pi} & A\ar[d]^{\beta_{1}}\ar[r] & 0\\
& & \frac{M}{L}\ar@{=}[r]\ar[d] & \frac{M}{L}\ar[d]\\
& & 0 & 0.}$$
Finally, we conclude that
$L\subseteq M'\subseteq N$ and so
$\frac{M'}{L}\simeq \frac{N}{M}\in \Delta(i)^{\oplus}$ and
$\frac{N}{M'}\simeq \frac{M}{L}\in \Delta(i')^{\oplus}$.
\

(b) We proceed by induction on $k$. If $k=1,$ we set $\xi_{1}:=\eta_{1}$.\\
Let $k\geq 2.$ Then, by induction, we have defined $\xi_{k-1}$ satisfying  (b). We construct the following exact and commutative diagram
$$\xymatrix{& 0\ar[d] & 0\ar[d]\\
& M_{0}\ar@{=}[r]\ar[d] & M_{0}\ar[d]\\
0\ar[r] & M_{k-1}\ar[r]\ar[d] & M_{k}\ar[r]\ar[d] & X_{k}\ar[r]\ar@{=}[d] & 0\\
0\ar[r] & \oplus_{s=1}^{k-1}X_{s}\ar[r]\ar[d]  & L_{k}\ar[r]\ar[d] & X_{k}\ar[r] & 0\\
& 0 & 0}$$
By Lemma \ref{claveD} (d), we have that the buttom exact sequence splits. Then, we
have that $L_{k}\simeq \oplus_{s=1}^{k}X_{k}$. Therefore, the second column in the above diagramm is the required exact sequence.
\end{proof}

In the following definition, we use that $\ltt{F}_{f}(\Delta)=\ltt{F}'_{f}(\Delta),$ see Proposition \ref{f=f'}.
\begin{defi} Let $\R$ be a  locally finite $\K$-ringoid,  $\B:=\{\B_{i}\}_{i<\alpha}$ be  an admissible family of subcategories of
$\R,$  and let $\Delta={}_\B\Delta.$
For $M\in \ltt{F}_{f}(\Delta),$ we consider a filtration
$$\xi:\quad 0=M_{0}\subseteq M_{1}\subseteq M_{2}\subseteq \dots \subseteq M_{m-1}\subseteq M_{m}=M,$$
where $X_{k}:=M_{k}/M_{k-1}\in \Delta(i_{k})^{\oplus}.$ In this case, we have the set
$$\Phi_\xi(i):=\{k\in[1,m]\;|\; 0\neq X_{k}\in \Delta(i)^{\oplus}\}.$$
\begin{enumerate}
\item [(a)] The {\bf $\xi$-ladder filtration multiplicity} $[M:\Delta(i)]_\xi,$ of $\Delta(i)$ in $M,$  is
the cardinality of $\Phi_\xi(i).$  In general, the ladder filtration multiplicity $[M:\Delta(i)]_\xi$ could be depending on $\xi.$
\item [(b)] We define, the $\xi$-{\bf ladder} $\Delta$-length of $M$
$$\ell'_{\Delta,\xi}(M)=\sum_{\Delta(i)\in \Delta}[M:\Delta(i)]_\xi.$$
Observe that this sum is finite, since
only a finite number of $\Delta(i)$ appears in $\xi.$
\item [(c)] For $i<\alpha$ and  $k\in \Phi_\xi(i),$ we consider a decomposition
$$D_{k,i}(\xi):\qquad X_{k}=\bigoplus_{e\in J_{k}}\Delta_{e}(i)^{\mu_{e,k}}$$
of each $X_k,$
where $J_{k}\subseteq\sigma_{i}(\ltt{B})$ is finite. Let $D_{i}(\xi):=\{D_{k,i}(\xi)\}_{k\in\Phi_\xi(i)}$ be called the family of decompositions
 associated with the set $\Phi_\xi(i).$
We define  the
{\bf $\xi$-filtration multiplicity}  of $\Delta_{e}(i)$ in $M$ as follows:
\begin{displaymath}
[M:\Delta_{e}(i)]_{\xi,D_{i}(\xi)}:=\left\{\begin{array}{ll}
0 & \textrm{if $[M:\Delta(i)]_{\xi}=0,$}\\
{}\\
\sum_{k\in\Phi_\xi(i)}\,\mu_{e,k} & \textrm{if $[M:\Delta(i)]_{\xi}\neq 0.$}
\end{array} \right.
\end{displaymath}
\end{enumerate}
\end{defi}

\begin{rk}\label{inddesc} Note that $[M:\Delta_{e}(i)]_{\xi,D_{i}(\xi)}$ depends not only on $\xi$ but also on the chosen family $D_{i}(\xi)$ of decompositions
 associated with the set $\Phi_\xi(i).$ However, if $\Delta^{\oplus}\subseteq\finp_\rho(\Delta),$ then $[M:\Delta_{e}(i)]_{\xi,D_{i}(\xi)}$ does not depend on $D_{i}(\xi),$ since  by Lemma \ref{claveD} (a)  all the $\Delta_{e}(i)$ are local objects.
\end{rk}

\begin{pro}\label{filtracionordenada}  Let $\R$ be a  locally finite $\K$-ringoid,  $\B:=\{\B_{i}\}_{i< \alpha}$ be  an admissible family of subcategories of
$\R,$  $\Delta={}_\B\Delta$ and $M\in \ltt{F}_{f}(\Delta).$ Consider  a finite $\Delta$-filtration $\xi$ of $M$
$$\xi:\quad 0=M_{0}\subseteq M_{1}\subseteq M_{2}\subseteq \dots \subseteq M_{m-1}\subseteq M_{m}=M,$$
such that $M_k/M_{k-1}\in\Delta(j_k)^\oplus,$ and fix a family $D_{i}(\xi)$ of decompositions
associated with the set $\Phi_\xi(i)$ for each $i.$
\

Then,  there exist $\Delta$-filtrations
$\eta$ and $\varepsilon$ of $M,$   and decompositions $D_i(\eta),$ $D_i(\varepsilon),$ for each $i,$   satisfying the following conditions:
\begin{enumerate}
\item [(a)] $[M:\Delta_e(i)]_{\xi,D_i(\xi)}=[M:\Delta_e(i)]_{\eta,D_i(\eta)},$ for any $e\in\sigma_i(\B).$

\item [(b)] The filtration $\eta$ is well ordered. That is, there is a family of exact sequences
$$\eta=\{\eta_{b}:\xymatrix{0\ar[r] &  \overline{M}_{b-1}\ar[r] &  \overline{M}_{b}\ar[r] & \overline{X}_{b}\ar[r] & 0}\}_{b=1}^{m}$$ with
$\overline{M}_{0}:=0,$  $i_{1}\leq i_{2}\leq \cdots
\leq i_{m}$  and $\overline{X}_{b}\in \Delta(i_{b})^{\oplus}$.

\item [(c)] If $M\neq 0,$ the filtration $\varepsilon$ is strictly well ordered. That is, $\varepsilon$ has the form
$\varepsilon:\; 0=M_{0}'\subsetneq M_{1}' \subsetneq M_{2}' \subsetneq \dots  \subsetneq M_{a-1}'\subsetneq M_{a}'=M$
where $M_{k}'/M_{k-1}'\in \Delta(i'_{k})^{\oplus},$ for $k\in[1,a],$ $a\leq m$ and
$ i'_{1}<i'_{2}< i'_{3} < \dots <i'_{a-1}< i'_{a}.$ Moreover 
\begin{center}
$[M:\Delta_e(i)]_{\varepsilon,D_i(\varepsilon)}=[M:\Delta_e(i)]_{\eta,D_i(\eta)},$ for any $e\in\sigma_i(\B).$
\end{center}
\end{enumerate}
\end{pro}
\begin{proof} If $M=0,$ we have that (a) and (b) are trivial. Let  $M\neq 0.$
\

Let $\xi$ be the given filtration of $M.$ We may assume that
$$\xi:\quad 0=M_{0}\subsetneq M_{1}\subsetneq M_{2}\subsetneq \dots \subsetneq M_{m-1}\subsetneq M_{m}=M,$$
where $X_{k}:=M_{k}/M_{k-1}\in \Delta(i_{k})^{\oplus}.$
\

We prove (a) and (b), by induction on the $\xi$-ladder length
$n:=\ell'_{\Delta,\xi}(M).$ If $n=1$, the filtration $\xi$ is already well ordered and hence
$\eta:=\xi$ and $\varepsilon:=\xi$ satisfy the required properties.
\

Let $n\geq 2.$ Consider the family of exact sequences induced
by the filtration $\xi$ of $M$
$$\{\xi_{b}:\xymatrix{0\ar[r] & M_{b-1}\ar[r] & M_{b}\ar[r] &  X_{b}\ar[r] & 0}\}_{b=1}^{m}.$$
Since $\xi':=\xi-\{\xi_{m}\}$ is a filtration of $M_{m-1}$  and
$\ell'_{\Delta,\xi'}(M_{m-1})=m-1$, by induction, there is a well ordered
filtration
$$\eta'=\{\eta'_{b}:\xymatrix{0\ar[r] & M'_{b-1}\ar[r] &  M'_{b}\ar[r] &
Y_{b}\ar[r] & 0 }\}_{b=1}^{m-1}$$ of $M_{m-1}$ with
$i'_{1}\leq i'_{2}\leq \cdots \leq i'_{m-1}$ and
$[M_{m-1}:\Delta_e(i)]_{\xi',D_i(\xi')}=[M_{m-1}:\Delta_e(i)]_{\eta',D_i(\eta')},$ for any $e\in\sigma_i(\B).$
If $i_{m-1}'\leq j_{m}$ then $\eta:=\eta'\cup\{\xi_{m}\}$ satisfies the required conditions.\\
Suppose now that $j_{m}<i_{m-1}'.$ Let $l:=\mathrm{max}\{n\in [1,m-1]\;|\;
j_{m}<i_{m-n}'\}$. Observe that the filtration
$\eta'\cup\{\xi_{m}\}$ is almost the one we want, the only exact sequence
that is not ordered is precisely the
$\xi_{m}$. This can be rearranged by applying $l$-times
Lemma \ref{interc} (a) to $\eta'\cup \{\xi_{m}\}$.
\

In order to construct $\varepsilon,$ we use the well ordered filtration $\eta$ from (b).
 We proceed as follows. For each $b,$ we group the $i_{b}$
that are the same and rename them by $\lambda_{a}$. So we get $ \lambda_{1}< \lambda_{2}< \dots < \lambda_{a}$
and hence $\Delta(\lambda_{1}), \cdots, \Delta(\lambda_{a})$ are the different $\Delta(j)$ appearing
in the filtration $\eta$  of $M.$ Define
$s(i):=[M:\Delta(\lambda_{i})]_{\eta},$ $\alpha(i):=\sum_{j=1}^{i}s(j)$ and $\alpha(0):=0$.\\
We divide the filtration $\eta$ into the following pieces
$$\{\eta_{b}:\xymatrix{0\ar[r]  & M_{b-1}\ar[r] &
M_{b}\ar[r] & Y_{b}\ar[r] & 0}\}_{b=\alpha(l-1)+1}^{\alpha(l)},$$ with $l\in[1,a]$. For each $l\in[1,a],$ by Lemma \ref{interc} (b),  we obtain the following exact sequence
$$\varepsilon_{l}:\quad \xymatrix{0\ar[r] & M_{\alpha(l-1)}\ar[r] & M_{\alpha(l)}\ar[r] &
Z_{\alpha(l)}\ar[r] & 0}$$ Hence, by setting
$M_{0}'=0$ and  $M_{i}':=M_{\alpha(i)}$ for $i\in[1, a],$ we conclude that the filtration
$\varepsilon=\{\varepsilon_{i}\}_{i=1}^{a}$ satisfies the required properties. Finally, we bring out that, in the construction of $\eta$  and $\varepsilon,$ we have not added different factors as appearing in $\xi.$ These factors have  just been  reordered and regrouped to obtain  $\eta$  and $\varepsilon.$
\end{proof}

\section{Filtration multiplicities in ringoids}

Let $\mathcal{A}$ be an abelian category.  It is well known that a {\bf pre-radical} $\mathfrak{r}$ of $\mathcal{A}$ is a
 subfunctor of the identity functor $1_\A:\mathcal{A}\longrightarrow \mathcal{A}$. A pre-radical $\mathfrak{r}$ of $\mathcal{A}$ is
 additive if it is an additive functor.
 \

Let $\mathcal{A}$ be an abelian category with arbitrary coproducts. Given a set $\ltt{X}$ of objects in $\mathcal{A}$ and $M\in \mathcal{A},$ we recall that the trace
of $M,$ with respect to $\ltt{X},$ is
$\mathrm{Tr}_{\ltt{X}}(M):=\!\!\!\!\!\!
\!\!\!\!\!\!\displaystyle\sum _{\{f\in\mathrm{Hom}(X, M)\;\mid\; X\in \ltt{X}\}}
\!\!\!\!\!\!\!\!\!\!\!\!\!\!\!\!\!\!\mathrm{Im}(f).$ Note that, for any morphism $f:A\to B$  in $\mathcal{A},$ we have that $f(\mathrm{Tr}_{\ltt{X}}(A))\subseteq \mathrm{Tr}_{\ltt{X}}(B).$ Thus,  a pre-radical $\tau_{\ltt{X}}$ of
$\mathcal{A}$ can be defined as follows: $\tau_{\ltt{X}}(Z):=\mathrm{Tr}_{\ltt{X}}(Z)$ for any $Z\in\A,$  and $\tau_{\ltt{X}}(f)
:=f|_{\tau_{\ltt{X}}(A)}:\tau_{\ltt{X}}(A)\to \tau_{\ltt{X}}(B) ,$ for any morphism $f:A\longrightarrow B$  in $\mathcal{A}$. Note that,
the pre-radical $\tau_{\ltt{X}}$ is additive. In case $\X$ has just one element, say $X,$ we write $\tau_X$ instead of $\tau_\X.$
\

\begin{lem}\label{SubfTr}  Let $\mathcal{A}$ be an abelian category with arbitrary coproducts, and let $M=N\oplus N'$ be a decomposition of $M \in \A.$ Then,
$\tau_N\circ\tau_M=\tau_N$ and thus
$\tau_N$ is a subfunctor of $\tau_M.$
\end{lem}
\begin{dem} Let $X\in\A.$ Then $\Tr_M(X)\subseteq X$ and hence $\Tr_N(\Tr_M(X))\subseteq \Tr_N(X).$
\

Let $g\in\Hom_\A(N,X).$ Consider the factorization $N\xrightarrow{g'}\imagen\,(g)\to X$ of $g$ through its image.  Define the matrix morphism
$ f:=\begin{pmatrix}g & 0 \end{pmatrix}:M\to X.$ Note that $\imagen(f)=\imagen(g)\subseteq\Tr_M(X).$ Let $\imagen\,(f)\xrightarrow{j}\Tr_M(X)$ be the natural
inclusion. Then, for the composition  $N\xrightarrow{g'}\imagen\,(f)\xrightarrow{j}\Tr_M(X),$ we have
$$\imagen\,(g)=\imagen\,(g')=\imagen\,(j\circ g')\subseteq \Tr_N(\Tr_M(X)). $$
Therefore $\Tr_N(\Tr_M(X))=\Tr_N(X),$ proving the result.
\end{dem}

\

In what follows, we consider the abelian category $\mathcal{A}:=\Mod_\rho(\R),$  where
$\R$ is a $\K$-ringoid. Note that $\A$ has arbitrary coproducts and then $\tau_{\ltt{X}}$ is well defined, for any set $\X$ of objects in $\A.$ We recall that $\finp_\rho(\R)$ denotes the category of finitely presented right $\R$-modules.

\begin{defi} Let $\R$ be a  locally finite $\K$-ringoid,  $\B:=\{\B_{i}\}_{i< \alpha}$ be  an admissible family of subcategories of $\R.$ For each  $i<\alpha,$
we consider the additive pre-radicals
$$\tau_{i}(-):=\mathrm{Tr}_{\bigoplus_{j\leq i}\overline{P}(j)}(-)\quad\text{and}\quad
\overline{\tau_{i}}(-):=\mathrm{Tr}_{\bigoplus_{j< i}\overline{P}(j)}(-),$$
 where $P^{op}=\{P^{op}(i)\}_{i<\alpha}$ is the family of projective right $\R$-modules associated with the partition $\sigma(\B),$ and
$\overline{P}(j):=\bigoplus_{e\in\sigma_j(\B)}\,P^{op}_e(j).$
\

 Let $M$ be a right $\R$-module. The $i$-th $\B$-trace of $M$ is  $\tau_i(M)$ and  $\tau_{\B,M}:=\{\tau_i(M)\}_{i<\alpha}$ is the
 $\B$-{\bf trace filtration}of $M,$ which is a chain of submodules of $M.$
\end{defi}

\begin{lem}\label{SubfTr2} Let $\R$ be a  locally finite $\K$-ringoid, and let $\B:=\{\B_{i}\}_{i< \alpha}$ be  an admissible family of subcategories of $\R.$
Then, for any $i<\alpha,$ the
following statements hold true.
\begin{itemize}
\item[(a)] $\overline{\tau_{i}}$ is a subfunctor of $\tau_{i}$ and
$\overline{\tau_{i}}\circ\tau_i=\overline{\tau_{i}}.$
\item[(b)]  $\overline{\tau_{i}}=\sum_{j<i}\,\tau_j.$
\item[(c)]  $\tau_j(\tau_i(M))=\tau_k(M)$ for $k:=\min\{i,j\}.$
\end{itemize}
\end{lem}
\begin{dem} (a) follows from Lemma \ref{SubfTr}. To prove (b), let us consider  $X\in \Mod_\rho(\R).$ Then, we have the following sequence of equalities
\begin{align*}
\sum_{j<i}\,\tau_j(X) & =\sum_{j<i}\,\Tr_{\bigoplus_{k\leq j}\,\overline{P}(i)}\,(M)\\
                                 & = \Tr_{\bigoplus_{j<i}\big (\,\bigoplus_{k\leq j}\,\overline{P}(i)\big)}\,(X)\\
                                 & = \overline{\tau_{i}}(X).
\end{align*}
Finally, for the proof of (c). Note that $\tau_i(M)\subseteq M$ and thus $\tau_j(\tau_i(M))\subseteq \tau_j(M).$  Let
$j<i.$ Then $\tau_j(M)\subseteq\tau_i(M)$ and therefore
$\tau^2_j(M)=\tau_j(M)\subseteq\tau_j(\tau_i(M)).$ Hence, we conclude that
$\tau_j(M)=\tau_j(\tau_i(M))$ for every $j<i$. Similarly for $j\geq i,$ we can show that $\tau_j(\tau_i(M))=\tau_i(M).$
\end{dem}

\begin{lem}\label{ImagenN}  Let $\R$ be a  locally finite $\K$-ringoid,  $\B:=\{\B_{i}\}_{i<\alpha}$ be an admissible family of subcategories of $\R,$
 $\Delta={}_\B\Delta$ and let $0\to N\xrightarrow{\alpha} M\xrightarrow{\beta}   E\to  0$ be an exact sequence
with $E\in \Delta(i)^{\oplus}$ and $j<i$. Then, for every $f\in \Hom(\overline{P}(j),M),$ we
have that $\imagen\,(f)\subseteq N$.
\end{lem}
\begin{proof} We assert that $\mathrm{Hom}(\overline{P}(j),E)=0.$ Indeed, let
 $E=\bigoplus_{e \in J_{i}}\Delta_{e}(i)^{d_{e}},$ where $J_{i}\subseteq \sigma_i(\B)$ is a finite subset.
Since $$\Hom(\overline{P}(j),E)=\prod_{e\in J_{i}}\Hom (\overline{P}(j),\Delta_{e}(i))^{d_{e}},$$ it is enough to see that
$\Hom(\overline{P}(j),\Delta_{e}(i))=0$ for every $e\in J_{i}$.\\
Let $f:\overline{P}(j)\fl \Delta_{e}(i)$ be a morphism. Note that $\overline{P}(j)$ is projective, and thus,  there exists a morphism
$g:\overline{P}(j)\fl P^{op}_{e}(i)$ such that the following diagram commutes
$$\xymatrix{ & &  & \overline{P}(j)\ar[dl]_{g}\ar[d]^{f}\\
0\ar[r] & U_{e}(i)\ar[r]^(0.5){\gamma_{e}(i)} & P^{op}_{e}(i)\ar[r]^(0.5){\delta_{e}(i)} &
\Delta_{e}(i)\ar[r] &  0,}$$
where $U_{e}(i):=\mathrm{Tr}_{\oplus_{j<i}\overline{P}(j)}(P^{op}_{e}(i))$. Then, there is
$g':\overline{P}(j)\fl U_{e}(i)$ such that $g=\gamma_{e}(i)g';$ and thus, we have that
$f=\delta_{e}(i)g=\delta_{e}(i)\gamma_{e}(i)g'=0$. Proving that $\Hom(\overline{P}(j),\Delta_{e}(i))=0$.
\

Let $f:\overline{P}(j)\fl M$ be a morphism.  Hence $\beta f\in \Hom(\overline{P}(j),E)=0$ and therefore  $\mathrm{Im}\,(f)\subseteq N$.
\end{proof}

\begin{pro} \label{imagenbaja} Let $\R$ be a  locally finite $\K$-ringoid,  $\B:=\{\B_{i}\}_{i< \alpha}$  be an admissible family of subcategories of $\R,$
 $\Delta={}_\B\Delta,$ and $0\neq M\in \ltt{F}_{f}(\Delta).$ Consider a strictly well ordered filtration
$$\xi:\quad 0=M_{0}\subsetneq M_{1}\subsetneq  M_{2}\subsetneq \dots \subsetneq M_{a-1}\subsetneq M_{a}=M,$$
where $M_{k}/M_{k-1}\in \Delta(i_{k})^{\oplus}$ and $ i_{1}<i_{2} <\dots < i_{a-1}<i_{a}.$ Then, for any morphism
 $f:\overline{P}(j)\fl M,$  with
$j\leq i_{k}$ and $k\in [1,a],$  we have that $\mathrm{Im}\,(f)\subseteq N$ where
\begin{displaymath}
N=\left\{\begin{array}{ll}
M_{k}& \textrm{if $j=i_{k}, $}\\
{}\\
M_{k-1} & \textrm{if $j<i_{k}.$}
\end{array} \right.
\end{displaymath}
\end{pro}
\begin{proof}
Let $f:\overline{P}(j)\fl M$ with $j\leq i_{k} $ and $k\in [1,a].$  We consider
the following diagram
$$\xymatrix{ & & \overline{P}(j)\ar[d] & \\
0\ar[r] & M_{a-1}\ar[r]^{u_{a}} & M_{a}\ar[r]^{\pi_{a}} & X_{a}\ar[r] & 0,}$$
where $X_{a}\in \Delta(i_{a})^{\oplus}.$
Since $j\leq i_{k}<i_{a}$, by Lemma \ref{ImagenN}, there
exists a morphism $v_{a}:\overline{P}(j)\fl M_{a-1}$ such that $f=u_{a}v_{a}$.\\
Now, consider the diagram
$$\xymatrix{ & & \overline{P}(j)\ar[d]^{v_{a}} & \\
0\ar[r] & M_{a-2}\ar[r]^{u_{a-1}} & M_{a-1}\ar[r]^{\pi_{a-1}} & X_{a-1}\ar[r] & 0,}$$
where $X_{a-1}\in \Delta(i_{a-1})^{\oplus}.$
Since $j\leq i_{k}< i_{a-1}< i_{a}$, by Lemma \ref{ImagenN} there
exists a morphism $v_{a-1}:\overline{P}(j)\fl M_{a-2}$ such that $v_{a}=u_{a-1}v_{a-1}$. By iterating the same argument, we get the following diagram
$$\xymatrix{ & & \overline{P}(j)\ar[d]^{v_{k+2}} & \\
0\ar[r] & M_{k}\ar[r]^(0.6){u_{k+1}} & M_{k+1}\ar[r]^{\pi_{k+1}} & X_{k+1}\ar[r] & 0,}$$
where $X_{k+1}\in \Delta(i_{k+1})^{\oplus}$.
Since $j\leq i_{k}<i_{k+1}$, by Lemma \ref{ImagenN}, there exists $v_{k+1}:\overline{P}(j)\fl M_{k}$ such that
$v_{k+2}=u_{k+1}v_{k+1}$.
Then, by taking $\overline{f}:=v_{k+1},$ we have that $f=u_{a}u_{a-1}\dots u_{k+1}\overline{f}$.
Therefore $\mathrm{Im}(f)\subseteq M_{k}.$\\
Now if $j<i_{k},$ we consider the diagram
$$\xymatrix{ & & \overline{P}(j)\ar[d]^{v_{k+1}} & \\
0\ar[r] & M_{k-1}\ar[r]^(0.6){u_{k}} & M_{k}\ar[r]^{\pi_{k}} & X_{k}\ar[r] & 0,}$$
where $X_{k}\in \Delta(i_{k})^{\oplus}$.
Since $j<i_{k}$, there exists  $v_{k}:\overline{P}(j)\fl M_{k-1}$ such that  $v_{k+1}=u_{k}v_{k}$.
Then, by taking $\overline{f}:=v_{k}$ we have that $f=u_{a}u_{a-1}\dots u_{k}\overline{f}$.
Therefore $\mathrm{Im}(f)\subseteq M_{k-1}$.
\end{proof}

\begin{defi} Let $\R$ be a  locally finite $\K$-ringoid,  and let $\B:=\{\B_{i}\}_{i<\alpha}$  be an admissible family of subcategories of $\R.$ For any
$M\in\Mod_\rho(\R),$ the
{\bf $i$-th $\tau$-section} of $M$ is the quotient $\tau_i/\overline{\tau}_i(M).$ The {\bf support} of the $\B$-trace filtration of $M$ is the set
$$\Supp(\tau_{\B,M}):=\{i<\alpha\;:\; \tau_i/\overline{\tau}_i(M)\neq 0\}.$$
\end{defi}

\begin{teo}\label{filtraza} Let $\R$ be a  locally finite $\K$-ringoid,  $\B:=\{\B_{i}\}_{i<\alpha}$ be  an admissible family of subcategories of $\R,$
 $\Delta={}_\B\Delta,$ and $M\in\Mod_\rho(\R).$ Then, the following statements are equivalent.
 \begin{itemize}
 \item[(a)] $M$ has a finite $\Delta$-filtration.
 \item[(b)] There exist some  $i_0<\alpha$ such that $\tau_j(M)=M$ for any $j\geq i_0,$ $\Supp(\tau_{\B,M})$ is finite and
 $\tau_i/\overline{\tau}_i(M)\in\Delta(i)^{\oplus},$ for any $i<\alpha.$

 \end{itemize}
\end{teo}
\begin{dem}
(a) $\Rightarrow$ (b)
Let $0\neq M\in \ltt{F}_{f}(\Delta).$ Consider a strictly well ordered filtration
$$\xi:\quad 0=M_{0}\subsetneq M_{1}\subsetneq  M_{2}\subsetneq \dots \subsetneq M_{a-1}\subsetneq M_{a}=M,$$
where $M_{k}/M_{k-1}\in \Delta(i_{k})^{\oplus}$ and $ i_{1}<i_{2} <\dots < i_{a-1}<i_{a}.$
We have the following filtration $\Omega,$ which is composed of the following pieces

\begin{align*}
\Omega_{0}:\quad  &\quad 0=N_{0}=N_{1}= \cdots =N_{i'_{1}}= M_{0}\quad\forall\,i'_1\in[0,i_1),\\
\Omega_{1}:\quad  & \subsetneq N_{i_{1}}=N_{i_{1}+1}= \cdots =N_{i'_{2}}=M_1\quad\forall\,i'_2\in[i_1,i_2),\\
\Omega_{2}:\quad &  \subsetneq N_{i_{2}}=N_{i_{2}+1}=\cdots =N_{i'_{3}}=M_2\quad\forall\,i'_3\in[i_2,i_3),\\
 {} \quad & \quad \quad\quad\quad\quad   \dots \dots \dots \dots {} \\
\Omega_{a-2}:\quad & \subsetneq N_{i_{a-2}}=N_{i_{a-2}+1}=\cdots =N_{i'_{a-1}}=M_{a-2}\quad\forall\,i'_{a-1}\in[i_{a-2},i_{a-1}),\\
\Omega_{a-1}:\quad  & \subsetneq N_{i_{a-1}}=N_{i_{a-1}+1}=\cdots =N_{i'_a}=M_{a-1}\quad\forall\,i'_a\in[i_{a-1},i_a),\\
\Omega _{a}:\quad  & \subsetneq N_{i_{a}} := M_a.
\end{align*}
In order to prove the result, it is enough to see that $\Omega=\tau_{\B,M}$ and $\tau_j(M)=M$ for any $j\geq i_a.$
\

Note that, for $j>i_a,$ we have $\tau_j(M)=\tau_{i_a}(M)+\mathrm{Tr}_{\bigoplus_{i_a<k\leq j}\overline{P}(k)}(M).$ Thus, we only need to check that  $\tau_{i_a}(M)=M$ and $N_{i}=\tau_i(M),$ for all $0\leq i\leq i_{a}.$ To perform that, we will follow a series of steps as follows.
\

(i) $\tau_{i'_1}(M)=M_0=0\quad\forall\,i'_1\in[0,i_1).$
\

Indeed, since $\tau_0(M)\subseteq  \tau_{i'_{1}}(M),$ it is enough to see that
$\tau_{i'_{1}}(M)=0.$ Let $f:\overline{P}(j)\fl M$ with $j\leq i'_1<i_{1}$.
By Proposition \ref{imagenbaja}, it follows that $\mathrm{Im}\,(f)\subseteq M_{0}=0,$ proving that $\tau_{i'_{1}}(M)=\mathrm{Tr}_{\bigoplus_{ j\leq i'_1}\overline{P}(j)}\,(M)=0.$
\

(ii) $\tau_{i'_2}(M)=M_{1}\in \Delta(i_{1})^{\oplus}\quad\forall\,i'_2\in[i_1,i_2).$
\

 Firstly, we assert that $\tau_{i_{1}}(M)=M_1.$ Indeed, note that $M_{1}\in \Delta(i_1)^{\oplus},$ since $M_{1}/M_{0}\in \Delta(i_1)^{\oplus}$. Thus
$M_{1}=\bigoplus_{e\in J_{i_{1}}}\Delta_{e}(i_{1})^{\mu_{e,1}}$ for some finite subset
 $J_{i_{1}}\subseteq \sigma_{i_1}(\B)$.
Observe now, that there exists an epimorphism
$$\bigoplus_{e\in J_{i_{1}}}P^{op}_{e}(i_{1})^{\mu_{e,1}}\fl  \bigoplus_{e\in J_{i_{1}}}\Delta_{e}(i_{1})^{\mu_{e,1}}
=M_{1}\subseteq M,$$
and therefore  $M_{1}\subseteq \mathrm{Tr}_{\bigoplus_{ j\leq i_1}\overline{P}(j)}\,(M) =  \tau_{i_{1}}(M).$
On the other hand
$$\tau_{i_{1}}(M)=\mathrm{Tr}_{\overline{P}(i_{1})}(M)+
\mathrm{Tr}_{\bigoplus_{ j< i_1}\overline{P}(j)}\,(M) =\mathrm{Tr}_{\overline{P}(i_{1})}(M)$$
since, by Proposition \ref{imagenbaja}, we know that  $\mathrm{Tr}_{\{\overline{P}(j)\mid j< i_{1}\}}(M)=0$.  Let $f\in\Hom(\overline{P}(i_{1}),M).$ Then,  by Proposition \ref{imagenbaja}, we have that $\mathrm{Im}\,(f)\subseteq M_{1}$ and so $\tau_{i_{1}}(M)\subseteq M_{1};$ proving that $\tau_{i_{1}}(M)=M_{1}.$
\

At this point, we have $M_1=\tau_{i_{1}}(M)\subseteq \tau_{i'_{2}}(M).$ To finish the proof of (ii), we only have to see that $\tau_{i'_{2}}(M)\subseteq \tau_{i_{1}}(M).$ Let
$j\leq i'_2< i_{2}$ and $f\in\Hom(\overline{P}(j),M).$ Then, by Proposition \ref{imagenbaja}, it follows that
$\mathrm{Im}\,(f)\subseteq M_{1}=\tau_{i_{1}}(M)$ and thus  $\tau_{i'_{2}}(M)\subseteq \tau_{i_{1}}(M).$
\

(iii) $\tau_{i'_{3}}(M)= M_{2}\quad\forall\,i'_3\in[i_2,i_3).$
\

Firstly, we assert that $\tau_{i_{2}}(M)=M_{2}.$ Indeed, consider the exact sequence

$$\xymatrix{0\ar[r] & M_{1} \ar[r] & M_{2}\ar[r] & X_{2}\ar[r] & 0,}$$
where $X_{2}\in \Delta(i_{2})^{\oplus}$. By (ii), we know that $M_{1}=\mathrm{Tr}_{Q_1}\,(M),$ where $Q_1:=\bigoplus_{j\leq i_1}\overline{P}(j).$
There exists an epimorphism $f:Q_1^{(I_1)}\to M_{1},$ for the set $I_1:=\Hom(Q_1,M_1).$ On the other hand,
since $X_{2}\in \Delta(i_{2})$, there exists an epimorphism $h:\overline{P}(i_{2})^{m_{2}}\fl X_{2}.$
Then, we have the following exact and commutative diagram
$$\xymatrix{ 0\ar[r] & Q_1^{(I_1)}\ar[r]\ar[d]_{f} & Q_1^{(I_1)}\bigoplus\overline{P}(i_{2})^{m_{2}}\ar[r]\ar[d]_{g}
& \overline{P}(i_{2})^{m_{2}}\ar[r]\ar[d]_{h} & 0\\
0\ar[r] & M_{1}\ar[r] & M_{2}\ar[r] & X_{2}\ar[r] & 0,}$$
where $g$ is an epimorphism. Therefore
$$M_{2}\subseteq \mathrm{Tr}_{\bigoplus_{ j\leq  i_2}\overline{P}(j)}\,(M)=\tau_{i_{2}}(M).$$
Now, let  $f:\overline{P}(j)\fl M$ with $j\leq i_{2}.$ Then,  By Proposition \ref{imagenbaja},
we get that $\mathrm{Im}\,(f)\subseteq N,$ where $N=M_{1}$ or $N=M_{2}.$ In any case, we conclude that
$\mathrm{Im}\,(f)\subseteq M_{2},$ since $M_{1}\subseteq M_{2}$. Hence $\tau_{i_{2}}(M)\subseteq M_2$ and so
$\tau_{i_{2}}(M)=M_{2}.$ Now, by following the same arguments as we did in (ii), we can show that $\tau_{i'_{3}}(M)=\tau_{i_{2}}(M);$ proving (iii).
\

Note that the above procedure can be repeated in order to get that $N_{i}=\tau_i(M),$ for all $0\leq i\leq i_{a}.$ Finally, by following the process we did in (iii), we obtain an epimorphism $Q_{a-1}^{(I_{a-1})}\bigoplus\overline{P}(i_{a})^{m_{a}}\to M$ and thus
$\tau_{i_{a}}(M)=M.$
\

(b) $\Rightarrow$ (a) Assume the hypothesis of (b). If $\Supp(\tau_{\B,M})=\emptyset,$ by using transfinite induction, it can be shown that $M=0$ and thus
$M\in\F_f(\Delta).$
\

Let $\Supp(\tau_{\B,M})=\{i_1<i_2<\cdots <i_a\}.$ Consider $M_0:=\tau_0(M)$ and $M_k:=\tau_{i_k}(M)$ for $k\in[1,a].$ Note that, for any
$i\not\in\Supp(\tau_{\B,M}),$ Lemma \ref{SubfTr2} implies that $\tau_i(M)=\sum_{j<i}\tau_j(M).$
\

In the proofs of the following assertions, we use transfinite induction.
\

$(0)$ $\tau_{i'_1}(M)=M_0=0$ for any $i'_1\in[0,i_1).$
\

Indeed, let $S_{i_1}=\{i'_1\in[0,i_1)\;:\;\tau_{i'_1}(M)=0\}.$ Note that $0\in S_{i_1}$ since
$\tau_0(M)=\sum_{j<0}\tau_j(M)=0.$ Let $\beta+1\in[0,i_1)$ and $\beta\in S_{i_1}.$ Since $j<\beta+1$ implies that $j\leq\beta,$ it follows that
$\tau_{\beta+1}(M)=\sum_{j<\beta+1}\tau_j(M)\subseteq\tau_\beta(M)=0,$ and thus $\beta+1\in S_{i_1}.$ \\
Let $\gamma\in[0,i_1)$ be a limit ordinal and let $\delta\in S_{i_1}$ for any $\delta<\gamma.$ Then $\tau_\gamma(M)=\sum_{\delta<\gamma}\tau_j(M)=0.$ Thus, by transfinite induction, we get that $(0)$ holds.
\

$(1)$ $\tau_{i'_2}(M)=M_1$ for any $i'_2\in[i_1,i_2).$
\

Indeed, let $S_{i_2}=\{i'_2\in[i_1,i_2)\;:\;\tau_{i'_2}(M)=M_1\}.$ It is clear that $i_1\in S_{i_2}.$ Let $i_1<\beta+1<i_2$ and $\tau_\beta(M)=M_1.$ Then,
$M_1\subseteq\tau_{\beta+1}(M)=\sum_{j<\beta+1}\tau_j(M)\subseteq\tau_\beta(M)=M_1$ and hence $\beta+1\in S_{i_2}.$\\
Let $\gamma\in[i_1,i_2)$ be a limit ordinal and let $\delta\in S_{i_2}$ for any $\delta\in[i_1,\gamma).$ Then, by using $(0)$ we can get the following equalities
\begin{align*}
\tau_\gamma(M)&=\sum_{j<\gamma}\tau_j(M)\\
                         &=\sum_{j<i_1}\tau_j(M)+\sum_{i_1\leq \delta<\gamma}\tau_\delta(M)\\
                         &=M_1.
\end{align*}
Thus, by transfinite induction, we get that $(1)$ holds.
\

Note that the above procedure in $(0)$ and $(1)$ can be repeated to obtain that $\tau_{i'_k}(M)=M_k$ for any $i'_k\in[i_{k-1},i_k),$ and $\tau_j(M)=M$ for
$j\geq i_a.$ Thus, we have a finite chain of submodules $0\subseteq M_0\subseteq M_1\subseteq \cdots\subseteq M_a=M$ such that
$M_t/M_{t-1}=\tau_{i_t}/\overline{\tau_{i_t}} (M)\in\Delta(i_t).$ Therefore, $M\in\F_f(\Delta).$
\end{dem}

\begin{rk}\label{Rkfiltraza} Let $\R$ be a  locally finite $\K$-ringoid,  $\B:=\{\B_{i}\}_{i<\alpha}$  be an admissible family of subcategories of $\R,$
 $\Delta={}_\B\Delta,$ and $M\in\Mod_\rho(\R)$ be such that $\Supp(\tau_{\B,M})=\{i_1<i_2<\cdots<i_a\},$ for some finite ordinal $a.$ In the proof of Theorem \ref{filtraza}, we have shown the following:
\begin{itemize}
\item[(a)] $\tau_j(M)=0$ for all $j\in[0,i_1);$
\item[(b)] $\tau_j(M)=M_k:=\tau_{i_k}(M)$ for all $j\in[i_k,i_{k+1})$ and $k\in[1,a);$
\item[(c)] the finite chain of submodules $0\subseteq M_0\subseteq M_1\subseteq \cdots\subseteq M_a=M$ satisfies  that
$M_t/M_{t-1}=\tau_{i_t}/\overline{\tau_{i_t}} (M).$
\end{itemize}
\end{rk}

\begin{teo}\label{filt-sum-direct}   Let $\R$ be a locally finite $\K$-ringoid,  $\B:=\{\B_{i}\}_{i< \alpha}$ be
an admissible family of subcategories of $\R$ and $\Delta={}_\B\Delta$. If $\Delta\subseteq\finp_\rho(\R)$ then all the objects $\Delta_e(i)$ are local and  the following statements hold true.
\begin{itemize}
\item[(a)] For any $M\in\F_f(\Delta),$ the filtration multiplicity $[M:\Delta_{e}(i)]$ does not depend on a given $\Delta$-filtration of $M.$
\item[(b)] $\F_{f}(\Delta)\subseteq\finp_\rho(\R)$ and it is  a locally finite $\K$-ringoid.
\end{itemize}
\end{teo}
\begin{dem}  Let $\Delta\subseteq\finp_\rho(\R).$ By Proposition \ref{LARp3}, we have that $\finp_\rho(\R)$ is
a locally finite $\K$-ringoid. Moreover, by \cite[Proposition 4.2 (d)]{Aus} we have that $\Delta^\oplus\subseteq\finp_\rho(\R)$ and thus all $\Delta_e(i)$ are
local objects (see Lemma \ref{claveD} (a)).
\

(a) Let $0\neq M\in \F_{f}(\Delta).$  Since $\Delta^\oplus\subseteq\finp_\rho(\R)$ and all the objects $\Delta_e(i)$ are local, the proof of Theorem \ref{filtraza} implies that
$$[M:\Delta_e(i)]_\xi=[M:\Delta_e(i)]_{\xi'}=[M:\Delta_e(i)]_{\tau_{\B,M}},$$
where $\xi'$ is the strictly well ordered filtration of $M,$ constructed  by the proof of Proposition \ref{f=f'} and Proposition \ref{filtracionordenada}. Note that
$[M:\Delta_e(i)]_{\tau_{\B,M}}$ does not depend on any given filtration. Therefore the proof of (a) is complete.
\

(b) Let $ M\in \F_{f}(\Delta).$ Since $\F_{f}(\Delta)$ is closed under extensions and $\Delta^\oplus\subseteq\finp_\rho(\R),$ by induction on the
$\xi$-ladder length $\ell'_{\Delta,\xi}(M),$ we can show that $M\in\finp_\rho(\R).$ Therefore $\F_{f}(\Delta)\subseteq\finp_\rho(\R).$
\

Assume now that $M=L\bigoplus N$ in $\Mod_\rho(\R).$ Since $M\in\finp_\rho(\R),$ it follows that  $M=L\bigoplus N$ in $\finp_\rho(\R).$
Consider the split exact sequence
$$\xi:\quad \xymatrix{0\ar[r] & L\ar[r]^{f} & M\ar[r]^{g} & N\ar[r] & 0},$$
given by the decomposition $M=L\bigoplus N.$  Thus we have the following split exact sequence
$$\varepsilon_{i}(\xi):\quad \xymatrix{0\ar[r] & \varepsilon_{i}(L)\ar[r]^{\varepsilon_{i}(f)} & \varepsilon_{i}(M)\ar[r]^{\varepsilon_{i}(g)}
& \varepsilon_{i}(N)\ar[r] & 0,}$$
for $\varepsilon_{i}=\tau_i$ or $\varepsilon_{i}=\overline{\tau_{i}}.$ Note that
\[(*)\quad\tau_{i}/\overline{\tau_{i}}(M)=\frac{\tau_{i}(L)\oplus\tau_{i}(N)}{\overline{\tau}_{i}(L)\oplus\overline{\tau_{i}}(N)}=\frac{\tau_{i}(L)}{\overline{\tau}_{i}(L)}\bigoplus\frac{\tau_{i}(N)}{\overline{\tau}_{i}(N)}.\]

Since $M\in \F_{f}(\Delta),$ by Theorem \ref{filtraza}, there exists $i_{0}\in \mathbb{N}$ such that
$\tau_j(M)=M$ for any $j\geq i_{0}$. Moreover, by $ (*)$ we have $\tau_{i}/\overline{\tau_{i}}(M)\in\Delta(i)^\oplus$ for any $i<\alpha.$
\

From the split-exact sequences of the form $\varepsilon_{i}(\xi),$ for $\varepsilon_{i}=\tau_i$ or $\varepsilon_{i}=\overline{\tau_{i}},$  we get the following commutative and exact diagram,
where all the rows  are split exact sequences
$$\xymatrix{&\\&\\ (**)\\&\\&}\quad\xymatrix{ & 0 & 0 & 0\\
0\ar[r] & \tau_i/\overline{\tau}_i(L)\ar[r]\ar[u] & \tau_i/\overline{\tau}_i(M)\ar[r]\ar[u] &
\tau_i/\overline{\tau}_i(N)\ar[r]\ar[u] & 0\\
0\ar[r] & \tau_i(L)\ar[r]^{\tau_i(f)}\ar[u] & \tau_i(M)\ar[r]^{\tau_i(g)}\ar[u] & \tau_i(N)\ar[r]\ar[u] & 0\\
0\ar[r] & \overline{\tau}_i(L)\ar[r]^{\overline{\tau}_i(f)}\ar[u] & \overline{\tau}_i(M)\ar[r]^{\overline{\tau}_i(g)}\ar[u]
& \overline{\tau}_i(N)\ar[r]\ar[u] & 0\\
& 0\ar[u] & 0\ar[u] & 0\ar[u]}$$
If $ \tau_i/\overline{\tau}_i(M)=0,$ we conclude that
$ \tau_i/\overline{\tau}_i(L)=0=\ \tau_i/\overline{\tau}_i(N).$ In particular $\Supp(\tau_{\B,L})\cup\Supp(\tau_{\B,N})\subseteq\Supp(\tau_{\B,M}).$
\

Let $ \tau_i/\overline{\tau}_i(M)\neq 0$. Since
$ \tau_i/\overline{\tau}_i(M)\in\Delta(i)^{\oplus},$ it follows that
$$ \tau_i/\overline{\tau}_i(L)\bigoplus  \tau_i/\overline{\tau}_i(N)\in\Delta(i)^{\oplus}.$$
Since $\finp_\rho(\R)$ is a Krull-Schimidt category, we get that
$ \tau_i/\overline{\tau}_i(L)\in\Delta(i)^{\oplus}$ and
$ \tau_i/\overline{\tau}_i(N)\in\Delta(i)^{\oplus}.$  Furthermore, from $(**),$ Theorem \ref{filtraza} and the fact that
$\tau_j(M)=M$ for every $j\geq i_{0},$ we get that $M=\tau_{i_0}(L)\oplus \tau_{i_0}(N)$ and $\tau_{i_0}(L), \tau_{i_0}(N)\in \F_{f}(\Delta)$.
But $M=L\oplus N$ and thus $0=(L/\tau_{i_0}(L))\oplus (N/\tau_{i_0}(N)).$ Therefore $L=\tau_{i_0}(L)$ and $N=\tau_{i_0}(N),$ proving that $\F_{f}(\Delta)$
is closed under direct summands.
\end{dem}

\begin{cor}\label{C1filt-sum-direct}  Let $\R$ be a locally finite $\K$-ringoid, which is right support finite,  $\B:=\{\B_{i}\}_{i< \alpha}$ be
an admissible family of subcategories of $\R$ and $\Delta={}_\B\Delta$. Then  all the objects $\Delta_e(i)$ are local and the following statements hold true.
\begin{itemize}
\item[(a)] For any $M\in\F_f(\Delta),$ the filtration multiplicity $[M:\Delta_{e}(i)]$ does not depend on a given $\Delta$-filtration of $M.$
\item[(b)] $\F_{f}(\Delta)\subseteq\finp_\rho(\R)$ and it is  a locally finite $\K$-ringoid.
\end{itemize}
\end{cor}
\begin{proof} It follows from Proposition \ref{DSuppFin} and Theorem \ref{filt-sum-direct}.
\end{proof}

We recall that a class $\X$ of objects, in an abelian category $\A,$ is {\bf pre-resolving} if it is closed under extensions and for any exact sequence
$0\to A\to B\to C\to 0,$ with $B,C\in\X,$ it follows that $A\in\X.$ We prove that $\ltt{F}_{f}(\Delta)$ is a pre-resolving class, and in order to do that, we
start with the following lemma.

\begin{lem}\label{epicasouno}
Let $\R$ be a locally finite $\K$-ringoid,  $\B:=\{\B_{i}\}_{i< \alpha}$ be
an admissible family of subcategories of $\R,$ and $\Delta={}_\B\Delta\subseteq\finp_\rho(\R)$. Then, the following statements hold true.
\begin{enumerate}
\item [(a)]
Let $u:L\fl M$ be a monomorphism with $M\in \Delta(i)^{\oplus}$. For $e\in \sigma_{i}(\B)$, we have that
$\Hom(U_{e}(i),L)=0,$ where $U_{e}(i):=\mathrm{Tr}_{\oplus_{j<i}\overline{P}(j)}(P^{op}_{e}(i))$.

\item [(b)] Let $\xi:\,\,\xymatrix{0\ar[r] & L\ar[r]^{u} & M\ar[r]^{\pi} & N\ar[r] & 0}$ be an
exact sequence with $M,N\in \Delta(i)^{\oplus}$. Then $L\in \Delta(i)^{\oplus}$.
\end{enumerate}
\end{lem}
\begin{proof}
(a)  First, we show that $\Hom(P_{e'}^{op}(j),L)=0$ for $j<i$ and $e'\in \sigma_{j}(\B)$.\\
Indeed, we have that $M=\bigoplus_{e\in J_{i}}\Delta_{e}(i)^{\mu_{e}}$ with $J_{i}\subseteq \sigma_{i}(\B)$
a finite subset. Then $$\Hom(P_{e'}^{op}(j),M)\simeq \bigoplus_{e\in J_{i}} \Hom(P_{e'}^{op}(j),\Delta_{e}(i))^{\mu_{e}}=0$$
since $\Hom(P_{e'}^{op}(j),\Delta_{e}(i))=0$ for $j<i$ and for every $e'\in \sigma_{j}(\B)$.
Now, let $\alpha: P_{e'}^{op}(j)\fl L$ be a morphism. Then  $u\alpha\in \Hom(P_{e'}^{op}(j),M)=0.$
Since $u$ is a monomorphism, we have that $\alpha=0$. This proves that $\Hom(P_{e'}^{op}(j),L)=0$ for
$j<i$ and $e'\in \sigma_{j}(\B)$.\\
Therefore,  for $j<i,$ it follows that
$$\Hom(\overline{P}(j),L)\simeq \prod_{e'\in \sigma_{j}(\B)}\Hom(P_{e'}^{op}(j),L)=0$$
since $\overline{P}(j):=\displaystyle \bigoplus_{e'\in\sigma_j(\B)}\,P_{e'}^{op}(j)$. Consider $X:=\mathrm{Hom}\Big(\bigoplus_{j<i}\overline{P}(j),U_{e}(i)\Big).$ From the equality $U_{e}(i)=\mathrm{Tr}_{\oplus_{j<i}\overline{P}(j)}(P_{e}(i)),$ there exists an epimorphism
$$\lambda:(\bigoplus_{j<i}\overline{P}(j))^{(X)}\fl U_{e}(i).$$ 
 Let $\gamma\in \Hom(U_{e}(i),L).$ Then  $\gamma\lambda\in \Hom \Big(\big(\bigoplus_{j<i}\overline{P}(j)\big)^{(X)},L\Big).$
But
$$\Hom \Big(\big(\bigoplus_{j<i}\overline{P}(j)\big)^{(X)},L\Big)
\simeq \prod_{j<i}\prod_{x\in X}\Hom(\overline{P}(j),L)=0.$$
Hence $\gamma\lambda=0$ and thus $\gamma=0$, since $\lambda$ is an
epimorphism, proving that  $\Hom(U_{e}(i),L)=0$.
\

(b) Since $N\in \Delta(i)^{\oplus},$ we have that $N=\bigoplus_{e\in K_{i}}\Delta_{e}(i)^{\nu_{e}}$ with $K_{i}\subseteq \sigma_{i}(\B)$
a finite subset. For each $e\in K_{i},$ there is an exact sequence
$$\xymatrix{0\ar[r] & U_{e}(i)\ar[r] &
P_{e}(i)\ar[r] & \Delta_{e}(i)\ar[r] & 0.}$$
By applying $\Hom(-,L)$ to the above sequence,  we obtain the exact sequence
$$\xymatrix{\Hom(U_{e}(i),L)\ar[r] & \Ex^{1} (\Delta_{e}(i),L)\ar[r] & \Ex^{1}(P_{e}(i),L).}$$
Since $\Hom(U_{e}(i),L)=0$ by (a), and $\Ex^{1}(P_{e}(i),L)=0,$ it follows
that $\Ex^{1}(\Delta_{e}(i),L)=0$ for each $e\in K_{i}$. Then
$$\Ex^{1}(N,L)=\prod_{e\in K_{i}}\Ex^{1}(\Delta_{e}(i),L)^{\nu_{e}}=0.$$
We conclude that $\xi$ splits and thus  $L\oplus N=M\in \Delta(i)^{\oplus}.$ Finally, from the fact that
 $\finp_\rho(\R)$ is a Krull-Schmidt category, we get that $L\in \Delta(i)^{\oplus}.$
\end{proof}

\begin{pro}\label{ResClass}
Let $\R$ be a locally finite $\K$-ringoid,  $\B:=\{\B_{i}\}_{i< \alpha}$ be 
an admissible family of subcategories of $\R,$ and let $\Delta={}_\B\Delta\subseteq\finp_\rho(\R)$. Then $\ltt{F}_{f}(\Delta)$ is a pre-resolving class.
\end{pro}
\begin{proof} By Remark \ref{filclodesext}, we know that $\ltt{F}_{f}(\Delta)$ is closed under extensions. It remains to show that  $\ltt{F}_{f}(\Delta)$  is closed under kernels of epimorphisms between its objects.
\

Let $\xi:\,\,\xymatrix{0\ar[r] & L\ar[r]^{u} & M\ar[r]^{\pi} & N\ar[r] & 0}$ be an
exact sequence with $M,N\in \ltt{F}_{f}(\Delta)$. Let $\{\tau_{i}(M)\}_{i<\alpha}$ and
$\{\tau_{i}(N)\}_{i<\alpha}$ be the $\B$-trace filtrations
of $M$ and $N,$ respectively. By Theorem \ref{filtraza}, we have that $\Supp(\tau_{\B,M})$ is finite, that is,
$\Supp(\tau_{\B,M})=\{i_1<i_2<\cdots <i_a\}.$  Then $\tau_{j}(M)=M$  for every $j\geq i_{a}$.
Since $\pi$ is an epimorphism, we have that $\tau_{j}(N)=N$ for every $j\geq i_{a}$. Moreover, by using that $\pi$
is an epimorphism and the fact that $\bigoplus_{j\leq i}\overline{P}(j)$ and  $\bigoplus_{j<i}\overline{P}(j)$ are projectives, we conclude that
$\pi(\epsilon_{i}(M))=\epsilon_{i}(N)$ for every $i<\alpha$, where $\epsilon_{i}=\tau_{i}$ or $\epsilon_{i}=\overline{\tau}_{i}$
(that is $\epsilon_{i}(\pi):=\pi|_{\epsilon_{i}(M)}:\epsilon_{i}(M)\fl \epsilon_{i}(N)$ is an epimorphism).
Let $\overline{L}_{i}:=\mathrm{Ker}(\overline{\tau}_{i}(\pi))$ and $L_{i}:=\mathrm{Ker}(\tau_{i}(\pi)).$ Then, for
each $i<\alpha,$ we obtain the following commutative and exact diagram
$$\xymatrix{0\ar[r] & \overline{L}_{i}\ar[r]\ar[d]_{u_{i}} &
\overline{\tau}_{i}(M)\ar[r]^{\overline{\tau}_{i}(\pi)}\ar[d]_{v_{i}} & \overline{\tau}_{i}(N)\ar[r]\ar[d]_{w_{i}} & 0\\
0\ar[r] & L_{i}\ar[r] & \tau_{i}(M)\ar[r]^{\tau_{i}(\pi)} & \tau_{i}(N)\ar[r] & 0,}$$
where $u_{i},v_{i}$ and $w_{i}$ are monomorphisms.
By the Snake's Lemma, there exists the following exact sequence
$$\xymatrix{0\ar[r] & \frac{L_{i}}{\overline{L}_{i}}\ar[r] & \frac{\tau_{i}(M)}{\overline{\tau}_{i}(M)}
\ar[r] & \frac{\tau_{i}(N)}{\overline{\tau}_{i}(N)}\ar[r] & 0.}$$
Since $M,N\in \ltt{F}_{f}(\Delta)$, by Theorem \ref{filtraza}, we obtain that
$\frac{\tau_{i}(M)}{\overline{\tau}_{i}(M)}, \frac{\tau_{i}(N)}{\overline{\tau}_{i}(N)} \in \Delta(i)^{\oplus}$.
By Lemma \ref{epicasouno}, it follows that $\frac{L_{i}}{\overline{L}_{i}}\in \Delta(i)^{\oplus}$.\\
Recall that $\Supp(\tau_{\B,M})=\{i_1<i_2<\cdots <i_a\}.$ Hence,  by Remark \ref{Rkfiltraza}, the following statements hold true
\begin{itemize}
\item[(a)] $\tau_j(M)=0$ for every $j\in[0,i_1),$
\item[(b)] $\tau_j(M)=M_k:=\tau_{i_k}(M)=\overline{\tau}_{i_{k+1}}(M)$ for every $j\in[i_k,i_{k+1})$ and $k\in[1,a-1),$
\item[(c)] the finite chain of submodules $0\subseteq M_0\subseteq M_1\subseteq \cdots\subseteq M_a=M$ satisfies  that
$M_t/M_{t-1}= \tau_{i_{t}}/\tau_{i_{t-1}}(M)=\tau_{i_t}/\overline{\tau_{i_t}} (M).$
\end{itemize}
For $i=i_{a},$ we have that $\tau_{i}(M)=M$ and hence $\tau_{i}(\pi)=\pi.$ Therefore 
 $L_{i_{a}}=L$. We set $L_{k}:=L_{i_{k}}$ for $k\in [1,a].$  Hence we have the following filtration
$$0=L_{0}\subseteq L_{1}\subseteq L_{2}\subseteq \dots \subseteq L_{a-1}\subseteq L_{a}=L.$$
By the item (b)  $\tau_{i_k}(M)=\overline{\tau}_{i_{k+1}}(M)$ for  $k\in[1,a-1),$ and so
$\tau_{i_{k}}(\pi)=\overline{\tau}_{i_{k+1}}(\pi)$. Therefore, we conclude that
$L_{i_{k}}=\overline{L}_{i_{k+1}}$ for every $k\in[1,a-1)$.
Then
$\frac{L_{k}}{L_{k-1}}=\frac{L_{i_{k}}}{L_{i_{k-1}}}=
\frac{L_{i_{k}}}{\overline{L_{i_{k}}}}\in \Delta(i_{k})^{\oplus}$ for $k\in [1,a]$.
This give us a finite filtration of $L$ proving that $L\in \ltt{F}_{f}(\Delta)$.
\end{proof}

\section{Stratifying ideals in ringoids}

In this section we introduce and study the notion of ideally standardly stratified ringoid. We prove that standardly stratified ringoids and ideally standardly stratified ringois are only equivalent notions under a specific condition. It is also shown that certain equivalent characterizations of standardly stratified algebras and quasi-hereditary algebras are not necessarily equivalent anymore in the realm of ringoids.

\begin{defi} Let $\R$ be a ringoid. An ideal $I\unlhd\R$ is {\bf right stratifying} if $I^2=I$  and
 $I(-,a)\in\proj_\rho(\R)$ for any $a\in\R.$ We say that $I$ is {\bf right hereditary} if it is right stratifying and $I\rad_\R(-,?)I=0.$
 A right stratifying (respectively,  hereditary) chain in $\R$ is a chain $\{I_i\}_{i<\alpha}$ of ideals of $\R$ such that
 $\sum_{i<\alpha}\,I_i= \R$ and  $I_i/I'_i$ is right stratifying (respectively, hereditary) in $\R/I'_i,$ where $I'_i:=\sum_{j<i}\,I_j.$
\end{defi}

\begin{lem}\label{AuxB} Let $\R$ be a Krull-Schmidt $\K$-ringoid, and let $\B:=\{\B_{i}\}_{i<\alpha}$ be an exhaustive family of subcategories
of $\R.$ Then, the following statements hold true.
\begin{itemize}
\item[(a)] $\sum_{j<i}\,I_{\B_j}=I_{\bigcup_{j<i}\B_j},$ where $I_\emptyset(a,b):=\{0\}$ for $a,b\in\R.$
\item[(b)] $\sum_{j<\alpha}\,I_{\B_j}=\R.$
\end{itemize}
\end{lem}
\begin{dem} (a) Let $f\in I_{\bigcup_{j<i}\B_j}(x,y).$ Then, $f$ factorizes through some $b\in\bigcup_{j<i}\B_j.$ Therefore $f\in I_{\B_j}(x,y)$ for some
$j<i,$ and thus $f\in\sum_{j<i}\,I_{\B_j}(x,y),$ proving that $I_{\bigcup_{j<i}\B_j}\subseteq\sum_{j<i}\,I_{\B_j}.$
\

Let $f\in \sum_{j<i}\,I_{\B_j}(x,y).$ Then, $f=\sum_{k=1}^n\,f_k$ for some $f_k\in I_{\B_{j_k}}(x,y)$ with $j_k<i.$ In particular, each $f_k$ is the
composition of morphisms $x\xrightarrow{t_k}b_{j_k}\xrightarrow{h_k}y,$ where $b_{j_k}\in\B_{j_k}.$ Let $b:=\bigoplus_{k=1}^n\,b_{j_k}.$ Then, we have
the matrix morphisms $x\xrightarrow{t}b\xrightarrow{h}x$ such that $f=ht.$ Since $b\in\B_j$ for $j:=\max\{j_1,j_2,\cdots,j_n\}<i,$ it follows that
$f\in I_{\bigcup_{j<i}\B_j}(x,y),$ proving that $\sum_{j<i}\,I_{\B_j}\subseteq I_{\bigcup_{j<i}\B_j}.$
\

(b) It follows from (a), since $\bigcup_{j<\alpha}\B_j=\R.$
\end{dem}

\begin{defi}\label{idealystandarquasi}
Let $\R$ be a Krull-Schmidt $\K$-ringoid. We say that $\R$ is a  right {\bf ideally} standardly stratified (respectively, quasi-hereditary) $\K$-ringoid, with respect to an exhaustive family $\B:=\{\B_{i}\}_{i< \alpha}$ of subcategories of $\R,$ if the associated chain $\{I_{\B_i}\}_{i<\alpha}$ of ideals of $\R$
 is right stratifying (respectively, hereditary).
 \end{defi}

\begin{rk} A right ideally  quasi-hereditary $\K$-ringoid $\R,$ with respect to an exhaustive family of subcategories
$\B=\{B_i\}_{i<\alpha}$ of $\R$ such that $\alpha\leq\aleph_0,$ is called quasi-hereditary category in \cite{Martin}.
\end{rk}

\begin{lem}\label{Delta1} Let $\R$ be a locally finite $\K$-ringoid, and let $\B=\{\B_{i}\}_{i< \alpha}$ be an  exhaustive family  of subcategories of
$\R$   such that
$$(I_{\B_j}/I'_{\B_j})\;\rad_{\R/I'_{\B_j}}(-,?)\;(I_{\B_j}/I'_{\B_j})=0,$$
for any $j<\alpha.$ Then, the following
statements hold true.
\begin{itemize}
\item[(a)] $\rad_\R(e,e')=I'_{\B_i}(e,e')$ for any $e,e'\in\sigma_i(\B)$ and $i<\alpha.$
\item[(b)] $\Hom(\Delta_e(i),\Delta_{e'}(i))=0$ for any $e\neq e'$ in $\sigma_i(\B)$ and $i<\alpha.$
\item[(c)] $\End\,(\Delta_{e}(i))\simeq \frac{\mathrm{End}_{\R}(e)}{\mathrm{rad}(\mathrm{End}_{\R}(e))}$ for any $e\in\sigma_i(\B)$ and $i<\alpha.$
\end{itemize}
\end{lem}
\begin{dem} Let $e, e'$ in $\sigma_i(\B).$ Since $I'_{\B_i}:=\sum_{j<i}\,I_{\B_j}=I_{\bigcup_{j<i}\B_j},$ we can adapt some part of the proof given in
\cite[Theorem 3.6 (i)]{Martin} to get (a). Finally, (b) and (c) follow from (a), Lemma \ref{claveD1} and  Lemma \ref{claveD2}.
\end{dem}

\begin{teo}\label{FilTrProj} Let $\R$ be a locally finite $\K$-ringoid and $\B=\{\B_{i}\}_{i< \alpha}$ be an  exhaustive family  of subcategories of
$\R.$ Then, the following statements are equivalent,  for  $i<\alpha$ and $e\in\sigma_i(\B).$
\begin{itemize}
\item[(a)] $\Tr_{\bigoplus_{j<i}\overline{P}(j)}\,(P_e(i))\in\F_f(\bigcup_{j<i}\,\Delta(j)).$
\item[(b)] The set $\{j<\alpha\;:\; I_{\B_j}(-,e)/I'_{\B_j}(-,e)\neq 0\}$ is finite, there is some $i_0<\alpha$ such
that $I_{\B_j}(-,e)=\R(-,e)$ for $j\geq i_0,$ and
$$I_{\B_t}(-,e)/I'_{\B_t}(-,e)\in\proj_\rho(\R/I'_{\B_t})$$ for any  $t<\alpha.$
\end{itemize}
\end{teo}
\begin{dem} Let $e\in\sigma_i(\B)$ and $t<\alpha.$ By Lemma \ref{AuxB} and Proposition \ref{DSuppFin}, we have $I_{\B_t}(-,e)=\tau_t(P^{op}_e(i))$ and
$I'_{\B_t}(-,e)=\overline{\tau}_t(P^{op}_e(i)).$ In particular, $\overline{\tau}_i(P^{op}_e(i))=\Tr_{\bigoplus_{j<i}\overline{P}(j)}\,(P^{op}_e(i))$ and
$\Supp(\tau_{\B,P^{op}_e(i)})=\{j<\alpha\;:\; I_{\B_j}(-,e)/I'_{\B_j}(-,e)\neq 0\}.$
\

(a) $\Rightarrow$ (b) By (a) and  the following exact sequence
$$\xymatrix{0\ar[r] & \overline{\tau}_i(P^{op}_e(i))\ar[r] & P^{op}_{e}(i)\ar[r] & \Delta_{e}(i)\ar[r] & 0},$$
it follows that $P^{op}_{e}(i)\in \ltt{F}_{f}\Big(\bigcup_{j\leq i}\Delta(j)\Big).$  Then, by Theorem \ref{filtraza}, we get that $\Supp(\tau_{\B,P^{op}_e(i)})$ is finite, there is some $i_0<\alpha$ such that $\tau_j(P^{op}_e(i))=P^{op}_e(i)$ for $j\geq i_0,$ and $\tau_k/\overline{\tau}_k(P^{op}_e(i))\in\Delta(k)^\oplus$ for any $k<\alpha.$
\

Let $t<\alpha.$ For each $h\in\sigma_t(\B),$ we have $\Delta_h(t)=\R(-,h)/I'_{\B_t}(-,h)$ and thus $\Delta_h(t)\in\proj_\rho(\R/I'_{\B_t}).$ Then,
$\tau_t/\overline{\tau}_t(P^{op}_e(i))\in\Delta(t)^\oplus$ implies that $\tau_t/\overline{\tau}_t(P^{op}_e(i))\in\proj_\rho(\R/I'_{\B_t}).$
\

(b) $\Rightarrow$ (a) Let (b) holds true. We need to show that $\overline{\tau}_i(P^{op}_e(i))\in\F_f(\bigcup_{j<i}\,\Delta).$ We may assume that $\overline{\tau}_i(P^{op}_e(i))\neq0.$
\

By hypothesis, there is some $k_0<\alpha$ such that $I_{\B_k}(-,e)=I_{\B_{k_0}}(-,e)=\R(-,e),$ for any $k\geq k_0.$ Consider the set
$S:=\{k\leq k_0\;:\; I_{\B_k}(-,e)=I_{\B_{k_0}}(-,e)\}.$ Since $S\neq\emptyset$ there exists $k_1:=\min\,S.$ Therefore,
$I_{\B_k}(-,e)=I_{\B_{k_1}}(-,e)=\R(-,e)$ for any $k\geq k_1,$ and $I_{\B_j}(-,e)\subsetneq I_{\B_{k_1}}(-,e)$ for $j<k_1.$
\

We assert that $i<k_1.$ Indeed, suppose that $k_1\leq i.$ Then
\begin{align*}
\overline{\tau}_i(P^{op}_e(i))&=\sum_{j<i}\,\tau_j(P^{op}_e(i))\\
                                      &= \sum_{j<k_1}\,\tau_j(P^{op}_e(i))+\sum_{k_1\leq j<i}\,\tau_j(P^{op}_e(i))\\
                                      &=\sum_{j<k_1}\,\tau_j(P^{op}_e(i))+P^{op}_e(i)\\
                                      &=P^{op}_e(i),
\end{align*}
and thus $\Delta_e(i)=P^{op}_e(i)/\overline{\tau}_i(P^{op}_e(i))=0,$ contradicting Proposition \ref{DSuppFin} (a); proving that $i<k_1.$ Let
$\Supp(\tau_{\B,P^{op}_e(i)})=\{i_1<i_2<\cdots<i_a\}.$ Note that $i_a<k_1.$
\

We assert that $\overline{\tau}_i(P^{op}_e(i))=I_{\B_{i_k}}(-,e)$ for some $k\in[1,a]$ with $i_k<i.$\\
Indeed, we have two cases to consider: (1) Let $i=i_k$ for some $k\in[1,a].$ Since $\overline{\tau}_i(P^{op}_e(i))\neq0,$ we have that $k\geq 2.$ Then, by
Remark \ref{Rkfiltraza}, we obtain
\begin{align*}
\overline{\tau}_i(P^{op}_e(i))& =\sum_{j<i_k}\,I_{\B_j}(-,e)\\
                                    & =\sum_{j<i_{k-1}}\,I_{\B_j}(-,e)+\sum_{i_{k-1}\leq j<i_k}\,I_{\B_j}(-,e)\\
                                     & =\sum_{j<i_{k-1}}\,I_{\B_j}(-,e)+I_{\B_{i_{k-1}}}(-,e)\\
                                     &= I_{\B_{i_{k-1}}}(-,e).
\end{align*}
(2) Let $i\neq i_k$ for any $k\in[1,a].$ In particular $\overline{\tau}_i(P^{op}_e(i))=\tau_i(P^{op}_e(i))=I_{\B_i}(-.e).$ Moreover, there is some $k\in[1,a)$ such
that $i\in[i_k,i_{k+1}).$ Then, by Remark \ref{Rkfiltraza}, we have that $I_{\B_i}(-.e)=I_{\B_{i_k}}(-,e);$ proving our assertion in both cases.
\

Once we have that $\overline{\tau}_i(P^{op}_e(i))=I_{\B_{i_k}}(-,e)$ for some $k\in[1,a].$ In order to see that $\overline{\tau}_i(P^{op}_e(i))\in\F_f(\Delta),$ by Remark
\ref{Rkfiltraza}, it is enough to prove that $\frac{I_{\B_{k}}(-,e)}{I'_{\B_k}(-,e)}\in\Delta(k)^{\oplus}$ for any $k<\alpha.$
\

Let $k<\alpha.$ By  hypothesis we have that
$$\frac{I_{\B_{k}}(-,e)}{I'_{\B_{k}(-,e)}}
\in \proj_\rho\Big(\frac{\mathfrak{R}}{I'_{\B_k}}\Big).$$
Then by \cite[Lemma 3.5]{Martin}, there is some $e'\in\B_k$ such that
$\frac{I_{\B_{k}}(-,e)}{I'_{\B_k}(-,e)}\simeq\frac{\R(-,e')}{I'_{\B_k}(-,e')}.$ Moreover, since $e'\in\B_k$ and $\R$ is locally finite, it follows that
$e'=\oplus_{i=1}^{n_e}\,t_j^{m_j},$ where $t_1,\dots,t_{n_e}$ are locally and pairwise non isomorphic objects in $\B_k.$ In case, some $t_j\in\B_l$ and $l<k,$ we have that $\R(-,t_j)=\B_l(-,t_j).$ Thus, we may assume that $t_j\in\sigma_k(\B),$ for any $j\in[1,n_e].$ Then
$$\frac{I_{\B_{k}}(-,e)}{I'_{\B_k}(-,e)}\simeq
\bigoplus_{i=1}^{n_e}\Big(\frac{\R(-,t_j)}{I'_{\B_k}(-,t_j)}\Big)^{m_j}=\bigoplus_{i=1}^{n_e}\Delta_{t_j}(k)^{m_j};$$
proving that  $\overline{\tau}_i(P^{op}_e(i))\in \ltt{F}_{f}(\bigcup_{j<i}\Delta(j))$.
\end{dem}

\begin{cor}\label{C2filt-sum-direct}  Let $(\R,\tilde{\A})$ be a right standardly stratified $\K$-ringoid, with $\R$ locally finite, and let $\Delta={}_{\tilde{\A}}\Delta$
be the $\tilde{\A}$-standard family of right $\R$-modules. Then, all the standard modules
$\Delta_e(i)$ are local and the following statements hold true.
\begin{itemize}
\item[(a)] For any $M\in\F_f(\Delta),$ the filtration multiplicity $[M:\Delta_{e}(i)]$ does not depend on a given $\Delta$-filtration of $M.$
\item[(b)] $\F_{f}(\Delta)\subseteq\finp_\rho(\R)$ and it is  a locally finite $\K$-ringoid.
\end{itemize}
\end{cor}
\begin{proof} Let $\tilde{\A}=\{\tilde{\A}_i\}_{i<\alpha}$ be the given parition of $\ind\,(\R).$ By Proposition \ref{P=AS},  we have the  exhaustive family
$\B(\A):=\{\B_{i}(\A)\}_{i<\alpha}$ of $\R.$ Then ${}_{\tilde{\A}}\Delta={}_{\B(\A)}\Delta$ since $\sigma(\B(\A))=\tilde{\A}.$ For simplicity, we
write $\B=\B(\A)$ and $\B_i=\B_i(\A)$ for any $i<\alpha.$ Since $(\R,\tilde{\A})$ is a standardly stratified $\K$-ringoid, the conditions in Theorem \ref{FilTrProj} (b) hold.
\

We start by proving that $\Delta\subseteq\finp_\rho(\R).$ Let $i<\alpha$ and $e\in\sigma_i(\B).$ If $\overline{\tau}_i(P^{op}_e(i))=0$ then
$\Delta_e(i)$ is equal to $P^{op}_e(i),$ which is finitely presented. Assume that $\overline{\tau}_i(P^{op}_e(i))\neq0,$ and let
$\Supp(\tau_{\B,P^{op}_e(i)})=\{i_1<i_2<\cdots<i_a\}.$
\

We assert that $I_{\B_{i_k}}(-,e)$ is finitely generated for any $k\in[1,a].$\\
 Indeed, by Remark \ref{Rkfiltraza}, we have
$I'_{\B_{i_1}}(-,e)=\sum_{j<i_1}\,I_{\B_j}(-,e)=0,$
and thus, by hypothesis, $I_{\B_{i_1}}(-,e)\in\proj_\rho(\R/I'_{\B_{i_1}}).$ Then, there is some $e'\in\R$ such that
$I_{\B_{i_1}}(-,e)=\R(-,e')/I'_{\B_{i_1}}(-,e');$ proving that $I_{\B_{i_1}}(-,e)$ is a finitely generated right $\R$-module. As before, we have that
$I'_{\B_{i_2}}(-,e)=\sum_{j<i_2}\,I_{\B_j}(-,e)=I_{\B_{i_1}}(-,e)$ and $I_{\B_{i_2}}(-,e)/I'_{\B_{i_2}}(-,e)\in\proj_\rho(\R/I'_{\B_{i_2}}).$ Therefore, we
get that the quotient $I_{\B_{i_2}}(-,e)/I_{\B_{i_1}}(-,e)$ is a a finitely generated right $\R$-module. Then, the exact sequence
$0\to I_{\B_{i_1}}(-,e)\to I_{\B_{i_2}}(-,e)\to I_{\B_{i_2}}(-,e)/I_{\B_{i_1}}(-,e)\to 0$ implies that $ I_{\B_{i_2}}(-,e)$ is finitely
generated. It is clear, by induction, that the assertion above holds.
\

In the proof of Theorem \ref{FilTrProj}, we proved that $\overline{\tau}_i(P^{op}_e(i))=I_{\B_{i_k}}(-,e)$ for some $k\in[1,a].$ Thus $\overline{\tau}_i(P^{op}_e(i))$ is
finitely generated. Therefore from the exact sequence $0\to \overline{\tau}_i(P^{op}_e(i))\to P_e(i)\to \Delta_e(i)\to 0$ and \cite[Proposition 4.2 (c) i)]{Aus}, we
conclude that $\Delta_e(i)$ is finitely presented; and thus $\Delta\subseteq\finp_\rho(\R).$ Hence, the result follows from
Theorem \ref{filt-sum-direct}.
\end{proof}

\begin{defi} Let $\R$ be a locally finite $\K$-ringoid and let $\B:=\{\B_{i}\}_{i< \alpha}$ be an exhaustive family of subcategories of $\R.$ We say that $\B$ is {\bf right noetherian} if for any $i<\alpha$ and $e\in\sigma_i(\B)$ the following statements hold: $\Supp(\tau_{\B,P_e^{op}(i)})$ is finite and
there is some $i_0<\alpha$ such that $I_{\B_j}(-,e)=P_e^{op}(i)$ for any $j\geq i_0.$
\end{defi}

\begin{cor}\label{stiss} Let $\R$ be a locally finite $\K$-ringoid and let $\B:=\{\B_{i}\}_{i<\alpha}$ be an exhaustive family of subcategories of
$\R.$ Then, the following statements are equivalent.
\begin{itemize}
\item[(a)] $\B$ is right noetherian and $\R$ is right ideally standardly stratified with respect to $\B.$
\item[(b)] For the partition  $\sigma(\B)$ of $\ind\,(\R),$ related with the family  $\B,$ we have that $(\R,\sigma(\B))$ is a right standardly stratified $\K$-ringoid.
\end{itemize}
\end{cor}
\begin{dem} (a) $\Rightarrow$ (b)  It follows directly from Theorem \ref{FilTrProj}.
\

(b) $\Rightarrow$ (a) By hypothesis, we have that Theorem \ref{FilTrProj} (b) holds for any $i<\alpha$ and $e\in\sigma_i(\B).$ We need to show that
$$\forall\,t<\alpha\;\forall\,a\in\R\quad I_{\B_t}(-,a)/I'_{\B_t}(-,a)\in\proj_\rho(\R/I'_{\B_t}).$$
Let $t<\alpha$ and $a\in\R.$ We may assume that $a\in\ind\,(\R).$ Since $\sigma(\B)$ is a partition of $\ind\,(\R),$ by Proposition \ref{P=AS}, there is some
$i<\alpha$ such that $a\in\sigma_i(\B).$ Then, by Theorem \ref{FilTrProj} (b) we get that $I_{\B_t}(-,a)/I'_{\B_t}(-,a)\in\proj_\rho(\R/I'_{\B_t}).$
\end{dem}

\begin{teo}\label{stqs} Let $\R$ be a locally finite $\K$-ringoid and let $\B:=\{\B_{i}\}_{i< \alpha}$ be an  exhaustive family of subcategories of
$\R.$ Then, the following statements are equivalent.
\begin{itemize}
\item[(a)] $\B$ is right noetherian and $(\R,\B)$ is a right ideally quasi-hereditary $\K$-ringoid.
\item[(b)]  For the partition  $\sigma(\B)$ of $\ind\,(\R),$  we have that $(\R,\sigma(\B))$ is a right quasi-hereditary $\K$-ringoid and  
$\Hom(\Delta_e(i),\Delta_{e'}(i))=0$ for $e\neq e'$ in $\sigma_i(\B).$
\end{itemize}
 \end{teo}
 \begin{dem} (a) $\Rightarrow$ (b)  Since $\R$ is right ideally quasi-hereditary, it follows from Lemma \ref{Delta1} that
 $I'_{\B_i}(e,e')=\rad_\R(e,e')$ for any $e,e'\in\sigma_i(\B)$ and $i<\alpha.$ Then, by  Corollary \ref{stiss}, Lema \ref{claveD1} and Lemma \ref{claveD2}, we get (b).
 \

  (b) $\Rightarrow$ (a) Since $(\R,\sigma(\B))$ is a right quasi-hereditary $\K$-ringoid and
  $$\Hom(\Delta_e(i),\Delta_{e'}(i))=0\quad\text{for}\quad e, e'\in\sigma_i(\B)$$
  for any $i<\alpha,$ it follows from   Lema \ref{claveD1} and Lemma \ref{claveD2} that $I'_{\B_i}(e,e')=\rad_\R(e,e')$ for any $e,e'\in\sigma_i(\B)$ and
  $i<\alpha.$ We assert that
  $$(*)\quad \rad_\R(e,e')\subseteq I'_{\B_i}(e,e')\quad\forall\,e,e'\in\ind\,(\B_i),\forall\,i<\alpha.$$
  Indeed, let $i<\alpha$ and $e,e'\in\ind\,(\B_i).$ If $e,e'\in\sigma_i(\B),$ then $\rad_\R(e,e')= I'_{\B_i}(e,e').$ Assume that one of them, say $e,$ belongs to $\B_j$ for some $j<i.$ Thus $I_{\B_j}(e,e')=\R(e,e')$ and therefore $I'_{\B_i}(e,e')=\sum_{k<i}\,I_{\B_k}(e,e')=\R(e,e');$ proving that
  $\rad_\R(e,e')\subseteq I'_{\B_i}(e,e').$
  \

  Let $e,e'\in\ind\,(\B_i)$ and $x,y\in\R.$ Then, by $(*)$ we get
  $$I_{\B_i}(e',x)\, \rad_\R(e,e')\,I_{\B_i}(y,e) \subseteq I_{\B_i}(e',x) I'_{\B_i}(e,e')I_{\B_i}(y,e)\subseteq I'_{\B_i}(y,x).$$
  Therefore, we conclude that $I_{\B_i}\,\rad_\R\,I_{\B_i}\subseteq I'_{\B_i}$ for any $i<\alpha.$ Then, as in the proof of \cite[Theorem 3.6 (i)]{Martin})
  and using that $I'_{\B_i}=I_{\bigcup_{j<i}\B_j},$ we obtain that $I_{\B_i}/I'_{\B_i}\,\rad_\R\,I_{\B_i}/I'_{\B_i}=0$ for any $i<\alpha.$  Then, by Corollary \ref{stiss} we get (a).
 \end{dem}

 \section{Rings with enough idempotents}

 In this section, we define the terms ``standardly stratified'' and ``quasi-hereditary'' for a very special subclass of $\K$-algebras (possibly  without $1$) 
 having  good enough properties that allow us to define the standard modules. As we will see,  there are plenty of them and appear in different contexts, for example, as a generalization of Ringel's notion of species \cite{Rin2} or in connection with the Galois covering in the sense of Bongartz and Gabriel \cite{BG} and De la Pe\~na-Martinez \cite{DM}. 
 \

 Let $\Lambda$ be a $\K$-algebra such that $\Lambda^2=\Lambda$ (here  $\Lambda^2$ denotes the $\K$-submodule of $\Lambda$ generated by all
 the products $xy,$ for $x,y\in\Lambda$). We denote by $\Mod(\Lambda)$ the category of all unitary left $\Lambda$-modules $M,$ where unitary means that
 $\Lambda M=M.$ The finitely generated unitary left $\Lambda$-modules is a full subcategory of $\Mod(\Lambda)$ and it is usually denoted by
 $\modu(\Lambda).$ The class of finitely generated projective objects in $\Mod(\Lambda)$ is denoted by $\proj(\Lambda).$ We denote by $f.\ell(\K)$ the class of all the $\K$-modules of {\bf finite length}.
\

A $\K$-algebra {\bf with enough idempotents} ({\bf w.e.i $\K$-algebra}, for short) is a pair $(\Lambda,\{e_i\}_{i\in I}),$ where $\Lambda$ is a $\K$-algebra and
$\{e_i\}_{i\in I}$ is a family of orthogonal idempotents of $\Lambda$ such that  $\Lambda=\oplus_{i\in I}\,e_i\Lambda=\oplus_{i\in I}\Lambda e_i.$ Note that,
for such algebra, we have that $\Lambda^2=\Lambda$ and $\Lambda=\oplus_{(i,j)\in I^2}\,e_i\Lambda\,e_j.$ Moreover, it is said that $(\Lambda,\{e_i\}_{i\in I})$ is {\bf Hom-finite} if
$\{e_j\Lambda e_i\}_{i,j\in I}\subseteq f.\ell(\K).$
\

Let $(\Lambda,\{e_i\}_{i\in I})$ be an w.e.i $\K$-algebra. The $\K$-ringoid $\R(\Lambda)$ associated with $(\Lambda,\{e_i\}_{i\in I})$ is defined as follows:
the objects of $\R(\Lambda)$ is the set $\{e_i\}_{i\in I},$ and the set of morphisms from $e_i$ to $e_j$ is $\Hom_{\R(\Lambda)}(e_i,e_j):=e_j\Lambda e_i.$ The composition
of morphism in $\R(\Lambda)$ is given by the multiplication of $\Lambda.$ We recall that $Y:\R(\Lambda)\to \Mod_\rho(\R(\Lambda))$ is the Yoneda's
contravariant functor, where $Y(e):=\Hom_{\R(\Lambda)}(-,e).$
\

 The following result is more or less known in the mathematical folklore, but for completeness and the benefit or the reader, we state it and give a proof.

\begin{pro}\label{WeiEquiv1} Let $(\Lambda,\{e_i\}_{i\in I})$ be a w.e.i $\K$-algebra. Then, the functor
$$\delta:\Mod_\rho(\R(\Lambda))\to \Mod(\Lambda^{op}), \quad M\mapsto \oplus_{i\in I}\,M(e_i)$$
is an isomorphism of categories, and  $\delta(Y(e_i))=e_i\Lambda$ for any $i\in I.$
\end{pro}
\begin{dem} Let $f:M\to N$ in  $\Mod_\rho(\R(\Lambda)).$ For each $i\in I,$ we have $f_{e_i}:M(e_i)\to N(e_i)$ and thus $\delta(f):=\oplus_{i\in I}\,f_{e_i}.$
\

The structure of $\Lambda$-module on $\delta(M):$\\
Let $\lambda\in\Lambda=\oplus_{i,j}\,e_j\Lambda e_i$ and $m\in\delta(M)=\oplus_{i\in I}\,M(e_i).$ Then, we have that $\lambda=\sum_{i,j}\,\lambda_{i,j}$
and $m=\sum_{i}\,m_i,$ where $\lambda_{i,j}\in e_j\Lambda e_i$ and $m_i\in M(e_i).$  Since $\lambda_{i,j}:e_i\to e_j$ is
a morphism in $\R(\Lambda),$ we obtain $M(\lambda_{i,j}):M(e_j)\to M(e_i).$ We set $(m\cdot\lambda)_t:=\sum_i\,M(\lambda_{t,i})(m_i).$ It is a
routine calculation
to show that $\delta(M)$ is a right $\Lambda$-module. Observe that $m\cdot e_j=m_j$ and thus $\delta(M)\cdot e_j=M(e_j),$ for any $j\in I.$ Let us consider
$e:=\sum_{i\in\Supp(m)}\,e_i,$ where $\Supp(m):=\{i\in I\;:\;m_i\neq 0\}.$ Since $m\cdot e_j=m_j,$ it follows that $m\cdot e=m$ and thus
$\delta(M)\cdot\Lambda=\delta(M).$
\

Consider the correspondence
$$\varepsilon:\Mod(\Lambda^{op})\to \Mod_\rho(\R(\Lambda)),\quad X\mapsto (e_i\mapsto Xe_i).$$
Let $g:X\to Y$ in $\Mod(\Lambda^{op})$ and $\lambda_{i,j}:e_i\to e_j$ in $\R(\Lambda).$ Le $X(\lambda_{i,j}):Xe_j\to Xe_i$ and
$\varepsilon_{e_j}(g):Xe_j\to Ye_j$ be defined as $X(\lambda_{i,j})(xe_j):=xe_j\lambda_{i,j}$ and $\varepsilon_{e_j}(g)(xe_j):=g(xe_j).$ It can be seen that
$\varepsilon(g):\varepsilon(X)\to\varepsilon(X)$ is a morphism in $\Mod_\rho(\R(\Lambda)),$ and moreover, it is a functor.
\

Let $M\in\Mod_\rho(\R(\Lambda).$ We know that $\delta(M)\cdot e_j=M(e_j).$ Therefore
$$(\varepsilon\delta(M))(e_j)=\delta(M)\cdot e_j=M(e_j).$$
Let $X\in\Mod(\Lambda^{op}).$ Since $X\Lambda=X$ and $\Lambda=\oplus_{i\in I}\Lambda e_i,$ we get
$$\varepsilon\delta(X)=\bigoplus_{i\in I}\varepsilon\delta(X)e_i=\bigoplus_{i\in I}\delta(X)\cdot e_i=\oplus_{i\in I}Xe_i=X.$$
Thus $\delta$ is an isomorphism of categories with inverse $\varepsilon.$ Finally, we have
$$\delta(Y(e_i))=\bigoplus_{j\in I}Y(e_i)(e_j)=\bigoplus_{j\in I}e_i\Lambda e_j=e_i\Lambda.$$
\end{dem}

\begin{rk} \label{aurxRingoids} Let $(\Lambda,\{e_i\}_{i\in I})$ be a Hom-finite w.e.i $\K$-algebra.
Let $\overline{\R}(\Lambda):=\proj_\rho(\R(\Lambda)).$ Then,  $\overline{\R}(\Lambda)$ is a locally finite $\K$-ringoid. Moreover, It is well known \cite{Mitchell}, that the restriction functor
$$\Psi:\Mod_\rho(\overline{\R}(\Lambda))\to \Mod_\rho(\R(\Lambda)),\quad F\mapsto F|_{\R(\Lambda)}$$
 is an equivalence of categories and $\Psi((-,Y(e_i)))=Y(e_i)$ for any $i\in I.$ Therefore $\Psi(\proj_\rho(\overline{\R}(\Lambda)))=\proj_\rho(\R(\Lambda)).$ Thus, by using that $\overline{\R}(\Lambda)$ is a locally finite $\K$-ringoid, we can translate in terms of $\R(\Lambda)$ (and also in terms of $\Lambda$) all the
 result that we have proven for  locally finite $\K$-ringoids.
\end{rk}

\begin{lem}\label{auxLHF} Let $(\Lambda,\{e_i\}_{i\in I})$ be a Hom-finite w.e.i. $\K$-algebra and let  $f$ and $g$ be idempotents in 
$\Lambda.$ Then, the following statements hold true:
\begin{itemize}
\item[(a)] $g\Lambda f \subseteq f.\ell(\K);$
\item[(b)] $f\Lambda\simeq g\Lambda$ $\Leftrightarrow$ $\Lambda f\simeq \Lambda g.$
\end{itemize}
\end{lem}  
\begin{dem} (a) We have the finite sums $f=\sum_{k,l}\,e_{k,l}$  and $g=\sum_{i,j}\,e_{i,j},$ where 
$\,e_{k,l}\in e_k\Lambda e_l$ and $\,e_{i,j}\in e_i\Lambda e_j;$  and thus
$g\Lambda f=\sum_{i.j,k,l}e_{i,j}\Lambda e_{k,l}.$
Moreover, each $e_{i,j}\Lambda e_{k,l}\subseteq e_i\Lambda e_l$ and so it has finite length as $\K$-module. Therefore $g\Lambda f$ has finite length as $\K$-module. 
\

(b) It follows by applying the functor $\Hom(-\Lambda)$ to the given isomorphism and by using that 
$\Hom(\Lambda e,\Lambda)\simeq e\Lambda$ and $\Hom(e\Lambda,\Lambda)\simeq \Lambda e,$ for any $e^2=e\in\Lambda.$
\end{dem}

\begin{pro}\label{WeiEquiv2} For a  w.e.i $\K$-algebra  $(\Lambda,\{e_i\}_{i\in I}),$  the following statements are equivalent.
\begin{itemize}
\item[(a)] $(\Lambda,\{e_i\}_{i\in I})$ is Hom-finite.
\item[(b)] $\proj(\Lambda^{op})$ is a locally finite $\K$-ringoid.
\item[(c)] $\proj_\rho(\R(\Lambda))$ is a locally finite $\K$-ringoid.
\item[(d)] $\proj(\Lambda)$ is a locally finite $\K$-ringoid.
\item[(e)] $\proj(\R(\Lambda))$ is a locally finite $\K$-ringoid.
\end{itemize}
\end{pro}
\begin{dem} Let $i,j\in I.$ Then,  we have the isomorphisms of $\K$-modules
$$e_j\Lambda e_i=\Hom_{\R(\Lambda)}(e_i,e_j)\simeq\Hom_\Lambda(e_i\Lambda,e_j\Lambda).$$
Therefore, the fact that (b) (respectively, (d)) implies (a) follows easily, and the equivalence between (b) (respectively, (d))  and (c) (respectively, (e)) can be obtained from Proposition \ref{WeiEquiv1}. Let us prove that
(a) implies (b).
\

Assume that $(\Lambda,\{e_i\}_{i\in I})$ is Hom-finite. Then, it is clear that $\proj(\Lambda^{op})$ is a Hom-finite $\K$-ringoid. In order to prove that
$\proj(\Lambda^{op})$ is a Krull-Schmidt category, it is enough by \cite[49.10]{W} to see that $e\Lambda e$ is a semiperfect ring for any $e^2=e\in\Lambda.$
Indeed, let $e^2=e\in\Lambda.$ Then, by Lemma \ref{auxLHF} (a), we get that $e\Lambda e$ has finite length as $\K$-module and thus it is an
Artin ring. In particular $e\Lambda e$ is semiperfect. The fact that (a) implies (d) can be shown in a similar way.
\end{dem}

\begin{cor} \label{WeiEquiv3} For a Hom-finite w.e.i $\K$-algebra  $(\Lambda,\{e_i\}_{i\in I}),$  the following statements hold true.
\item[(a)] $\End(e\Lambda)$ and $\End(\Lambda e)$ are Artin rings, for any $e^2=e\in\Lambda.$
\item[(b)] For each $i\in I,$ there exists a unique (up to permutations)  family $\overline{e}_i:=\{e_{k,i}\}_{k=1}^{n_i}$ of primitive orthogonal idempotents in 
$\Lambda$ such that $e_i=\sum_{k=1}^{n_i}e_{k,i}.$ 
\end{cor}
\begin{dem} (a) Let $e^2=e\in\Lambda.$ Then, by Lemma \ref{auxLHF} $e\Lambda e\in f.\ell(\K).$  Finally, since 
$\End(e\Lambda)\simeq e\Lambda e\simeq \End(\Lambda e)$ as $\K$-modules, we get (a). 
\

(b) Let $i\in I.$ By Proposition \ref{WeiEquiv2} (b), there is a decomposition 
\begin{center}
$(*)\quad e_i\Lambda=\oplus_{k=1}^{n_i} P_{k,i}$ with $P_{k,i}$ local, for all $k,i.$ 
\end{center}
Since  $e_i=e_i^2\in e_i\Lambda,$ we get from $(*)$ the unique decomposition $e_i=\sum_{k=1}^{n_i}e_{k,i}$ of $e_i.$ Therefore, the family $\{e_{k,i}\}_{k=1}^{n_i}$ consists of orthogonal idempotents in $\Lambda.$ Hence $P_{k,i}=e_{k,i}\Lambda$ for each $k,i.$ But now, since each $P_{k,i}$ is local, we get that $e_{k,i}\Lambda e_{k,i}\simeq \End(e_{k,i}\Lambda)$  has only
trivial idempotents. But the latest condition is equivalent that $e_{k,i}$ be primitive. 
\end{dem}

\begin{cor} \label{WeiEquiv3'} Let  $(\Lambda,\{e_i\}_{i\in I})$ be a Hom-finite w.e.i $\K$-algebra and $\ind\,\{e_i\}_{i\in I}$ be the quotient of the set $\cup_{i\in I}\overline{e}_i$ (see Corollary \ref{WeiEquiv3}) by the equivalence relation $\sim,$ where 
$f\sim g$ if, and only if, $f\Lambda\simeq g\Lambda.$ Denote by $[e]$ the equivalence class of $e\in \cup_{i\in I}\overline{e}_i.$ Then, the following statements hold true
\begin{itemize}
\item[(a)] $\ind\,\overline{\R}(\Lambda)=\{ \delta^{-1}(e\Lambda)\;:\; [e]\in \ind\,\{e_i\}_{i\in I}\};$
\item[(b)] $\ind\,\proj(\Lambda^{op})=\{e\Lambda\;:\;[e]\in \ind\,\{e_i\}_{i\in I}\};$ 
\item[(c)] $\ind\,\proj(\Lambda)=\{\Lambda e\;:\; [e]\in \ind\,\{e_i\}_{i\in I}\}.$
\end{itemize}
\end{cor}
\begin{dem} By Proposition \ref{WeiEquiv1} $\proj(\Lambda^{op})=\delta(\overline{\R}(\Lambda)).$ Then, by Remark \ref{aurxRingoids}, Proposition \ref{WeiEquiv2}, Lemma \ref{LARp1} and Corollary \ref{WeiEquiv3}, we get (a) and (b). In order to show (c), by Corollary \ref{WeiEquiv3}, we have that $\Lambda e_i=\oplus_{k=1}^{n_i}\,\Lambda e_{k,i}$ and $\Lambda e_{k,i}$ is local, for all $i,k.$ Consider the relation on $\cup_{i\in I}\overline{e}_i$ given by:  $f\approx g$ if and only if $\Lambda f\simeq \Lambda g.$ By Lemma \ref{auxLHF} (b), we have that $\approx$ coincide with $\sim.$ Thus,  we obtain (c) in a similar way as we did for (b).
\end{dem}

\begin{defi} Let  $(\Lambda,\{e_i\}_{i\in I})$ be a Hom-finite w.e.i $\K$-algebra. For $M\in\Mod(\Lambda^{op}),$ the support of $M$ is
$$\Supp\,(M):=\{e\in\ind\,\{e_i\}_{i\in I}\;:\;Me\neq 0\}.$$
We say that $\Lambda$ is {\bf right support finite} if $\Supp\,(e\Lambda)$ is finite for any $e\in\ind\,\{e_i\}_{i\in I}.$   Dually,
$\Lambda$ is {\bf left support finite} if $\Supp\,(\Lambda e)$ is finite for any
$e\in\ind\,\{e_i\}_{i\in I}.$ Finally, $\Lambda$ is support finite if it is
right and left support finite.
\end{defi}

\begin{rk} \label{WeiEquiv4} Let  $(\Lambda,\{e_i\}_{i\in I})$ be a Hom-finite w.e.i $\K$-algebra.
\begin{itemize}
\item[(1)] We say that $(\Lambda,\{e_i\}_{i\in I})$ is {\bf basic} if $e_i$ is primitive for each $i$ and $e_i\Lambda\not\simeq e_j\Lambda$ for $e_i\neq e_j.$ Note that $(\Lambda,\{e_i\}_{i\in I})$ is basic if, and only if, $\ind\,\{e_i\}_{i\in I}=\{e_i\}_{i\in I}.$
\item[(2)] By Proposition \ref{WeiEquiv1}, Remark \ref{aurxRingoids}  and Corollary \ref{WeiEquiv3'}, we can see that
$\Lambda$ is right (resp. left) support finite if, and only if, the ringoid $\overline{\R}(\Lambda)$ is  right (resp. left) support finite.
\end{itemize}
\end{rk}

In what follows, we show a natural way to construct basic Hom-finite w.e.i. $\K$-algebras, which are also support finite. By following K. Bongartz and P. Gabriel \cite{BG}, let $\K$ be a field and
$Q$ be a quiver (which may be infinite), $Q_0$ is the set of vertices and $Q_1$ is the set of arrows. A path $\gamma$ in $Q,$ of length $n\geq 1,$ is of
the form $\gamma=\alpha_n\alpha_{n-1}\cdots\alpha_1$ for arrows $\alpha_i\in Q_1,$  and can be visualised as
$a_0\xrightarrow{\alpha_1}a_1\to\cdots \to a_{n-1}\xrightarrow{\alpha_n}a_n.$  We say that $\gamma$ starts at the vertex $a_0$ and ends at the vertex
$a_n.$ The vertices in $Q$ can be seen as paths of length $0,$ and for each $a\in Q_0$ its corresponding path of length zero will be denoted by $\varepsilon_a.$
For each non negative integer $n,$ we denote by $Q_n$ the set of all paths of length $n.$ Let $\K Q_n$ be the $\K$-vector space whose base is the set
$Q_n.$
\

The path  $\K$-algebra is the $\K$-vector space $\K Q:=\bigoplus_{n\geq 0}\,\K Q_n$ whose product of two basis vectors are given by the concatenation
of paths. Note that $\varepsilon_Q:=\{\varepsilon_a\}_{a\in Q_0}$ is a family of orthogonal idempotents in $\K Q,$ and $\varepsilon_b Q_n\varepsilon_a$ is the
set of all paths of length $n,$ which start at $a$ and end at $b.$ Moreover the pair $(\K Q,\varepsilon_Q)$ is a $\K$-algebra with enough idempotents. We
 denote by $J_Q$ the ideal in $\K Q$ generated by the set $Q_1.$ An ideal $I$  of $\K Q$ is {\bf admissible} if $I\subseteq J_Q^2$ and for each
 $x\in Q_0$ there is a natural number $n_x$ such that $I$ contains each path of length $\geq n_x$ which starts or ends at $x.$ For any admissible ideal
 $I$ of $Q,$ we consider the quotient path $\K$-algebra $\K(Q,I):=\K Q/I$ and the set of ortogonal idempotents $e_{Q,I}:=\{e_a\}_{a\in Q_0},$ where
 $e_a:=\varepsilon_a+I.$ We recall that a quiver $Q$ is {\bf locally finite} if for each vertex $x\in Q_0$ there is a finite number of arrows in $Q_1,$ which start
 or end at $x.$ The main properties, from our point of view, of quotient path $\K$-algebras can be summarized in the following proposition.

 \begin{pro} \label{WeiEquiv5} Let $Q$ be a locally finite quiver (which may be infinite), $\K$ be a field and $I$ be an admissible ideal of $\K Q.$ Then, the following statements hold true.
 \begin{itemize}
 \item[(a)] The pair $(\K(Q,I),e_{Q,I})$ is a basic Hom-finite w.e.i $\K$-algebra.
 \item[(b)] $\K(Q,I)$ is support finite.
 \item[(c)] $\proj\,(\K(Q,I))$ and  $\proj\,(\K(Q,I)^{op})$ are locally finite $\K$-ringoids.
 \item[(d)] $\ind\,\proj\,(\K(Q,I))=\{\Lambda e_a\}_{a\in Q_0}.$
 \item[(e)] $\ind\,\proj\,(\K(Q,I)^{op})=\{e_a\Lambda\}_{a\in Q_0}.$
 \end{itemize}
 \end{pro}
 \begin{dem} For a proof of (a) and (b), see \cite[2.1]{BG}. The items (c), (d) and (e), can be obtained from Corollary \ref{WeiEquiv3'}.
 \end{dem}

 By Corollary \ref{WeiEquiv3'}, we know that the rings with a nice setting, where  we can define the standard modules, are precisely the Hom-finite
 $\K$-algebras with enough idempotents.
 \

Let $(\Lambda,\{e_i\}_{i\in I})$ be a Hom-finite w.e.i $\K$-algebra.  Then
 $\ind\,\proj(\Lambda^{op})=\{e\Lambda\;:\;[e]\in \ind\,\{e_i\}_{i\in I}\}.$ Choose a partition $\tilde{\A}=\{\tilde{\A}_i\}_{i<\alpha}$
 of the set $\ind\,\{e_i\}_{i\in I}.$  Define ${}_{\Lambda^{op}}P_e(i):=e\Lambda,$ for any $[e]\in\tilde{\A}_i.$ Let
 ${}_{\Lambda^{op}}P:=\{{}_{\Lambda^{op}}P(i)\}_{i\leq\alpha},$ where ${}_{\Lambda^{op}}P(i):=\{{}_{\Lambda^{op}}P_e(i)\}_{e\in\tilde{\A}_i}.$ The
 family of $\tilde{\A}$-standard right $\Lambda$-modules ${}_{\Lambda^{op}}\Delta=\{\Delta(i)\}_{i<\alpha},$ where
 $\Delta(i):=\{\Delta_{e}(i)\}_{e\in \tilde{\A}_i},$ is defined as follows
 $$\Delta_{e}(i):=\frac{{}_{\Lambda^{op}}P_{e}(i)}{\mathrm{Tr}_{\oplus_{j<i}\overline{P}(j)}({}_{\Lambda^{op}}P_{e}(i))},$$
 where $\overline{P}(j):=\bigoplus_{r\in\tilde{\A}_j}\,{}_{\Lambda^{op}}P_r(j).$  Let $P:=\delta^{-1}({}_{\Lambda^{op}}P),$ where
 $\delta:\Mod_\rho(\R(\Lambda))\to \Mod(\Lambda^{op})$ is the isomorphism of Proposition \ref{WeiEquiv1}. Then, by Corollary \ref{WeiEquiv3'} (a), it can be shown that
 $\delta({}_{\tilde{\A}}\Delta_e(i))={}_{\Lambda^{op}}\Delta_e(i).$

 \begin{defi}\label{ssringwei}Let $(\Lambda,\{e_i\}_{i\in I})$ be a Hom-finite w.e.i $\K$-algebra. We say that the pair $(\Lambda,\tilde{\A})$ is a right {\bf standardly  stratified $\K$-algebra} if  $\tilde{\A}$ is a partition of $\ind\,\{e_i\}_{i\in I}$  such that
$\mathrm{Tr}_{\oplus_{j<i}\overline{P}(j)}({}_{\Lambda^{op}}P_{e}(i))\in \ltt{F}_f(\bigcup_{j<i}\Delta(j)),$
for any $i<\alpha$ and $e\in\tilde{\A}_i.$
\end{defi}

\begin{rk} \label{WeiEquiv6} Let $(\Lambda,\{e_i\}_{i\in I})$ be a Hom-finite w.e.i $\K$-algebra. Consider 
$\overline{\R}(\Lambda):=\proj_\rho(\R(\Lambda))$ as we did in Remark \ref{aurxRingoids}. Then  $\overline{\R}(\Lambda)$ is a locally finite $\K$-ringoid such that the restriction functor
$$\Psi:\Mod_\rho(\overline{\R}(\Lambda))\to \Mod_\rho(\R(\Lambda)),\quad F\mapsto F|_{\R(\Lambda)}$$
 is an equivalence of categories and $\Psi((-,Y(e_i)))=Y(e_i),$ for any $i\in I.$ 
 \
 
Let  $\tilde{\A}=\{ \tilde{\A}_j\}_{j<\alpha}$ be a partition of the set
 \begin{center}
 $\ind\,(\overline{\R}(\Lambda))=\{ E_e:=\delta^{-1}(e\Lambda)\;:\; [e]\in \ind\,\{e_i\}_{i\in I}\}$ (see Corollary \ref{WeiEquiv3'} (a)).
 \end{center} 
 Then, $\Psi(\tilde{\A})$ is
 a partition  of $\ind\,\{e_i\}_{i\in I}.$ Moreover, for $E=\delta^{-1}(e\Lambda)\in\tilde{\A}_j,$ we have $\Psi({}_{\tilde{\A}}\Delta_E(i))={}_{\Psi(\tilde{\A})}\Delta_e(i).$
 Therefore, by using that $\overline{\R}(\Lambda)$ is a locally finite $\K$-ringoid, we can translate in terms of  $\Lambda$ all the
 results that we have proven for  locally finite $\K$-ringoids.
\end{rk}

\begin{cor} \label{WeiEquiv7} Let $(\Lambda,\{e_i\}_{i\in I})$ be a Hom-finite w.e.i $\K$-algebra, and let $\tilde{\A}$ be a partition of 
$\ind\,\{e_i\}_{i\in I}$ such
that $(\Lambda,\tilde{\A})$ is a right  standardly  stratified $\K$-algebra. Then, all the standard modules
$\Delta_e(i)$ are local and the following statements hold true.
\begin{itemize}
\item[(a)] For any $M\in\F_f(\Delta),$ the filtration multiplicity $[M:\Delta_{e}(i)]$ does not depend on a given $\Delta$-filtration of $M.$
\item[(b)] $\F_{f}(\Delta)\subseteq\finp(\Lambda^{op})$ and it is  a locally finite $\K$-ringoid.
\end{itemize}
\end{cor}
\begin{dem} It follows from Remark \ref{WeiEquiv6} and Corollary \ref{C2filt-sum-direct}.
\end{dem}

\begin{cor} \label{WeiEquiv8} Let $Q$ be a locally finite quiver (which may be infinite), $\K$ be a field and $I$ be an admissible ideal of $\K Q.$ Then, for any
partition $\tilde{\A}$ of $e_{Q,I},$ each of the standard module
$\Delta_e(i)$ is local and the following statements hold true.
\begin{itemize}
\item[(a)] For any $M\in\F_f(\Delta),$ the filtration multiplicity $[M:\Delta_{e}(i)]$ does not depend on a given $\Delta$-filtration of $M.$
\item[(b)] $\F_{f}(\Delta)\subseteq\finp(\K(Q,I)^{op})$ and it is  a locally finite $\K$-ringoid.
\end{itemize}
\end{cor}
\begin{dem} By Proposition \ref{WeiEquiv5}, we have that $(\K(Q,I),e_{Q,I})$ is a basic Hom-finite w.e.i $\K$-algebra, which is also support finite. Then, the result follows from Remark \ref{WeiEquiv6} and Corollary \ref{C1filt-sum-direct}.
\end{dem}

\begin{ex} Let $Q$ be the following locally finite quiver
\[
\begin{tikzcd}[->,>=stealth,shorten >=1pt,auto,node distance=4cm,
                thick,main node/.style={circle,draw,font=\Large\bfseries}]
0 \arrow[r, "\beta"] \arrow[out=45,in=90,loop,swap,"\alpha"]&
  1 \arrow[r, "\beta"]\arrow[out=45,in=90,loop,swap,"\alpha"]&
    2 \arrow[r, "\beta"]\arrow[out=45,in=90,loop,swap,"\alpha"]&
\cdots
\end{tikzcd}
\]

 Consider the quotient path $\K$-algebra $\Lambda:=\K(Q,I),$ where $\K$ is a field and $I$ is the admissible ideal $<\alpha^2, \beta^2,\alpha\beta,\beta\alpha>.$  For
 each $i\in Q_0,$ we have the idempotent $e_i:=\varepsilon_i+I$ of $\Lambda.$  In what follows, we choose different partitions of $\{e_i\}_{i\in Q_0}$ and
 we will see if $\Lambda$ is standardly stratified (or not) with respect to these partitions.
 \

 \begin{itemize}
 \item[(1)] Consider $\tilde{\A}=\{\tilde{A}_i\}_{i\in<\aleph_0},$ where $\tilde{A}_i:=\{e_i\}.$ In this case, we have that
 $\Delta(i)=\{\Delta_{e_i}(i)=\Lambda e_i\}$ and thus $(\Lambda,\tilde{\A})$ is standardly stratified. However, it is not quasi-hereditary since
 $\End(\Delta_{e_0}(0))\simeq e_0\Lambda e_0$ is not a division ring.
 \item[(2)] Consider $\tilde{\B}=\{\tilde{B}_0, \tilde{B}_1, \tilde{B}_2\},$ where $\tilde{B}_0:=\{e_1\},$ $\tilde{B}_1:=\{e_0\}$ and
 $\tilde{B}_2:=\{e_i\}_{i\geq 2}.$ In this case, we get $\Delta_{e_1}(0)=\Lambda e_1,$ $\Delta_{e_0}(1)=\Lambda e_0/S(1)$ and
 $\Delta_{e_i}(2)=\Lambda e_i,$ for any $i\geq 2,$ where $S(1)=\Lambda e_1/\rad\,(\Lambda e_1).$ Note that $\tilde{\B}$ is finite, however
 $(\Lambda,\tilde{\B})$ is not standardly stratified, since $\Lambda e_0\not\in\F_f(\Delta).$
 \item[(3)] Consider $\tilde{\C}=\{\tilde{\C}_i\}_{i\in<\aleph_0},$ where $\tilde{\C}_0:=\{e_1\},$ $\tilde{\C}_1:=\{e_0\}$ and
 $\tilde{\C}_i:=\{e_i\},$ for $i\geq 2.$ In this case, we get $\Delta_{e_1}(0)=\Lambda e_1,$ $\Delta_{e_0}(1)=\Lambda e_0/S(1)$ and
 $\Delta_{e_i}(i)=\Lambda e_i,$ for any $i\geq 2.$ Note that $\tilde{\C}$ is infinite, however
 $(\Lambda,\tilde{\C})$ is not standardly stratified, since $\Lambda e_0\not\in\F_f(\Delta).$
 \end{itemize}
\end{ex}

\footnotesize

\vskip3mm \noindent Octavio Mendoza Hern\'andez:\\ Instituto de Matem\'aticas, Universidad Nacional Aut\'onoma de M\'exico\\
Circuito Exterior, Ciudad Universitaria,
C.P. 04510, M\'exico, D.F. MEXICO.\\ {\tt omendoza@matem.unam.mx}

\vskip3mm \noindent Martin Ort\'iz Morales:\\ Facultad de Ciencias, Universidad  Aut\'onoma del Estado de M\'exico\\
{\tt mortizmo@uaemex.mx}

\vskip3mm \noindent Edith Corina S\'aenz Valadez:\\ Departamento de Matem\'aticas, Facultad de Ciencias, Universidad Nacional Aut\'onoma de M\'exico\\
Circuito Exterior, Ciudad Universitaria,
C.P. 04510, M\'exico, D.F. MEXICO.\\ {\tt corina.saenz@gmail.com}

\vskip3mm \noindent Valente Santiago Vargas:\\ Departamento de Matem\'aticas, Facultad de Ciencias, Universidad Nacional Aut\'onoma de M\'exico\\
Circuito Exterior, Ciudad Universitaria,
C.P. 04510, M\'exico, D.F. MEXICO.\\ {\tt valente.santiago.v@gmail.com}


\begin{thebibliography}{20}
\bibitem{ADL} I. \'Agoston, V. Dlab, E. Luk\'acs. Stratified algebras.
\emph{Math. Rep. Acad. Sci. Canada} 20 (1) 22-28 (1998).

\bibitem{AHLU} I.  \'Agoston, D. Happel, E. Luk\'acs,  L. Unger.
Standardly Stratified Algebras and Tilting. {\it{Journal of  Algebra}} 226(1), 144-160 (2000).

\bibitem{AHLU2} I. \'Agoston, D. Happel, E. Luk\'acs,  L. Unger. Finitistic dimension
of standardly stratified algebras. {\it{Comm. in algebra}} 28(6), 2745-2752 (2000).

\bibitem{ADL2} I. \'Agoston, V. Dlab, E. Luk\'acs. Standardly stratified extensions algebras.
{\it{Comm. in algebra}} 33, 1357-1368 (2005). 

\bibitem{ADL3}  I. \'Agoston, V. Dlab, E. Luk\'acs. Approximations of algebras by standardly
stratified algebras. J. Algebra 319, 4177-4198 (2008).


\bibitem{Aus2} M. Auslander. A functorial approach to representation theory, 
{\it {Representations of algebras, Lecture Notes in Mathematics}}, vol. 944, Springer-Verlag, 105-179, (1982).

\bibitem{Aus} M. Auslander.  Representation Theory of Artin Algebras I.  {\it{Comm. in algebra}} 3(1), 177-268, (1974).

\bibitem{BG} K. Bongartz, P. Gabriel. Covering spaces in representation theory. {\it{Invent. Math.}} 65 (3), 331-378, (1982).

\bibitem{CPS1} E. Cline, B. Parshall, L. Scott.
Derived categories, quasi-hereditary algebras and algebraic groups.
{\emph{in: Proceedings of the Ottawa-Moosonee Workshop in Algebra, 1987,
Mathematics Lecture Notes series, Carleton University and Universite d'Ottawa}}, (1988).

\bibitem{CPS2}  E. Cline, B. Parshall, L. Scott. Stratifying
endomorphism algebras. {\it{Mem. AMS }}. 591, (1996).

\bibitem{CPS0}  E. Cline, B. Parshall, L. Scott. Algebraic stratification in reprsentation categories. 
{\it {Journal of algebra}} 117, Issue 2, 504-521, (1988).

\bibitem{DM} J. A. De la Pe\~na, R. Martinez. The universal cover of a quiver with relations. {\it{Journal in Pure and Applied Algebra}}, 30, 277-292, (1983).

\bibitem{Dlab0} V. Dlab. Quasi-hereditary algebras. Appendix to Yu. A. Drozd,  V. Kirichenko. 
Finite dimensional algebras. {\it {Springer-Verlag}} (1993).

\bibitem{Dlab} V. Dlab. Quasi-hereditary algebras revisited.
{\it{An. St. Univ. Ovidius Constanta}}, 4,
43-54, (1996).

\bibitem{Dla} V. Dlab.  Properly Stratified algebras. C. R. Acad. Sci. Paris S\'er. I Math. 331, no. 3, 191-196, (2000)

\bibitem{DR2} V. Dlab, C. M. Ringel. Quasi-hereditary algebras. 
{\it{Illinois Journal of Mathemathics}}, 33, no. 2, 280-291, (1989).

\bibitem{DR1} V. Dlab, C. M. Ringel. The module Theoretical Approach to
Quasi-hereditary algebras. {\it{Repr. Theory and Related Topics,
London Math. Soc. LNS}} 168, 200-224, (1992).

\bibitem{ES} K. Erdmann, C. S\'aenz. On standardly stratified algebras.
{\it{Comm. in algebra}}, 32,  3429-3446,  (2003).

\bibitem{Fre} P. Freyd. Representations in abelian category. Proceedings of the 
conference on categorical algebra. {\it{La Jolla}} (1966), 95-120.

\bibitem{Frisk1} A. Frisk. Two-step tilting for standardly stratified algebras. 
Algebras Discrete Math. 3, 38-59, (2004).

\bibitem{Frisk2} A. Frisk. Dlab's theorem and tilting modules for stratified algebras.
J. Algebra, 314(2), 507-537, (2007).

\bibitem{FriskMaz} A. Frisk, V. Mazorchuk. Properly stratified algebras and tilting.
Proc. London Math. Soc. A (3), 92(1), 29-61, (2006).

\bibitem{FKM} V. Futorny, S. Konig, V. Mazorchuk. Categories of induced modules and
standarly stratified algebras. Algebr. Represent. Theory  5(3), 259-276, (2002).


\bibitem{Gabriel} P. Gabriel. Des cat\'egories ab\'eliennes. Bulletin de la S. M. F., tome 90, (1962), 323-448.

\bibitem{He} A. Heller. Homological algebra in abelian categories. {\it{The Annals of Math,}} Second Series, Vol. 68, no. 3, (1958), 484-525.

\bibitem{Krause}  H. Krause. Krull-Schmidt categories and projective covers.
{\it{Expo. Math.}} 33, num. 4,  535-549, (2015).


\bibitem{Krause2} H. Krause. Highest weight categories and recollements. {\it{arXiv:1506.01485.}}

\bibitem{MMS1} E. N. Marcos, O. Mendoza, C. S\'aenz.
Stratifying systems via relative simple modules.
{\it{J.  Algebra}} 280, 472-487, (2004).

\bibitem{MMS2} E: N. Marcos, O. Mendoza, C. S\'aenz. Stratifying systems via relative projective modules.
{\it{Comm. in algebra}} 33, 1559-1573, (2005).

\bibitem{MVO2} R. Mart\'{\i}nez-Villa, M. Ort\'{\i}z-Morales.Tilting theory and functor categories II. Generalized Tilting. {\it{Applied categorical structures.}} vol. 21, 311-348, (2013).

\bibitem{MVO1}  R. Mart\'{\i}nez-Villa, M. Ort\'{\i}z-Morales. Tilting theory and functor categories I. Classical Tilting.{\it{Applied categorical structures.}} vol. 22, 595-646, (2014).

\bibitem{MVS1} R. Mart\'{\i}nez-Villa, \O. Solberg.  Graded and Koszul categories. {\it{Applied categorical structures.}} vol. 18, 615-652, (2010).

\bibitem{MVS2} R. Mart\'{\i}nez-Villa, \O. Solberg.   Artin-Schelter regular algebras and categories. {\it{Journal of pure and applied algebra.}} vol. 215, 546-565, (2011).

\bibitem{MVS3} R. Mart\'{\i}nez-Villa. Solberg \O. Noetherianity and Gelfand-Kirilov dimension of components. {\it{Journal of algebra.}} 323 , no. 5, 1369-1407, (2010).


\bibitem{Mazor1} V. Mazorchuck. Stratified algebras arising in Lie Theory. Representation
of finite dimensional algebras and related topics in Lie theory and geometry. Fields
Inst. Commun. 40, Amer. Math. Soc., Providence, RI, 245-260, (2004)

\bibitem{Mazor2} V. Mazorchuk. On the finitistic dimension of stratified algebras.
Algebra Discrete Math. 2004(3), 77-88, (2004).

\bibitem{Mazor3} V. Mazorchuk. Koszul duality for stratified algebras II. 
Standardly stratified algebras. J. Aust. Math. Soc. 89 (1), 23-49, (2010).


\bibitem{MazOv} V. Mazorchuk, S. Ovsienko. Finitistic dimension of properly
stratified algebras. Adv Math, 186 (1), 251-265, (2004).

\bibitem{MazPark} V. Mazorchuk,  A. Parker. On the relation between finitistc and
good filtration dimensions. {\it{Comm. in algebra}}, 32 (5), 1903-1917, (2004).

\bibitem{MSX} O.  Mendoza, C. S\'aenz C,  C. Xi.
Homological systems in module categories over pre-ordered sets. {\it{Quart. J. Math.}}
60, 75-103, (2009).

\bibitem{MenSan} O. Mendoza, V. Santiago.
Homological systems in triangulated categories.
{\it{ Appl. Categor.  Struct.}} (2014). doi: 10.1007/s10485-014-9384-5.

\bibitem{Mitchell} B. Mitchell. Rings with several objects. {\it{ Advances in Math.}}, 8, 1-161, (1972).

\bibitem{Martin} M. Ortiz. The Auslander-Reiten components seen as quasi-hereditary categories.
{\it{M. Appl Categor Struct}} (2017). https://doi.org/10.1007/s10485-017-9493-z

\bibitem{PR} M. I. Platzeck,  I. Reiten. Modules of finite projective dimension for
standardly stratified algebras. {\it{Comm. in algebra}} 29, 973-986, (2001).

\bibitem{Rin2} C.M.  Ringel. Representation of $K$-species and bimodules. {\it{J. Algebra}} 41, 269-302, (1976).

\bibitem{Rin} C.M. Ringel. The category of modules with good
filtrations over a quasi-hereditary algebra has almost split
sequences. {\it{Math. Z.}} 208, 209-223, (1991).

\bibitem{Scott}  L. L. Scott. Simulating algebraic geometry with algebra I: The algebraic theory of derived categories, in {\it The Arcata Conference on Representations of Finite Groups (Arcata, Calif., 1986)} {\it {Proc. Sympos. Pure Math. AMS}}, 47, pp. 2, 71-281, (1987).

\bibitem{W} R. Wisbauer. Foundations of Module and Ring Theory. A Handbook for Study and Research. University of Dusseldorf, 1991. Gordon and Breach Science Publishers, Reading.

\bibitem{Webb} P. Webb.  Standard Stratifications of EI categories and Alperin's weight conjecture. {\it{Journal of algebra}}, 320(12), 4073-4091,  2008.
    
\bibitem{Xi} C. Xi. Standardly stratified algebras and cellular algebras. Math. Proc. Cambr. Phil. Soc. 
133, (2002),  37-53.
\end{thebibliography}
\end{document}